\documentclass[11pt,reqno]{amsart}
\usepackage{tikz}
\usepackage{subfig}
\usepackage{epstopdf}
\usepackage{epsfig}
\usepackage{math}
\usepackage{mathtools}
\usepackage{algorithm}
\usepackage[noend]{algpseudocode}

\usepackage[numbers,sort&compress]{natbib}

\newcommand{\half}{\frac{1}{2}}
\newcommand{\energy}{\mathcal{E}}
\newcommand{\entropy}{\mathcal{H}}

\newcommand{\dW}{\mathcal{W}_2}

\newcommand{\prox}{\text{prox}}

\newcommand{\Amat}{\mathsf{A}}
\newcommand{\Bmat}{\mathsf{B}}

\newcommand{\Imat}{\mathsf{I}}

\newcommand{\indi}{\chi}
\newcommand{\Smat}{\mathsf{S}}
\newcommand{\Hmat}{\mathsf{H}}
\newcommand{\rd}{\mathrm{d}}

\newcommand{\grad}{\nabla}

\newcommand{\tmax}{t_\textrm{max}}
\newcommand{\Rd}{{\mathbb{R}^d}}
\newcommand{\veci}{{\bf i}}
\newcommand{\vece}{{\bf e_i}}
\newcommand{\Ul}{u^{(l)}}
\newcommand{\ulp}{u^{(l+1)}}
\newcommand{\Hl}{{\Hmat^{(l)}}}

\newcommand{\normHl}[1]{ \| #1 \|_{\Hl}}
\newcommand{\normtwo}[1]{ \| #1 \|_{2}}
\newcommand{\vecrho}{\boldsymbol{\rho}}
\newcommand{\vecm}{\bold{m}}

\usepackage{tikz}
\usetikzlibrary{shapes,arrows}
\tikzstyle{decision} = [diamond, draw, fill=blue!20, 
    text width=4.5em, text badly centered, node distance=3cm, inner sep=0pt]
\tikzstyle{block} = [rectangle, draw, fill=blue!20, 
    text width=5em, text centered, rounded corners, minimum height=4em]
\tikzstyle{line} = [draw, -latex']
\tikzstyle{cloud} = [draw, ellipse,fill=red!20, node distance=3cm,
    minimum height=2em]
    
\usetikzlibrary{positioning}
\tikzset{main node/.style={circle,fill=blue!20,draw,minimum size=1cm,inner sep=0pt},  }

\usepackage{soul}
\usepackage{color}

\usepackage{hyperref}
\usepackage{graphicx} \usepackage{placeins}

\begin{document}
\title[Fisher information Regularization]{Fisher information regularization schemes for Wasserstein gradient flows}
\author{Wuchen Li}
\address{Mathematics department, University of California, Los Angeles 90095}
\email{wcli@math.ucla.edu}
\author{Jianfeng Lu}
\address{Departments of Mathematics, Physics, and Chemistry, Duke University, Box 90320,  Durham, NC 27708.}
\email{jianfeng@math.duke.edu}
\author{Li Wang}
\address{School of Mathematics, University of Minnesota, Twin cities, MN 55455.}
\email{wang8818@umn.edu}

\keywords{Time discretization; Gradient flow; Fisher information; Optimal transport; Schr{\"o}dinger bridge problem.}

\maketitle

\begin{abstract}
We propose a variational scheme for computing Wasserstein gradient flows. The scheme builds upon the Jordan--Kinderlehrer--Otto framework with the Benamou-Brenier's dynamic formulation of the quadratic Wasserstein metric and adds a regularization by the Fisher information. This regularization can be derived in terms of energy splitting and is closely related to the Schr{\"o}dinger bridge problem. It improves the convexity of the variational problem and automatically preserves the non-negativity of the solution. As a result, it allows us to apply sequential quadratic programming to solve the sub-optimization problem. We further save the computational cost by showing that no additional time interpolation is needed in the underlying dynamic formulation of the Wasserstein-2 metric, and therefore, the dimension of the problem is vastly reduced. Several numerical examples, including porous media equation, nonlinear Fokker-Planck equation, aggregation diffusion equation, and Derrida-Lebowitz-Speer-Spohn equation, are provided. These examples demonstrate the simplicity and stableness of the proposed scheme. 
\end{abstract}

\section{Introduction}
Consider the general continuity equation of the form:
\begin{equation}\label{gd}
\partial_t \rho = - \nabla \cdot (\rho v) := \nabla \cdot [\rho \nabla (U'(\rho) + V + W \ast \rho )], \quad \rho(0, \cdot) = \rho_0\,,
\end{equation}
where $\rho(t,x)$, $x\in\Omega\subset \mathbb{R}^n$ is the particle
density function, $U(\rho)$ is an internal energy, $V(x)$ is a drift
potential, and $W(x,y)=W(y,x)$ is an interaction potential. $\nabla$,
$\nabla\cdot$ are gradient and divergence operators with respect to $x$
in $\Omega$. This equation can be derived as a mean-field limit of
particle systems with a number of physical and biological
applications, such as granular materials \cite{CMV03}, chemotaxis,
animal swarming \cite{CFTV10, BCCD16}, and many others.  In
particular, the Fokker-Planck equation \cite{JKO98}, porous medium
equation \cite{Otto}, aggregation equation \cite{TBL06, FHK11},
Keller-Segel equation \cite{KellerSegel1971}, and quantum
drift-diffusion equation \cite{GSG09quantum} all fall within this
framework.

As written, equation \eqref{gd} possesses two immediate properties: it
preserves the non-negativity of the solution and conserves total
mass. Therefore, in what follows, we will always consider nonnegative
initial data with mass one, so that the solution is in the set of
probability measures on $\Omega$, $\mathcal P (\Omega)$. The third
property of \eqref{gd} is the dissipation of the energy, which can be
seen as follows. Given an energy
$\energy: \mathcal {P}(\Omega) \rightarrow \mathbb{R} \cup \{+\infty
\}$,
we may formally define its gradient with respect to the quadratic
Wasserstein metric $\dW$ as
\[ \grad_{\dW} \energy(\rho) = - \grad \cdot \left( \rho \grad  \delta \energy \right)  \,, \]
where $\delta$ always denotes the first variation in $\rho$ throughout the paper. 
Comparing it with \eqref{gd}, one can write the velocity field as $v = - \nabla \delta \energy$, and view equation (\ref{gd}) as the gradient flow of the energy
\begin{equation} \label{eqn:energy}
\energy (\rho) = \int_{\Omega} \left[U(\rho(x)) + V(x) \rho(x) \right]\rd x + \half \int_{\Omega \times \Omega} W(x-y) \rho(x) \rho(y) \rd x \rd y\,.
\end{equation}
Differentiating the energy (\ref{eqn:energy}) along solutions of \eqref{gd}, one formally obtains the decreasing of energy along the gradient flow $\frac{d}{dt} \energy(\rho)(t) = - \int_{\RR^d} |v(t,x)|^2 \rho(t,x) \rd x$, which indicates that the solution evolves in the direction of \emph{steepest descent} of an energy. This property entails a full characterization of the set of stationary states, and provides a necessary tool to study its stability. 

Desired numerical methods for \eqref{gd} are to attain all three properties above at the discrete level, which, however, are rather challenging. Existing methods have been developed on different prospects of the equation. One kind of methods views it as an advection diffusion equation and employs finite difference, finite volume, or discontinuous Galerkin \cite{Filbet06, CK08, CCH15, LWZ17, SCS18, BCH}. Such methods are {\it explicit} or {\it semi-implicit} in time, so the per time computation is cheap. But they often suffer from stability constraints, due either to the degeneracy of the diffusion or the non-locality from the interaction potential, such as the mesa problem \cite{LTWZ18b}. Another approach leverages structural similarities between \eqref{gd} and equations from fluid dynamics to develop particle methods \cite{CB16, CCCC18,BLL12,OM17, CCP18}. On one hand, particle methods naturally conserve mass and positivity, and they can also be designed to respect the underlying gradient flow structure of the equation so as to dissipate the energy along time. On the other hand, a large number of particles is often required to resolve finer properties of solutions.

A third class of methods builds on a variational formulation, following the seminal work by Jordan, Kinderlehrer, and Otto \cite{JKO98}. Given a time step $\tau>0$, the scheme (known as the JKO scheme) recursively defines a sequence $\rho^k$ as
\begin{equation} \label{JKO}
\rho^0 = \rho_0, \quad \rho^{k+1} = \arg\min_{\rho\in \KK} \left\{  \dW^2 (\rho, \rho^k) +  2\tau \energy(\rho) \right\} \,,
\end{equation}
where
$\KK = \left\{ \rho: \rho \in \mathcal{P}(\Omega), ~
  \int_{\Omega}|x|^2 \rho \,\rd x < + \infty \right\}\,$, and $\dW$ denotes the quadratic Wasserstein distance between two
probability measures. Therefore, \eqref{JKO} offers a positivity
preserving, energy dissipating, and unconditionally stable time
discretization. A major bottleneck of this approach is the computation
of $\dW^2 (\rho, \rho^k)$, which is an infinite dimensional
minimization problem. Hence, early works that use \eqref{JKO} avoid
direct computing $\dW^2 (\rho, \rho^k)$ either by linearization
\cite{GT06, BCW09} or by diffeomorphisms \cite{BCC08,CM09, CRW16},
which lead to methods that lose some inherited properties in
\eqref{JKO} or are limited by complicated geometry and structure. Only
recent progress in computing $\dW$ has enabled the direct application
of \eqref{JKO} \cite{peyre2015entropic, BCL16, CPSV18, CCWW18}.

In the present work, we will adopt Benamou-Brenier's dynamic formulation for the Wasserstein distance \cite{BB00}. In particular, given two measures $\rho_0$ and $\rho_1$, their Wasserstein distance can be obtained by solving
\begin{equation} \label{BB00} 
\begin{dcases}
&\dW(\rho_0,\rho_1)   = \inf_{\rho,v}  \left\{  \int_0^1 \int_\Omega |v(t,x)|^2 \rho(t,x) \rd x  \rd t \right\}^{1/2} ,
\\ & \textrm{s.t.}  \qquad  \partial_t \rho + \nabla \cdot (\rho v) = 0  
\\ & \hspace{1.2cm} (\rho v) \cdot \nu = 0   \text{ on } \partial \Omega \times [0,1], \quad 
 \rho(0, x ) = \rho_0(x), \ \rho(1, x) = \rho_{1}(x) \,,
\end{dcases}
\end{equation}
where $\nu$ is the outer unit normal on the boundary of the domain $\Omega$. Adapting \eqref{BB00} into \eqref{JKO}, and let $m = \rho v$, we have the following computable reformulation of the JKO scheme: given $\rho^k(x)$, $\rho^{k+1}(x) = \rho(1,x)$ with $\rho(t,x)$ solving 
\begin{equation} \label{classicalJKO}
\begin{dcases}
& (\rho,m)=\arg\inf_{ \rho,m}~ \int_0^1\int_\Omega F(\rho,m)dxdt  + 2\tau\energy(\rho(1,\cdot))
\\ & \textrm{s.t.} \quad  \partial_t \rho + \nabla \cdot m = 0, ~ \rho(0,x) = \rho^k(x)\,, ~ m\cdot \nu = 0\,,
\end{dcases}
\end{equation}
where 
\begin{align} \label{PHI}
F(\rho, m) = \left\{  \begin{array}{cc}  \frac{\|m \|^2}{ \rho} & \textrm{ \hspace{-1.45cm} if } \rho>0\,, \\ 0 & \textrm{\hspace{0.2cm} if } (\rho, m)=(0,0) \,,
\\ +\infty & \textrm{ \hspace{-1cm} otherwise\,.}  \end{array} \right.
\end{align}
To solve \eqref{classicalJKO}--\eqref{PHI}, there are two sources of difficulties. One lies in the non-smooth function of $F$, so that second order information that often used to accelerate the optimization can not be applied. Occasionally, erroneous solution near $\rho =0$ may be produced. The other comes from the artificial time introduced in the dynamic formulation \eqref{BB00}, which increases the dimension of the problem. 

To overcome these two issues, we propose the following scheme: $\rho^{k+1} = \rho(1,x)$ where 
\begin{equation}\label{MP}
\begin{dcases}
&\rho^{k+1}(x)=\arg\inf_{\rho,m}~\int_\Omega \Big(\frac{\|m(x)\|^2}{\rho(x)}+{\beta^{-2}} \tau^2 \|\nabla\log\rho(x)\|^2\rho(x)\Big)\rd x + 2\tau \energy(\rho)\\
&\textrm{s.t.} \quad  \rho(x)-\rho^k(x)+\nabla\cdot m(x)=0,~ m \cdot \nu = 0\,.
\end{dcases}
\end{equation}
The additional term,
$\int_\Omega \|\nabla\log\rho(x)\|^2\rho(x) \rd x$ is the Fisher
information functional. It keeps $\rho$ away from zero (see
Theorem~\ref{thm-property}) and thus simplifies
$\int_\Omega F(\rho, m)\rd x$ to just
$\int_\Omega\frac{\|m\|^2}{\rho}\rd x$. More importantly, it improves
the convexity of the cost functional and gives access to the second
order sequential programming which enjoys much faster convergence. In
addition, we replace the time derivative in the dynamics by a {\it one
  step} finite difference. We shall show in Theorem~\ref{th3} that
such a simplification will not violate the first order accuracy of the
original JKO formulation \eqref{classicalJKO}. Furthermore, as we
shall see in Section \ref{section3}, compared to classical backward
Euler method that may suffer from ill-conditioned Jacobian \cite{BCH},
\eqref{MP} provides a symmetric, structure preserving version of
implicit method that is insensitive to the condition number of the
Hessian, and has guaranteed convergence. We note that the relation
between Fisher information and Schr{\"o}dinger bridge problem (SBP)
can be seen from
\cite{Carlen,Chen2016,Conforti,Leger,Leonard2013_surveyb}, and will be
further discussed in Section~\ref{section2.5}.

It is also important to mention that the Fisher information
regularization is closely related to the entropic regularization that
has been successfully applied in many optimal transport problems
\cite{cuturi13, peyre2015entropic,
  CarlierDuvalPeyreSchmitzer2015_convergence, gentil15, CPSV18}.  There, the
Kantorovich formulation based on the joint distribution $\pi(x,y)$
between two measures is adopted, and the entropic regularization term
$\int\int \pi(x,y)\log \pi(x,y)dxdy$ is added to the cost function, so
that an iterative projection method, Sinkhorn method or more general
Dykstra's algorithm, can be applied with linear convergence. This
method, when applied to gradient flow problem, has a major difficulty
in computing the proximal of the energy \eqref{eqn:energy} with
respect to the Kullback-Leibler divergence, which does not have a
closed form in general.

The paper is organized as follows. In the next section, we provide necessary background on the dynamical formulation of Schrodinger bridge problem (SBP), its relation with Wasserstein gradient flows and the Fisher information functional. We then derive the Fisher information regularized semi-discrete scheme in the end. In Section 3, we introduce a fully discrete scheme and study the properties of this new scheme. Numerical results are provided in Section \ref{section4}, and the paper is concluded in Section 5.

\section{Semi-discretization with Fisher information regularization}\label{section2.5}
In this section, we briefly review the Schr{\"o}dinger bridge problem, Fisher information regularization and Wasserstein gradient flow. We then weave together these ideas to derive our new regularized time discretization. 

\subsection{Schr{\"o}dinger Bridge problem and Fisher regularization}
Consider a bounded convex domain $\Omega\subset \mathbb{R}^n$, and the probability density space 
$$\mathcal{P}(\Omega)=\Big\{\rho\in L^1(\Omega)\colon \int_\Omega\rho(x)\rd x=1,~\rho(x)\geq 0\Big\}.$$
\begin{definition}[Schr{\"o}dinger bridge problem]
Denote $\textrm{SBP}\colon\mathcal{P}(\Omega)\times \mathcal{P}(\Omega)\rightarrow \mathbb{R}$. Given $\rho^0$, $\rho^1\in\mathcal{P}(\Omega)$, let
\begin{equation}
\label{eq:SB}
\textrm{SBP}(\rho^0, \rho^1)=\inf_{\rho, b} \int_0^1 \!\!\!\int_\Omega \| b(t,x)\|^2\,\rho(t,x)\,\rd x\,\rd t\,,
\end{equation}
where the infimum is taken among all drift functions $b\colon [0,1]\times \Omega\rightarrow \mathbb{R}^n$ and density functions $\rho\colon [0,1]\times \Omega\rightarrow \mathbb{R}$ satisfying the Fokker-Planck equation
\begin{equation}\label{eq: ct}
\partial_t\rho(t,x) + \nabla\cdot(\rho(t,x) b(t,x)) = \beta^{-1} \tau \Delta\rho(t,x),
\end{equation}
with the fixed initial and ending density functions
\begin{equation*}
\rho(0,x) = \rho^0(x), \quad \rho(1,x) = \rho^1(x)\,, \quad x \in \Omega \,,
\end{equation*}
and Neumann boundary condition for $b$: $b\cdot \nu = 0$ on $[0,1] \times \partial \Omega$.
\end{definition}
Here, $\beta$, $\tau$ are two given constant parameters in $\mathbb{R}$. Note specifically that when $\beta^{-1} \tau=0$,  $\textrm{SBP}(\rho^0, \rho^1)$ equals to the Wasserstein-2 distance between $\rho^0$ and $\rho^1$, where the problem \eqref{eq:SB} is equivalent to the Benamou--Brenier formula \cite{Benamou2000}. Here we are using the product $\beta^{-1} \tau$ as a regularization parameter just to facilitate the derivation for the gradient flow later. 

Interestingly, the variational problem \eqref{eq:SB} has the following symmetric reformulations. 

\begin{proposition}[Fisher information regularization]  \label{prop1}
Denote $\mathcal{H}(\rho)=\int_\Omega\rho(x)\log\rho(x)\rd x$, then
\begin{equation}\label{SBPN}
\textrm{SBP}(\rho^0,\rho^1)=\inf_{\rho, m} \int_0^1 \!\!\!\int_\Omega \left(\frac{\|m\|^2}{\rho}+ \beta^{-2} \tau^2 \rho \left\|\nabla \delta  \entropy \right\|^2\right)\,\rd x\,\rd t+2\beta^{-1} \tau (\mathcal{H}(\rho^1)-\mathcal{H}(\rho^0))
\end{equation}
subject to the dynamical constraint
\begin{equation} \label{ct2}
\partial_t\rho(t,x)+ \nabla\cdot m(t,x) =0,
\end{equation}
with initial and boundary conditions: 
\[
\rho(0,x) = \rho^0(x),~ \rho(1,x) = \rho^1(x) , ~x \in \Omega, \qquad m \cdot \nu = 0, ~ (t,x) \in [0,1] \times \partial \Omega\,.
\]
\end{proposition}

\begin{proof}
  First, rewrite the the Fokker-Planck equation \eqref{eq: ct} as
  follows:
\begin{equation*}
0 = \partial_t \rho +\nabla\cdot(\rho b)-\beta^{-1} \tau \Delta\rho = \partial_t \rho+\nabla\cdot(\rho (b-\beta^{-1} \tau \nabla \delta \entropy)) \,.
\end{equation*}
where we notice the fact that $\delta\mathcal{H}(\rho)=\log\rho+1$, and $\nabla\cdot(\rho\nabla\delta\mathcal{H}(\rho))=\nabla\cdot(\rho\nabla\log\rho)=\nabla\cdot(\nabla\rho)=\Delta\rho$. 

Denote $v = b - \beta^{-1} \tau  \nabla \delta \entropy$, and let $m = \rho v$, then \eqref{eq: ct} reduces to \eqref{ct2}. Next, we shall show that with the above definition of $m$, the cost functional in \eqref{SBPN} is the same as that in \eqref{eq:SB}. Indeed,
\begin{equation*}
\begin{split}
\int_0^1\int_{\Omega}\|b(t,x)\|^2\rho(t,x)\rd x\rd t=&\int_0^1 \int_{\Omega}\|v(t,x) + \beta^{-1} \tau \nabla \delta \mathcal{H}(\rho)(t,x)\|^2\rho(t,x)\rd x\rd t\\
=&\int_0^1\int_{\Omega}\{\|v\|^2\rho+\beta^{-2} \tau^2 \|\nabla \delta  \entropy\|^2\rho +2\beta^{-1} \tau \rho v\cdot\nabla \delta  \entropy\} \rd x\rd t\\
=&\int_0^1\int_\Omega\{ \frac{\|m\|^2}{\rho}+\beta^{-2} \tau^2  \|\nabla \delta  \entropy\|^2\rho+ 2\beta^{-1} \tau m\cdot \nabla \delta  \entropy\} \rd x\rd t.
\end{split}
\end{equation*}
We then show that last term in the above equation only depends on the initial and final condition. This is seen from the fact that 
\begin{equation}\label{term}
\begin{split}
&\int_0^1\int_{\Omega} m\cdot\nabla \delta  \entropy \rd x\rd t\\ 
=&-\int_0^1 \int_{\Omega}\delta  \entropy\nabla\cdot m \rd x \rd t \hspace{2.8cm} \textrm{Integration by parts w.r.t. $x$}  \\
=&\int_0^1\int_{\Omega} \delta  \entropy \partial_t \rho \rd x \rd t\\
=&\int_0^1 \frac{d}{dt}\mathcal{H}(\rho)\rd t = \mathcal{H}(\rho^1)-\mathcal{H}(\rho^0),
\end{split}
\end{equation}
where the second last equality comes from the definition of $L^2$
first variation. The other direction of the equivalence follows
similarly.
\end{proof}

Here, the symmetric version of Schr{\"o}dinger bridge problem relates to the optimal control problem of gradient flows. See related geometric studies in \cite{LegerLi2019_hopfcoleb, LiG}.
The additional term
\begin{equation*}
\mathcal{I}(\rho)=\int_\Omega \left\| \nabla \delta  \entropy \right\|^2\rho(x)\rd x=\int_\Omega \|\nabla\log\rho(x)\|^2\rho(x) \rd x
\end{equation*}
in the cost functional is named the Fisher information. 
In the sequel, we will apply the symmetric SBP to compute the Wasserstein gradient flow with $\mathcal I (\rho)$ serving as a regularization. The numerical benefits of this regularization will be discussed in Section \ref{section3}.

\subsection{Energy splitting and time discretization}\label{section2}
We are now ready to derive the main scheme \eqref{MP} of this paper. Starting from the classical JKO formulation \eqref{classicalJKO} of the Wasserstein gradient flow \eqref{gd}, we split the energy into two parts
\[
\energy(\rho) = \left( \energy(\rho) - \beta^{-1} \entropy(\rho)\right) + \beta^{-1} \entropy(\rho) := \energy_1 (\rho) + \energy_2(\rho)\,,
\]
and move $\energy_2$ to the flow constraint in \eqref{classicalJKO}. Here $\entropy$ is the entropy $\int_\Omega \rho \log \rho \rd x$ defined above. More specifically, given $\rho^{k}(x)$, we update $\rho^{k+1}(x) := \rho(1,x)$ by solving the following new form for $\rho(1,x)$:
\begin{equation} \label{JKOnew1}
\begin{dcases}
&(\rho, m) = \arg\inf_{\rho,m}~ \int_0^1 \int_\Omega \frac{\|m(t,x)\|^2}{\rho(t,x)}  \rd x \rd t  + 2\tau\energy_1(\rho(1,\cdot))
\\ & \textrm{s.t.} \quad  \partial_t \rho + \nabla \cdot m =  \tau \nabla \cdot (\rho \nabla \delta \energy_2(\rho)) = \tau \beta^{-1} \Delta \rho, 
\\ & \hspace{1cm}  \rho(0,x) = \rho^k(x)\,, \quad  (m - \tau \beta^{-1} \nabla \rho)\cdot \nu = 0\,.
\end{dcases}
\end{equation}
Intuitively, the difference between \eqref{classicalJKO} and \eqref{JKOnew1} lies in the flow of $\rho(t,x)$, $0<t<1$, between $\rho^k$ and $\rho^{k+1}$. In \eqref{classicalJKO}, the flow $\rho(t,x)$ is purely convective, and one controls it at final time $t=1$ using full energy $\energy$; whereas in \eqref{JKOnew1}, the diffusion effect in full energy (i.e., $\energy_2$) is moved to modify the flow so that the flow is both convective and diffusive, and therefore one only need to control its partial energy $\energy_1$ at the final time. Moreover, we observe that, the flux for $\rho(1,x)$ in this new form is the same as that in the original form \eqref{classicalJKO}. Since \eqref{classicalJKO} provides a first order approximation of $\rho(t,x)$ that resembles backward Euler scheme, the equivalence in the flux implies that \eqref{JKOnew1} is also a first order approximation in terms of $\tau$. Indeed, we rewrite \eqref{classicalJKO} using the Lagrangian multipliers 
\begin{align*}
\mathcal{L}_1 (\rho, m, \phi) &= \int_0^1 \int_\Omega \frac{\|m\|^2}{\rho} + \phi (\partial_t \rho + \nabla \cdot m) \rd x \rd t + 2\tau \energy(\rho(1, \cdot))
\\ & = \int_0^1 \int_\Omega \frac{\|m\|^2}{\rho} - \rho \partial_t \phi  - m \cdot \nabla \phi  \rd x \rd t + \int_\Omega \rho \phi \big|_{t=0}^{t=1}\rd x  + 2\tau \energy(\rho(1,\cdot))\,,  
\end{align*}
then the optimality condition $ \delta_{\rho, \phi, \rho(1,\cdot)}  \mathcal{L}_1= 0$ leads to 
\begin{align}
-\frac{\|m\|^2}{\rho^2} - \partial_t \phi = 0, \quad \frac{2m}{\rho} - \nabla \phi = 0 ,\quad  \phi(1,x) + 2\tau \delta  \energy(\rho(1,\cdot)) = 0\,.
\end{align}
Therefore $\phi$ satisfies the Hamilton-Jacobi equation $\partial_t \phi + \frac{1}{4} |\nabla \phi|^2  = 0$, and 
\begin{equation} \label{531}
m(1,x) = -\tau \rho(1, x) \nabla \delta  \energy(\rho(1,\cdot))\,.
\end{equation}
Plugging it into the constraint PDE in \eqref{classicalJKO}, one gets the flux for the original gradient flow equation \eqref{gd} after one time step $\tau$. Similarly, we rewrite \eqref{JKOnew1} as
\begin{align*}
\mathcal{L}_2 (\rho, m, \phi) &= \int_0^1 \int_\Omega \frac{\|m\|^2}{\rho} + \phi (\partial_t \rho + \nabla \cdot m - \tau \nabla \cdot (\rho \nabla \delta  \energy_2)) \rd x \rd t + 2\tau \energy_1(\rho(1, \cdot))
\\ & = \int_0^1 \int_\Omega \frac{\|m\|^2}{\rho} - \rho \partial_t \phi  - m \cdot \nabla \phi + \tau \rho \nabla \phi \cdot \nabla \delta \energy_2  \rd x \rd t + \int_\Omega \rho \phi \big|_{t=0}^{t=1}\rd x  + 2\tau \energy_1(\rho(1,\cdot)), 
\end{align*}
then the optimality condition $\delta_{m, \rho(1,\cdot)} \mathcal{L}_2 = 0$ leads to 
\[
\frac{2m}{\rho} - \nabla \phi = 0, \quad \phi(1,x) + 2\tau \delta  \energy_2(\rho(1,\cdot)) = 0\,.
\]
Consequently, $m(1,x) =  -\tau \rho(1,x) \nabla \delta  \energy_2(\rho(1,\cdot))$, which substituting back into the constraint of \eqref{JKOnew1} leads to the same flux for $\rho(1,x)$ as in \eqref{531}.  

Next, we rewrite \eqref{JKOnew1} in line with Proposition~\ref{prop1}. Let $\tilde{m} = m - \tau\beta^{-1} \nabla \rho $ (so that $\partial_t \rho + \nabla \cdot \tilde{m} = 0$), and plug it into the objective function in \eqref{JKOnew1}, we have 
\begin{align}
& \int_0^1\!\int_\Omega  \left( \frac{\|\tilde{m}\|^2}{\rho} + \tau^2 \beta^{-2} \frac{\|\nabla \rho\|^2}{\rho} + 2 \tau \beta^{-1} \frac{\tilde{m} \cdot  \nabla \rho}{ \rho}  \right) \rd t \rd x + 2 \tau \energy_2(\rho(1,\cdot)) \nonumber
\\  = & \int_0^1\!\int_\Omega \left( \frac{\|\tilde{m}\|^2}{\rho} + \tau^2 \beta^{-2} \rho \left\| \nabla \delta  \entropy\right\|^2 + 2 \tau \beta^{-1} \tilde{m} \cdot \nabla \delta  \entropy  \right) \rd t \rd x + 2 \tau \energy_2(\rho(1,\cdot)) \nonumber
\\  = & \int_0^1\!\int_\Omega \left( \frac{\|\tilde{m}\|^2}{\rho} + \tau^2 \beta^{-2} \rho \left\|  \nabla\delta  \entropy\right\|^2  \right) \rd t \rd x +  2\tau \beta^{-1} \left[ \entropy(\rho(1,\cdot)) - \entropy(\rho^k(x)) \right] + 2 \tau \energy_2(\rho(1,\cdot))  \nonumber
\\ = & \int_0^1\!\int_\Omega \left( \frac{\|\tilde{m}\|^2}{\rho} + \tau^2 \beta^{-2} \rho \left\|  \nabla\delta  \entropy\right\|^2  \right) \rd t \rd x + 2\tau \energy(\rho(1,\cdot))   -  2\tau \beta^{-1}  \entropy(\rho^k(x)) \,,
\label{530}
\end{align}
where the reformulation of the third term  in the second equation follows \eqref{term}. Omitting the tilde in \eqref{530}, \eqref{JKOnew1} can be reformulated as
\begin{equation} \label{JKOnew2}
\begin{dcases}
&\rho^{k+1}(x) = \arg\inf_{m, \rho}~ \int_0^1 \int_\Omega \frac{\|m(t,x)\|^2}{\rho(t,x)}  + \beta^{-2} \tau^2  \rho (t,x)\left\| \nabla \delta  \entropy (t,x)\right\|^2  \rd t \rd x  + 2\tau\energy(\rho(1,\cdot))
\\ & \quad \textrm{s.t.} \quad  \partial_t \rho + \nabla \cdot m = 0, ~ \rho(0,x) = \rho^k(x)\,, ~m \cdot \nu = 0. 
\end{dcases}
\end{equation}

In practice, we want to remove the additional dimension $t$ induced by
the flow, and therefore approximate the derivate in $t$ in the
constraint PDE of \eqref{JKOnew2} by a {\it one step} difference and
the integral in time in the objective function by a {\it one term}
quadrature. This leads to our main scheme \eqref{MP}. In the following
theorem, we show that, such an approximation does not violate the
first order accuracy of the original JKO scheme.

\begin{theorem}[Fisher information regularization scheme]\label{th3}
The minimizer of the variational problem \eqref{MP}
is a first-order time consistent scheme for Wasserstein gradient flow \eqref{gd}. 
\end{theorem}
\begin{proof}
First it is straightforward to check that variational problem \eqref{MP} is strictly convex. We then solve it by the Lagrange multiplier method. Define the Lagrangian as:   
\begin{equation*}
\mathcal{L}(\rho, \phi, m)= \mathcal{E}(\rho)+\int_\Omega \Big\{ \frac{1}{2\tau} \left[ \frac{\|m\|^2}{\rho}+\beta^{-2}\tau^2\|\nabla\delta\mathcal{H}\|^2\rho \right]+\phi(\rho-\rho^k)+\phi\nabla\cdot m\Big\}\rd x.
\end{equation*}
The critical solution of above variation problem forms 
\begin{equation*}
\delta_{m,\rho,\phi}\mathcal{L}=0\Rightarrow\left\{
\begin{aligned}
&\frac{1}{\tau}\frac{m}{\rho}= \nabla\phi \,,\\
&\delta\mathcal{E}(\rho)-\frac{1}{2\tau}\frac{\|m\|^2}{\rho^2}+ \half\beta^{-2}\tau\delta \int_\Omega \|\nabla\delta\mathcal{H}\|^2\rho \rd x+\phi=0 \,,\\
&\rho-\rho^k+\nabla\cdot m=0\,.
\end{aligned}\right.
\end{equation*}
The solution of above system $(m, \rho)$ satisfies
\begin{equation*}
\left\{
\begin{split}
&m= \tau\rho\nabla\phi\,, \\
&\phi=-\delta\mathcal{E}(\rho)+\frac{\tau}{2}\|\nabla\phi\|^2- \half\beta^{-2}\tau\delta \int_\Omega \|\nabla\delta\mathcal{H}\|^2\rho \rd x \,,\\
&\rho=\rho^k-\tau\nabla\cdot (\rho \nabla\phi).
\end{split}\right.
\end{equation*}
Denote the solution of above system as follows: $\rho=\rho^{k+1}$. 
We then derive the following update  
\begin{equation*}
\rho^{k+1}= \rho^k + \nabla \cdot m =\rho^k+ \tau \nabla\cdot(\rho^{k+1}\nabla\delta\mathcal{E}(\rho^{k+1}))+O(\tau^2).
\end{equation*}
Therefore scheme \eqref{MP} is a first order time discretization. 
\end{proof}
It is worth mentioning that there are several cases that the Fisher information regularization schemes are {\it exact} for the computation of gradient flows. 
\begin{proposition}[Exact cases]
If $\beta = \sqrt{\tau/2}$, then the iterative scheme \eqref{MP} is a first order scheme for the equation
\begin{equation*}
\partial_t\rho(t,x)=\nabla\cdot\Big(\rho\nabla\delta(\mathcal{I}(\rho)+\mathcal{E}(\rho))\Big).
\end{equation*}
\end{proposition}
\begin{proof}
If $\beta=\sqrt{\tau/2}$, then $\beta^{-2}\tau^2 = 2 \tau$, the scheme \eqref{MP} becomes 
\begin{equation}\label{scheme001}
\begin{split}
\rho^{k+1}(x)=&\arg\inf_{\rho, m}~\int_\Omega \Big(\frac{\|m\|^2(x)}{\rho(x)}+2\tau\|\nabla\delta\mathcal{H}(\rho)(x)\|^2\rho(x)\Big)\rd x+2\tau \mathcal{E}(\rho)\\
&\textrm{s.t.}\quad\quad \rho(x)-\rho^k(x)+\nabla\cdot m(x)=0.
\end{split}
\end{equation}
Following Theorem \ref{th3}, the algorithm is a consistent first time discretization of Wasserstein gradient flow for functional $\mathcal{I}(\rho)+\mathcal{E}(\rho)$.
\end{proof}

Several remarks are in order. 
\begin{remark}[Comparison with the classical JKO scheme]
Compared with the classical approach of JKO \eqref{classicalJKO}, our method does not require any inner time interpolation in the underlying dynamical formulation. It still preserves the first order time accuracy of the time discretization. 
\end{remark}

\begin{remark}[Schr{\"o}dinger Bridge problem proximal]
The variational problem \eqref{JKOnew2} can be viewed as a Schr{\"o}dinger bridge proximal method of Wasserstein gradient flows.
\end{remark}

\begin{remark}[Comparison with entropic regularization of gradient flow \cite{Peyre2015_entropic}]
We also compare our method with the entropic gradient flow studied in \cite{CarlierDuvalPeyreSchmitzer2015_convergence,Peyre2015_entropic}. A known fact is that when $\mathcal{H}(\rho)=\int \rho(x)\log\rho(x)\rd x$, the SBP problem has the static formulation \cite{dataSBP}
\begin{equation*}
\textrm{SBP}(\rho^0,\rho^1)=\inf_{\pi}\int\int \Big(\alpha \pi(x,y)\log\pi(x,y)+\|x-y\|^2\pi(x,y)\Big)\rd x\rd y,
\end{equation*}
 where $\alpha\geq 0$ is a constant and the infimum is over all joint histogram $\pi(x,y)\geq 0$ with marginals $\rho^0(x)$, $\rho^1(y)$. In \cite{CarlierDuvalPeyreSchmitzer2015_convergence,Peyre2015_entropic}, the algorithm applies the above static formulation and considers the iterative regularization algorithm for the computation of gradient flow. Our formulation mainly uses the dynamical formulation of SBP, especially its time symmetric version in Proposition \ref{prop1}.

\end{remark}

\begin{remark}[Generalized regularization functional]
Besides using $\mathcal{H}(\rho)=\int \rho\log\rho \rd x$, we can also study other types of regularizations, e.g., $\mathcal{H}(\rho)=\frac{1}{(1-\gamma)(2-\gamma)}\int (\rho^{2-\gamma}-1)\rd x$. We leave these studies for future works.
\end{remark}


\section{Full discretization and optimization algorithm}\label{section3}
In this section, we detail the spatial discretization and provide a complete algorithm for the fully discrete problem. The underlying principle for spatial discretization is to preserve the structure of Wasserstein metric tensor in the discrete sense so that it can be easily adapted to unstructured grid and more complicated equations with energy involving high order derivatives. Thanks to the Fisher information regularization, the resulting optimization is strictly convex and therefore gives access to second order Newton type optimization algorithms. 

\subsection{Spatial Discretization}
To better explain the idea, we first consider the discretization in one spatial dimension on uniform grid. Let $[0,L]$ be the computational domain and $\Delta x$ and $\tau$ be the spatial grid and temporal step respectively. Choose $ 0 = x_{\half} < x_{\frac{3}{2}} < \cdots < x_{N_x+\half} = L$, and define  
\begin{align*} 
& \rho_{j}^k = \rho( t_k, x_j),  \quad 1 \leq j \leq N_x, ~ k \in \mathbb{N}_+\,;
\\ & m_{j+\half}^k = m( t_k, x_{j+\half}),  \quad 0 \leq j \leq N_x, ~ k \in \mathbb{N}_+\,,
\end{align*}
where $x_j = j \Delta x,~ x_{j+\half} = (j+\half)\Delta x$, and $t_k = k \tau$. Note first that $m_{\half}^k =  m_{N_x+\half}^k = 0$ from the boundary condition, then the cost function in scheme \eqref{MP} can be discretized as 
\begin{align}
F(\boldsymbol{\rho}, \bold{m}) =&   \sum_{j=1}^{N_x-1} \left[ \frac{m_{j+\half }^2}{\rho_{j+\half }} + \beta^{-2}\tau^2 (\nabla \log \rho)_{j+\half }^2 \rho_{j+\half } \right] \Delta x  + 2 \tau \energy(\vecrho) \nonumber
\\ =& \sum_{j=1}^{N_x-1}  \left[ \frac{2 m_{j+\half }^2}{\rho_{j} + \rho_{j+1}} + \frac{\beta^{-2}\tau^2}{\Delta x^2} (\log \rho_{j+1} -\log \rho_{j})^2 \frac{ \rho_{j+1} + \rho_{j}}{2}  \right] \Delta x  + 2 \tau \energy(\vecrho)\,, \label{F1D}
\end{align}
where $\mathcal{E}(\boldsymbol{\rho}) $ in its general form reads
\[
\mathcal{E}(\boldsymbol{\rho}) = \sum_{j=1}^{N_x} [U(\rho_j) + V_j \rho_j ] \Delta x + \frac{1}{2} \sum_{j,l=1}^{N_x} W_{j,l} \rho_j \rho_l \Delta x^2\,.
\]
Here $\boldsymbol{\rho}$ and $\bold{m}$ are vector representations of vectors $\rho_{j}$ and $m_{j}$, i.e., $\vecrho = (\rho_1, \rho_2, \cdots \rho_{N_x})$, and $\bold{m} = (m_{1/2}, m_{3/2}, \cdots, m_{N_x+1/2})$.
The constraint is discretized with center difference in space as follows
\begin{eqnarray} 
\rho_{j} - \rho_{j}^k + \frac{1}{ \Delta x } (m_{j+\half} - m_{j-\half}) = 0, \quad 1\leq j \leq N_x, 
\label{constraint1D}
\end{eqnarray}
and the zero boundary conditions $m_{\half} = m_{N_x+\half} = 0$ is applied. 

Extension to two dimension is straightforward. Denote
\begin{align*} 
& \rho_{j,l}^k = \rho( t_k, x_j, y_l),  \quad 1 \leq j, l \leq N_x, ~~ k \in \mathbb{N}_+\,;
\\ & (m_x)_{j+\half,l}^k = m_x( t_k, x_{j+\half}, y_{l}),  \quad 0 \leq j, \leq N_x, ~~1\leq l \leq N_y,~~ k \in \mathbb{N}_+\,;
\\& (m_y)_{j,l+\half}^k = m_y( t_k, x_{j}, y_{l+\half}),  \quad 1 \leq j, \leq N_x,  ~~0\leq l \leq N_y,
    ~~ k \in \mathbb{N}_+\,,
\end{align*}
then the no-flux boundary condition on $m$ are imposed dimension by dimension, i.e.,  
\[
(m_x)_{\half,l} = (m_x)_{N_x + \half,l} = 0, ~~1 \leq l \leq N_y; \quad 
(m_y)_{j,\half} = (m_x)_{j, N_y + \half} = 0, ~~1 \leq j \leq N_x\,.
\]
The cost function then writes
\begin{align}
F(\boldsymbol{\rho}, \bold{m}_x, \bold{m}_y) = & \sum_{j=1}^{N_x-1}  \sum_{l = 1}^{N_y} \left[ \frac{2 (m_x)_{j+\half ,l }^2}{\rho_{j,l} + \rho_{j+1,l}} + \frac{\beta^{-2}\tau^2}{\Delta x^2} (\log \rho_{j+1,l} -\log \rho_{j,l})^2 \frac{ \rho_{j+1,l} + \rho_{j,l}}{2}  \right] \Delta x^2  \nonumber
\\  &  + \sum_{j=1}^{N_x}  \sum_{l = 1}^{N_y-1} \left[ \frac{2 (m_y)_{j ,l+\half }^2}{\rho_{j,l} + \rho_{j,l+1}} + \frac{\beta^{-2}\tau^2}{\Delta x^2} (\log \rho_{j,l+1} -\log \rho_{j,l})^2 \frac{ \rho_{j,l+1} + \rho_{j,l}}{2}  \right] \Delta x^2  \nonumber
\\& + 2 \tau \mathcal{E}(\boldsymbol{\rho}) \,, \label{F2D}
\end{align}
and the constraint becomes
\begin{align}
&\rho_{j,l} - \rho_{j,l}^k + \frac{1}{\Delta x} [(m_x)_{j+\half,l} - (m_x)_{j-\half,l}] + \frac{1}{\Delta y} [(m_x)_{j,l+\half} - (m_x)_{j,l-\half}] = 0, \nonumber
\\ &\hspace{9cm} 1\leq j \leq N_x , 1 \leq l \leq N_y\,. \label{constraint2D}
\end{align}
Generalization on graphs can be found in related studies \cite{LiSBP2,LiSBP}.

Upon spatial discretization, we therefore have the following finite dimensional variational problem: 
\begin{equation}\label{new_form3}
\begin{dcases}
& \min_{\vecrho \geq 0, \vecm} F(\vecrho, \vecm) := \sum_{\veci} \left[\frac{m^2_{\veci+\frac{1}{2} \vece}}{(\rho_{\veci}+\rho_{\veci+\vece})/2}+\frac{\beta^{-2}\tau^2}{{\bf \Delta x_\veci}^2}(\log \rho_{\veci}-\log \rho_{\veci+\vece})^2\frac{\rho_{\veci}+\rho_{\veci + \vece}}{2} \right] {\bf\Delta x} + 2\tau \energy(\vecrho) \\
& \text{s.t}\quad\quad   
\rho_{\veci}-\rho^k_{\veci}+\sum_{\bf e_i}\frac{1}{\bf \Delta x_i} (m_{\veci +\frac{1}{2}\vece}-m_{\veci-\frac{1}{2}\vece})=0\,.
\end{dcases}
\end{equation}
Here $\veci $ is a vector of sub index (e.g., $\veci = (j,l)$ in two dimension), ${\bf \Delta x}_\veci = \Delta x$ or $\Delta y$, ${\bf \Delta x} = \Pi_{\veci} {\bf \Delta x}_\veci$. Written in this way, the discretization can be directly generalized to unstructured grid. 

\begin{remark}
In practice, we will impose an non-negativity of $\rho: \rho_\veci \geq 0$ to avoid unexpected negative solution when the optimization is not fully converged, i.e., the iteration terminates when the stopping criteria is met. However, as we will show in Theorem~\ref{thm-property}, the non-negativity shall be preserved when the underlying optimization is solved exactly. 
\end{remark}

Denote its minimizer as $(\vecrho^*, \vecm^*)$, then
$\vecrho^{k+1} = \vecrho^*$.  We study the property of problem
\eqref{new_form3}. Note that the constraints contain both equalities
and inequalities, we will demonstrate that the Fisher information
regularization plays the crucial role of penalty function, which
enforces the density solution staying in the interior of probability
simplex. We next prove several properties of the proposed algorithm.
\begin{theorem} \label{thm-property} For each $k\in \mathbb{N}_+$, the
  following properties hold for scheme \eqref{new_form3}:
\begin{itemize}
\item[(i)] There exists a unique minimizer $\rho^{k+1}$ for the problem;
\item[(ii)]The modified energy decays 
\begin{equation*}
\frac{\beta^{-2}\tau}{2}\mathcal{I}(\rho^{k+1})+\mathcal{E}(\rho^{k+1})\leq \frac{\beta^{-2}\tau}{2}\mathcal{I}(\rho^k)+\mathcal{E}(\rho^k);
\end{equation*}
\item[(iii)] There exists a constant $c>0$, such that 
\begin{equation*}
\min_{j}\rho_j^{k+1}>c>0; 
\end{equation*}
\item[(iv)] The total mass is conserved 
\begin{equation*}
\sum_{j}\rho_j^{k+1}=\textrm{Constant}.
\end{equation*}
\end{itemize}
\end{theorem}
\begin{proof}
(i) The proof is based on the result of \cite{Li2018}. For the completeness of paper, we present it here. 
We shall show that 
\begin{itemize}
\item[(1)] The discrete Fisher information functional is shown to be positive infinity on the boundary of the probability set. Thus the minimizer of \eqref{new_form3} is obtained in the interior of simplex.
\item[(2)] The optimization problem \eqref{new_form3}  is strict convex in the interior of the constraint. 
\end{itemize} 
For notational convenience, we denote 
\begin{equation*}
\mathcal{K}(m,\rho)=\sum_{\veci } \frac{m^2_{\veci+\frac{\vece}{2}}}{(\rho_{\veci}+\rho_{\veci+\vece})/2}, \quad \mathcal{I}(m,\rho)=\sum_{\veci}\frac{1}{{\Delta x^2}}(\log \rho_{\veci}-\log \rho_{\veci+\vece})^2\frac{\rho_{\veci}+\rho_{\veci + \vece}}{2}.
\end{equation*}
We first show that the minimizer of \eqref{new_form3} in term of $\rho$ is strictly positive. This is true since $\mathcal{I}(\rho)$ is positive infinity on the boundary of simplex set, i.e.
\begin{equation*}
\lim_{\min_{i\in V}{ \rho_i}\rightarrow 0}\mathcal{I}(\rho)=+\infty,
\end{equation*}
where $V$ is the vertices set of the discretization. 
Suppose the above is not true, there exists a constant $M>0$, such that if there exists some $i^*\in V$, $\rho_{i^*}=0$, then
\begin{equation}\label{a}
\begin{split}
M\geq \mathcal{I}( \rho)=&\sum_{i+\frac{e_v}{2}\in E}\frac{1}{\Delta x^2}(\log \rho_i-\log \rho_{i+e_v})^2\frac{ \rho_i+ \rho_{i+e_v}}{2}\\
\geq& \sum_{i+\frac{e_v}{2}\in E}\frac{1}{\Delta x^2}(\log \rho_i-\log \rho_{i+e_v})^2 \frac{1}{2}\max\{ \rho_i,  \rho_{i+e_v}\}\,,
 \end{split}
 \end{equation} 
 where $E$ is the edge set of the discretization.
Notice that each term in \eqref{a} is non-negative, thus $$(\log \rho_i-\log \rho_{i+e_v})^2 \max\{ \rho_i,  \rho_{i+e_v}\}\leq 2M<+\infty,$$
for any edge$(i, i+e_v)\in E$. Since $ \rho_{i^*}=0$, the above formula further implies that for any $\tilde{\imath}\in N(i^*)$, $ \rho_{\tilde{i}}=0$. This is true since if $\rho_{i^*}\neq 0$, we have $$\lim_{ \rho_{i^*}\rightarrow 0}(\log \rho_{i^*}-\log \rho_{\tilde{\imath}})^2 \max\{ \rho_{i^*},  \rho_{\tilde{\imath}}\}=+\infty.$$ 
Similarly, we show that for any nodes $\tilde{\tilde{\imath}}\in N(\tilde{\imath})$, $\rho_{\tilde{\tilde{\imath}}}=0$. Here $N(\tilde{\imath})$ is the neighborhood of node $\tilde{\imath}$ in the discretization grids. We iterate the above steps a finite number of times.
Since the lattice graph is connected and the set $V$ is finite, we obtain $ \rho_i=0$, for any $i\in V$. This contradicts the assumption that $\sum_{i\in V} \rho_i=\textrm{Constant}$, which finishes the proof.

We now prove that $\mathcal{I}(\rho)$ is strictly convex in the variable $\rho$ with a constraint $\sum_{i\in V}\rho_i=\textrm{Constant}$, $\rho_i>0$, for any $i\in V$.  We shall show
\begin{equation}\label{claim3}
\min_{\sigma}~\{\sigma^T \mathcal{I}_{\rho\rho} \sigma~:~\sigma^T\sigma=1,~\sum_{i\in V}\sigma_i=0\}>0.
\end{equation}
Here $\mathcal{I}_{\rho\rho}=(\frac{\partial^2 \mathcal{I}(\rho)}{\partial \rho_i\partial \rho_j})_{i\in V, j\in V}\in\mathbb{R}^{|V|\times|V|} $, and $\sum_{i\in V}\sigma_i=0$ is the constraint for $\rho$ lying on the simplex set.
Notice the fact that \begin{equation}\label{Hessian_F}
\frac{\partial^2 \mathcal{I}(\rho) }{\partial \rho_i\partial \rho_j}=
\begin{cases}
-\frac{1}{\rho_i\rho_j}\frac{1}{\Delta x^2}t_{ij}&\textrm{if $j\in N(i)$}\ ;\\
\frac{1}{\rho_i^2}\sum_{k\in N(i)}\frac{1}{\Delta x^2}t_{ik} &\textrm{if $i=j$}\ ;\\
0& \textrm{otherwise},\\
\end{cases}
\end{equation}
where 
\begin{equation*}\label{tij}
t_{ij}=(\rho_i-\rho_j)(\log \rho_i-\log \rho_j)+(\rho_i+\rho_j)>0.
\end{equation*}
Hence
\begin{equation*}
\begin{split}
\sigma^T \mathcal{I}_{\rho\rho}(p) \sigma
=&\frac{1}{2}\sum_{(i,j)\in E}  t_{ij}\Big\{(\frac{\sigma_i}{\rho_i})^2+ (\frac{\sigma_j}{\rho_j})^2  -2 \frac{\sigma_i}{\rho_i} \frac{\sigma_j}{\rho_j}\Big\}\\
=&\frac{1}{2}\sum_{(i,j)\in E} t_{ij}(\frac{\sigma_i}{\rho_i}-\frac{\sigma_j}{\rho_j})^2\geq 0,
\end{split}
\end{equation*}
where $\frac{1}{2}$ is due to the convention that each edge $(i,j)\in E$ is summed twice. 

We next show that the strict inequality in \eqref{claim3} holds. Suppose \eqref{claim3} is not true, there exists a unit vector $\sigma^*$ such that 
\begin{equation*}
\sigma^{*T} \mathcal{I}_{\rho\rho} \sigma^*=\frac{1}{2}\sum_{(i,j)\in E} t_{ij}(\frac{\sigma_i^*}{\rho_i}-\frac{\sigma^*_j}{\rho_j})^2= 0.
\end{equation*}
Then $\frac{\sigma_1^*}{\rho_1}=\frac{\sigma_2^*}{\rho_2}=\cdots \frac{\sigma_n^*}{\rho_{|V|}}$.
Combining this with the constraint $\sum_{i\in V}\sigma_i^*=0$, we have  $\sigma_1^*=\sigma_2^*=\cdots=\sigma_{|V|}^*=0$, which contradicts that $\sigma^*$ is a unit vector.

Second, we show that
$\mathcal{K}(m,\rho)+\beta^{-2} \tau^2 \mathcal{I}(\rho)$ is strictly
convex in $(m,\rho)$. Notice that $(m,\rho)$ is in the interior of optimization domain, we
have $\rho_i>0$, thus the objective function is smooth. We shall show
that $\lambda(m,\rho)>0$, where \begin{equation}\label{claim4}
  \lambda(m,\rho):=\min_{h,\sigma}~
\begin{pmatrix}
h\\ \sigma
\end{pmatrix}^T
\left\{\begin{pmatrix}
\mathcal{K}_{mm} & \mathcal{K}_{m\rho}\\
\mathcal{K}_{\rho m} & \mathcal{K}_{\rho\rho}
\end{pmatrix}
+\beta^{-2} \tau^2\begin{pmatrix}
0 & 0\\
0 & \mathcal{I}_{\rho\rho}
\end{pmatrix}
\right\}
\begin{pmatrix}
h\\ \sigma
\end{pmatrix}
\end{equation}
subject to 
$$ h\in \mathbb{R}^{|E|},\quad \sigma \in\mathbb{R}^{|V|},\quad 
h^Th+\sigma^T\sigma=1,~\sum_{i\in V}\sigma_i=0.$$ 
Here, $\lambda(m,\rho)$ is the smallest eigenvalue of Hessian matrix for the objective function with tangent vectors $(h, \sigma)$. We last show that $\mathcal{K}(m,\rho)$ is a smooth, convex function in the interior of simplex set. We have
\begin{equation*}
\mathcal{K}(m,\rho)=\sum_{i+\frac{e_v}{2}\in E} \frac{2m_{i+\frac{e_v}{2}}^2}{\rho_i+\rho_{i+e_v}}.
\end{equation*}
Since $\frac{x^2}{y}$ is convex when $y>0$ and $\rho_i+\rho_{i+e_v}$ is concave on variables $\rho_i$, $\rho_{i+e_v}>0$. Then $\mathcal{K}$ is convex. From \eqref{claim3}, we have 
\begin{equation}\label{f}
\mathcal{J}(h,\sigma):=\begin{pmatrix}
h\\ \sigma
\end{pmatrix}^T
\begin{pmatrix}
\mathcal{K}_{mm} & \mathcal{K}_{m\rho}\\
\mathcal{K}_{\rho m} & \mathcal{K}_{\rho\rho}
\end{pmatrix}
\begin{pmatrix}h \\ \sigma\end{pmatrix}+\beta^{-2} \tau^2\sigma^T\mathcal{I}_{\rho\rho}\sigma \geq 0 .\end{equation}
We claim that the inequality in \eqref{f} is strict. Suppose there exists $(h^*,\sigma^*)$, such that \eqref{f} is zero, i.e.
\begin{equation*}
\mathcal{J}(h^*, \sigma^*)=0.
\end{equation*}
In this case, from \eqref{claim3}, $\sigma^*=0$ . Thus \eqref{f} forms 
\begin{equation*}
\mathcal{J}(h^*,\sigma^*)={h^*}^T\mathcal{K}_{mm} h^*=0.
\end{equation*}
Since $\mathcal{K}_{mm}=\textrm{diag}(\frac{4}{\rho_i+\rho_{i+e_v}})_{i+\frac{e_v}{2}\in E}$ is strictly positive, we have $h^*=0$, which contradicts the fact that $h^Th+\sigma^T\sigma=1$.
From the above statements, we prove that there exists a unique solution $\rho^{k+1}$

(ii) Denote $(m^*, \rho^{k+1})$ as the minimizer of variation problem \eqref{new_form3}. Then 
\begin{equation*}
F(\rho^k, 0)\leq F(\rho^{k+1},m^*).
\end{equation*}
This further implies 
\begin{equation*}
\beta^{-2} \tau^2\mathcal{I}(\rho^k)+2\tau\mathcal{E}(\rho^k)\leq \beta^{-2} \tau^2\mathcal{I}(\rho^{k+1})+2\tau \mathcal{E}(\rho^{k+1}),
\end{equation*}
which finishes the proof. 

(iii) holds since $\mathcal{I}(\rho)$ goes to infinity on the boundary
of simplex set.

(iv) is true, because the continuity equation in \eqref{new_form3} satisfies 
$$\sum_{i}(\rho^{k+1}_i-\rho^k_i)=-\sum_{i}\sum_{e_i}\frac{1}{\Delta x}(m_{i+\frac{e_i}{2}}-m_{i-\frac{e_i}{2}})=0.$$
This finishes the proof.
 \end{proof}

 \subsection{Optimization method}
To solve \eqref{new_form3}, we first rewrite our problem into a vector form. Let $u = (\vecrho, \vecm )$, then \eqref{new_form3} can be written as 
\begin{equation} \label{scheme100}
\min_{u} F(u), \quad \text{s.t.} \quad \Amat u = b, ~ \Smat u \geq 0,
\end{equation}
where $F(u)$ is defined in \eqref{F1D}, $\Amat$ is the matrix representation of the constraint \eqref{constraint1D} or \eqref{constraint2D}, and $\Smat$ is a selection matrix that only selects the $\rho$ components in $u$. Let $\indi$ be the indicator function, then \eqref{scheme100} can be further reformulated as 
\begin{equation} \label{objfun}
\min_{u} F(u) + \indi( u), \quad \quad  \indi ( u) = \left\{ \begin{array}{cc} 0 & \Amat u = b, ~ \Smat u \geq 0 \\ +\infty & \text{ otherwise} \end{array} \right. \,.
\end{equation}
Here $F(u)$ defined in either \eqref{F1D} or \eqref{F2D} is a smooth, convex function (provided $E$ is convex), and indicator function $\indi$ is also convex. Therefore we adopt the (approximate) sequential quadratic programming to solve it:
\begin{equation} \label{proximalNewton}
\begin{cases}{}
z^{(l+1)} \in \arg\min_{z} \half (z-u^{(l)})^T  \Hmat^{(l)} (z-u^{(l)}) + \nabla F(u^{(l)})^T (z-u^{(l)}) + \indi( z)\,,
\\ u^{(l+1)} =u^{(l)} + t_{l} (z^{(l+1)}-u^{(l)})\,.
\end{cases}
\end{equation}
where $\Hmat^{(l)}$ is either the Hessian $\nabla^2 F(u^{(l)})$ or an approximation of it.

\begin{tabbing}
aaaaa\= aaa \=aaa\=aaa\=aaa\=aaa=aaa\kill  
 \rule{\linewidth}{0.8pt}\\
 \textbf{Algorithm}: Sequential quadratic programming for one step regularized JKO\\
   \rule{\linewidth}{0.8pt}\\
  \1 \textbf{Input}: $\rho(t_0,x)$ $\rho(t_0,x) = 0$, $\text{Iter}_{\max}$\\
    \3 Parameter $\beta>0$, step size $\alpha_k\in(0,1)$, discretization parameters $\Delta x$, $\Delta t$ \\
  \1 \textbf{Output}: $\rho( t_k,x)$ for $1\leq k \leq n$\\
    \rule{\linewidth}{0.8pt}\\ 
1.  \1 \For $k = 2, 3,  \cdots ,n $ \textrm{\bf do } \\
2. \2  $m^{(0)} = m(t_k,x)$ \\
3. \2 $\rho^{(0)} = \left\{ \begin{array}{cc} 2 \rho(t_k,x) - \rho(t_{k-1},x) & \text{if} \quad 2 \rho(t_k,x) - \rho(t_{k-1},x) \geq1e-6 \\ \rho(t_k,x) & \text{if}  \quad 2 \rho(t_k,x) - \rho(t_{k-1},x) <1e-6  \end{array} \right.$ \\
4.  \2 $u^{(0)}  = (\rho^{(0)}, m^{(0)})$, $l=0$\\
5.  \2 \While $l \leq \text{Iter}_{\max}$, \textrm{\bf do} \\
6. \3 $z^{(l+1)} \in \arg\min_{z} \half (z-u^{(l)})^T  \Hmat^{(l)} (z-u^{(l)}) + \nabla F(u^{(l)})^T(z-u^{(l)}) + \indi( z)$ \\
7. \3 $u^{(l+1)} =u^{(l)} + t_{l} (z^{(l+1)}-u^{(l)})$  \\
8. \3 update $\Hmat^{(l+1)}$ and $\nabla F(u^{(l+1)})$ (see Remark~\ref{remark:H})\\
9. \2  \textbf{until} stopping criteria is achieved \\ 
10. \2 $\rho(t_{k+1},x) =  \rho^{(l+1)}$, \quad $ m(t_{k+1},x) = m^{(l+1)}$ \\
11.  \1 \End\\
   \rule{\linewidth}{0.8pt}
\end{tabbing}

There are several approaches to solve the subproblem in line 6. Among them, we have tried interior point method, projected preconditioned conjugate gradient method \cite{pCG01}, and first order fast iterative shrinkage thresholding algorithm (FISTA) \cite{FISTA09}. In our case where $\Hmat^{(l)}$ is sparse but ill conditioned, we found that the MATLAB built-in function `quadprog' with interior point solver performs the best.

\subsection{Convergence}
In this section, we analyze the convergence of \eqref{proximalNewton}, especially the role that $\Hmat$ plays. We have the following assumptions:
\begin{itemize}
\item[(A1)] $m\Imat \preceq \nabla^2F \preceq M \Imat$, \quad $M \geq m >0$; 
\item[(A2)] the subproblem in \eqref{proximalNewton} is solved exactly. 
\end{itemize}
First we note that $\ulp$ can be rewritten as
\begin{align}
\ulp &= \prox_{t_l \indi}^{\Hl} (\Ul  - t \Hl^{-1} \nabla F(\Ul) )  \nonumber
\\ & =  \arg\min_u \frac{1}{2t_l} \|  u - \Ul + t_l \Hl^{-1} \nabla F(\Ul)  \|_\Hl  + \indi (u)\,,
\label{ulp}
\end{align}
where $\|u\|_\Hl = u^T \Hl u$. Further, let $u^*$ be the unique minimizer to \eqref{objfun}, then $u^*$ solves
\begin{equation}
u^* = \prox_{t \indi}^{\Hmat} (u^*  - t \Hmat^{-1} \nabla F(u^*) ) \,,
\label{ustar}
\end{equation}
where $t>0$ and $\Hmat$ is any positive definite matrix. Our first result is, when $\Hl$ is only an approximation of $\nabla^2 F(\Ul)$, we get first order convergence with convergence rate depends on the condition number $\Hl^{-1} \nabla^2 F(\Ul)$. More specifically, we have
\begin{theorem}\label{quasiN}
Consider uniform time step $t$. Let $G_l = \int_0^1 \nabla^2 F(u^* + s(u_l - u^*)) \rd s$, then 
\begin{equation} \label{linearConv}
\norm{\ulp - u^*}_{\Hl} \leq \left| \frac{1-\kappa}{1+\kappa}\right| \norm{\Ul - u^*}_{\Hl}\,,
\end{equation}
where $\kappa$ is the condition number of $\Hl^{-1} G_l$.
\end{theorem}
To prove the above theorem, we need the following lemma on the contraction of the proximal operator. 
\begin{lemma} \label{lemma0}
If $u = \prox_\indi^\Hmat (x)$, $v = \prox_\indi^\Hmat (y)$, where $\indi$ is a convex function, and $\Hmat$ is a positive definite matrix, then we have $(u-v)^T \Hmat (x-y) \geq \norm{u-v}_\Hmat^2$. Consequently, $\| u - v\|_\Hmat \leq \| x- y\|_\Hmat$.
\end{lemma} 
The proof of this lemma is standard, so we omit the details and directly jump to the proof of Theorem \ref{quasiN}. 
\begin{proof}[Proof of Theorem \ref{quasiN}]
By virtue of \eqref{ulp}, and \eqref{ustar} with $\Hmat = \Hl$, we have 
\begin{align}
\normHl{\ulp - u^*} & = \normHl{\prox_{t \indi}^{\Hl} (\Ul  - t \Hl^{-1} \nabla F(\Ul) )  - \prox_{t \indi}^{\Hmat} (u^*  - t \Hmat^{-1} \nabla F(u^*) ) } \nonumber
\\ & \leq \normHl{ \Ul - u^* -  t \Hl^{-1} \nabla F(\Ul) + t \Hl^{-1} \nabla F(u^*)}  \nonumber
\\ &= \normHl{\Ul - u^* - t \Hl^{-1} G_l (\Ul - u^*)}  \nonumber
\\ & \leq \normHl{\Imat - t\Hl^{-1} G_l } \normHl {\Ul - u^*}\,. \label{16}
\end{align}
where the first inequality uses Lemma~\ref{lemma0}.
Since both $\Hl$ and $G_l$ are positive definite from assumption (A1), we denote $\lambda_1 \geq \lambda_2 \geq \cdots \geq \lambda_N>0 $ as the eigenvalues of $\Hl^{-1} G_l $, then
\begin{equation} \label{624}
\normHl{\Imat - t\Hl^{-1} G_l } \leq \max \{ \normHl{(1-t\lambda_1) \Imat} , ~\normHl{(1-t\lambda_N) \Imat} \}\,.
\end{equation}
Here we have used the fact that for two symmetric positive semi-definite matrix $\Amat$ and $\Bmat$, if $A \preceq B$, then $\normHl{\Amat} \leq \normHl{\Bmat}$. Indeed, since $\normHl{\Amat} = \sup_x {x^T \Amat^T \Hl \Amat x}/{x^T \Hl  x}$, let $y = \Hl^{\half} x$, we have $\normHl{\Amat} = \sup_y {y^T \Hl^{-\half} \Amat^T \Hl^\half \Hl^\half \Amat \Hl^{-\half} y}/{ y^T y}$, therefore, $\normHl{\Amat} = \|\Hl^{-\half} \Amat \Hl^\half \|_2 = \|\Amat\|_2$.

Choose $t = \frac{2}{\lambda_1 + \lambda_N}$ in \eqref{624} so that it minimize its RHS, and plug it into \eqref{16} to get the final result. 
\end{proof}

\begin{remark}[Comparison with proximal gradient]
Consider the proximal gradient method for solving \eqref{objfun}
\begin{equation} \label{proxgrad}
\ulp = \prox_{t\indi}(\Ul - t \nabla F(\Ul))\,.
\end{equation}
Comparing it to \eqref{ulp}, we see that \eqref{ulp} is a preconditioned version of \eqref{proxgrad}. Indeed, substituting  
\eqref{ustar} with $\Hmat = \Imat$ from \eqref{proxgrad}, we get
\begin{align} \label{conv22}
\normtwo{\ulp - u^*} \leq \normtwo{\Imat - t G_l} \normtwo{\Ul - u^*}  \leq \left|\frac{1-\kappa_G}{1+\kappa_G}\right| \normtwo{\Ul - u^*}\,,
\end{align}
where $\kappa_G$ is the condition number of $G_l$. Therefore when $G_l$ is ill-conditioned, which is the case in the presence of vacuum due to the nonlinear diffusion, the convergence rate in \eqref{conv22} is much slower than that in \eqref{linearConv}. 
\end{remark}

\begin{remark}[Choice of $\Hl$] \label{remark:H} In our problems when
  the energy term $\energy$ only contains internal and potential
  energies, both of which are local in $\rho$, we directly compute the
  Hessian of $F$ as $\Hmat$, since in this case the Hessian is sparse
  and very cheap to compute. When $\energy$ also contains interaction
  energy, the Hessian of $F$ is dense, and we instead approximate the
  Hessian of $F$ by replacing the interaction energy with entropy
  $ \int \rho \log \rho \rd x$, and adjusting the parameter in the
  Fisher information term to approximate the original Hessian. More
  specifically, for the general case where $F(u)$ is
\[
F(u) = \int_0^1  \!\! \int_\Omega \frac{m^2}{\rho}   + \beta^{-2} \tau^2\rho \left( \nabla \log \rho \right)^2\rd x \rd t  + 2\tau  \int V(x) \rho + U(\rho) + \half ( W\ast \rho) \rho \rd x.
\]
We compute the Hessian $\Hmat$ of 
\begin{equation}
\tilde{F}(u) = \int_0^1  \!\! \int_\Omega \frac{m^2}{\rho}   + \tilde{ \beta}^{-2} \tau^2 \rho \left( \nabla \log \rho \right)^2\rd x \rd t  + 2\tau  \int V(x) \rho + U(\rho) +  \rho \log \rho ~ \rd x 
\end{equation}
as an approximation of $\nabla^2 F$. Here $\tilde{\beta}^{-2}$ is an integer multiple of $\beta^{-2}$. 
\end{remark}

We close this subsection by stating following result that when $\Hl$
is exact Hessian of $F$, we obtain local quadratic convergence. We
omit the proof, which is standard.
\begin{theorem}
Assume further that $\normtwo{\nabla^2 F(x) - \nabla^2 F(y)} \leq L \normtwo{x-y}$. If $\Hl = \nabla^2 F(\Ul)$ in \eqref{proximalNewton}, then for sufficiently large $l$, $t_l \rightarrow 1$, and $\Ul$ satisfies
\[
\normtwo{\ulp - u^*} \leq \frac{L}{2m} \normtwo{ \Ul - u^*}^2\,.
\] 
\end{theorem}

\section{Numerical examples}\label{section4}
In this section, we demonstrate several numerical examples to show the accuracy and efficiency of the proposed scheme \eqref{new_form3}. The stopping criteria in the sub optimization problem (see line 9 in the Algorithm) is chosen as 
\[
|F(u^{(l+1)}) - F(u^{(l)})|/|F(u^{(l)})|< \text{TOL}\,,
\]
where $\text{TOL}$ is set to be $10^{-6}$ unless otherwise specified. 
\subsection{1D problem}
\subsubsection{Heat equation}
For heat equation, we directly choose $\beta = 1$, and let the initial condition be
\[
\rho(x,0) = e^{-100(x-1)^2}  + 10^{-5}\,, \qquad x \in [0,2]\,.
\]
In Fig.~\ref{fig1} on the left, we apply \eqref{new_form3} on a coarse mesh and compare the solution with the reference solution obtained by implicit diffusion solver on a fine mesh, and observe good agreements. 
When $\Delta x$ is sufficiently small, we check the first order accuracy of our scheme by computing the following relative error 
\begin{equation}\label{err1}
e_\tau = \| \rho_{\tau}(\cdot,\tmax) - \rho_{\tau/2}(\cdot,\tmax) \|_{l^1} = \sum_{j=1}^{N_x} |(\rho_\tau)_j(\tmax) - (\rho_{\tau/2})_j(\tmax)| \Delta x
\end{equation}
and error with respect to the reference solution 
\begin{equation} \label{err2}
e_\tau = \| \rho_\tau (\cdot,\tmax ) - \rho_{\text{ref}}(,\tmax)\|_{l^1}
\end{equation}
with decreasing $\tau$. 

\begin{figure}[h!]
\includegraphics[width=0.48\textwidth]{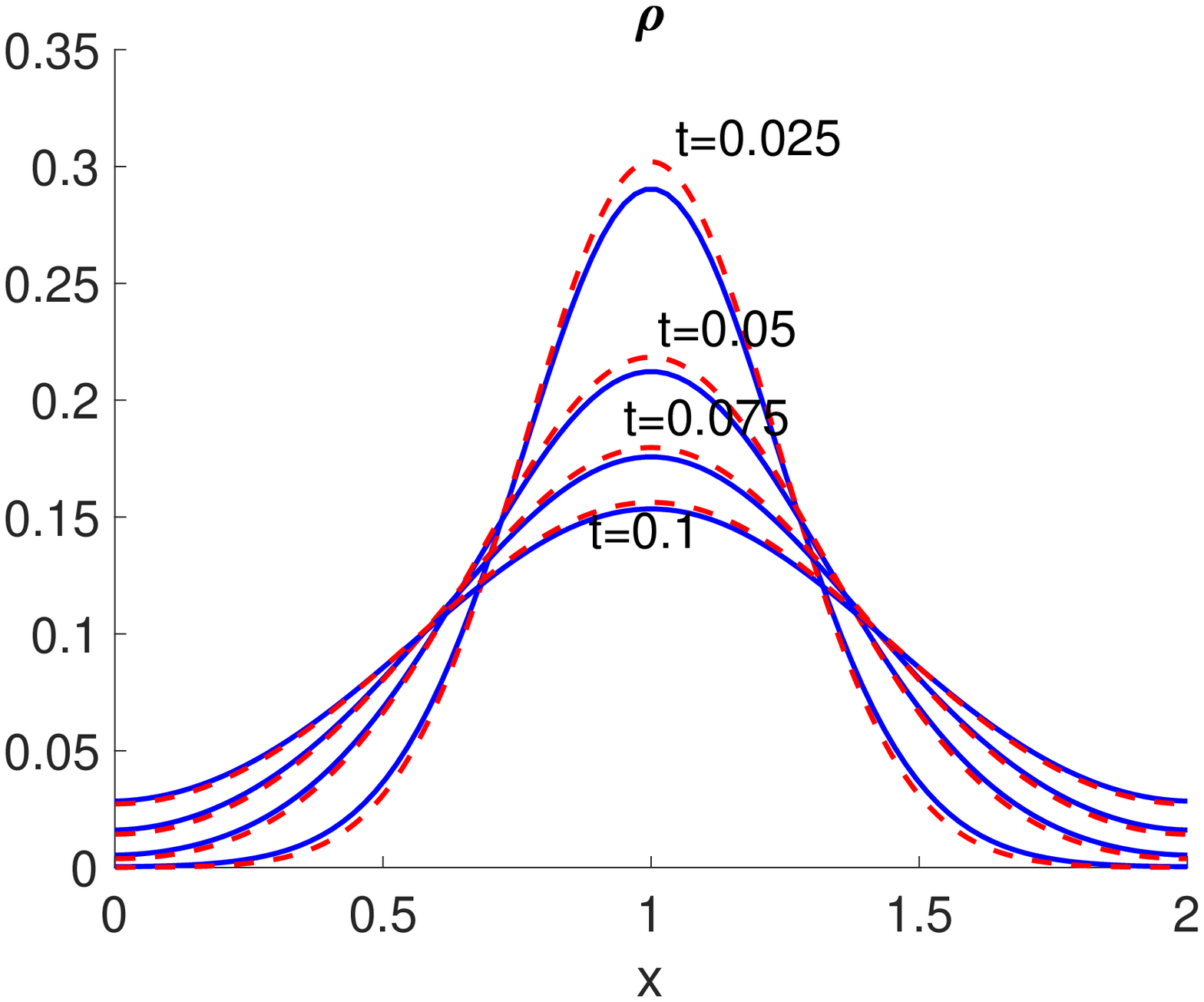}
\includegraphics[width=0.48\textwidth]{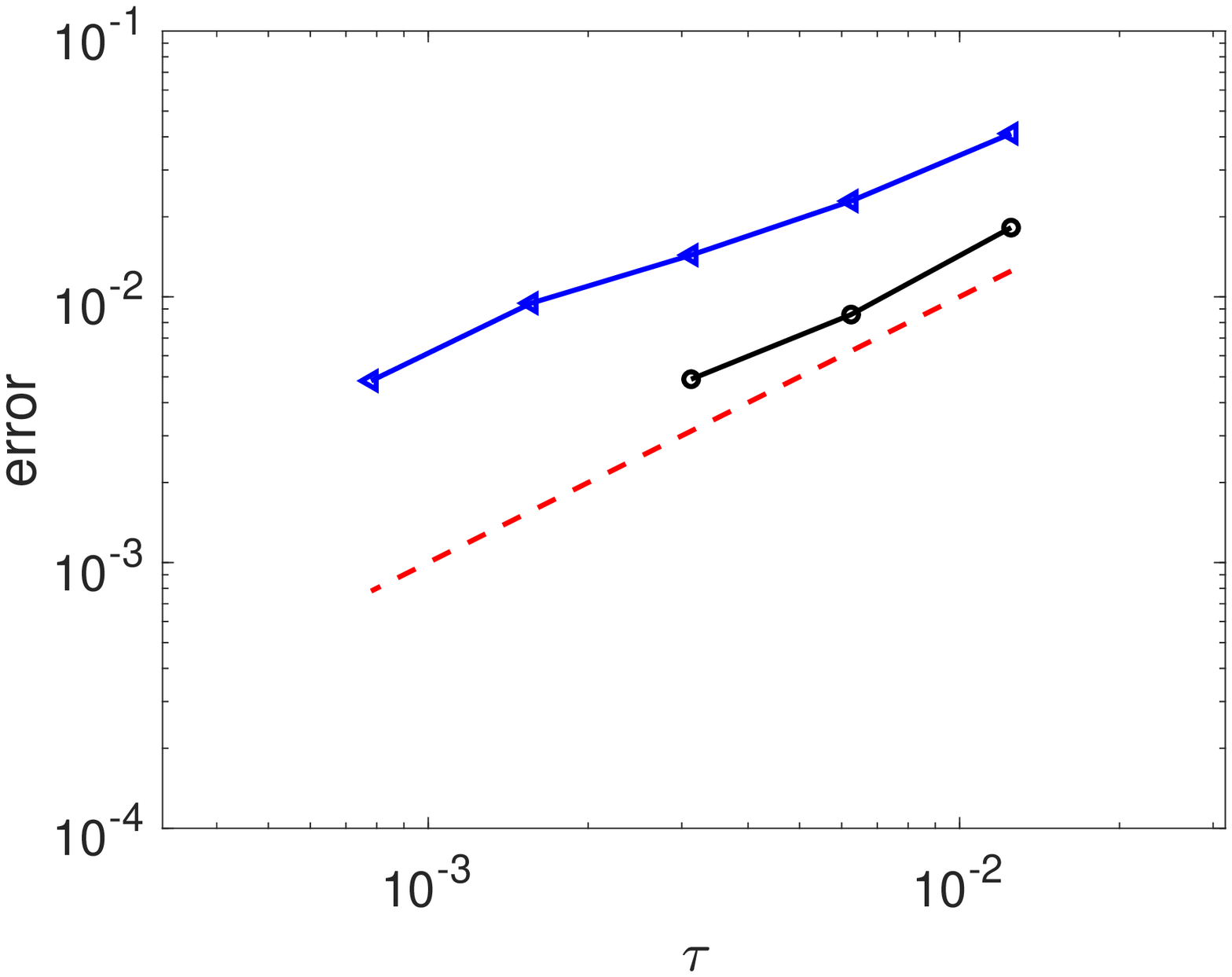}
\caption{Left: comparison of our scheme with the reference solution obtained by fully implicit diffusion solver with a refined grid (i.e., $\Delta x =0.0025 $, $\tau = 1.25\times 10^{-4}$). In our scheme, we used $\Delta x = 0.0202$, $\tau = 0.0025$, and the stoping criteria in the quadratic programming is $1.25\times 10^{-7}$. Right: check the order of accuracy with sufficiently small $\Delta x = 0.005$. The blue triangle is computed with \eqref{err2} and black circle with \eqref{err1}. The red dashed curve indicates the first order accuracy. Here $\tmax = 0.1$.}
\label{fig1}
\end{figure}

\subsubsection{Porous medium equation}
The porous medium equation
\begin{equation} \label{PMEeqn}
\partial_t \rho = \Delta \rho^{m}\,, \quad m>1, 
\end{equation}
can be considered as the Wasserstein gradient flow of the energy \eqref{eqn:energy}, with  ${U}(\rho) = \frac{1}{m-1} \rho^m$ and $V = W = 0\,$. A well-known family of exact solutions is given by Barenblatt profiles (c.f. \cite{VazquezPME}), which are densities of the form
\begin{equation} \label{eqn:Barenblatt}
\rho(x,t) = (t+t_0)^{-\frac{1}{m+1}} \left( C - \alpha \frac{m-1}{2m(m+1)}  x^2 (t+t_0)^{-\frac{2}{m+1}} \right)_{+}^{\frac{1}{m-1}} , \qquad \text{ for } C, t_0 > 0 .
\end{equation}
In our tests, we choose $m=2$, $t_0 = 10^{-3}$ and $C = 0.8$. We plot the evolution of the numerical solution over time in Fig.~\ref{fig:porous1}, and we observe good agreement with the exact solution of the form \eqref{eqn:Barenblatt}, which is shown in dashed curve. 
\begin{figure}[h!]
\centering
\includegraphics[width=0.48\textwidth]{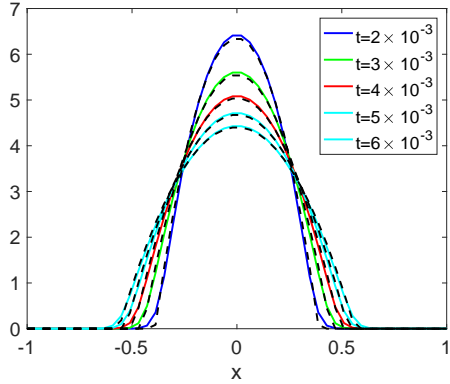}
\caption{Evolution of the solution $\rho(x,t)$ to the one dimensional porous medium equation \eqref{PMEeqn}, with $m=2$, on the domain $\Omega = [-1,1]$. Here the numerical parameters are $\Delta x = 0.0408$, $\tau = 0.5\times 10^{-3}$, $\beta = 1$, TOL = $10^{-8}$.}
\label{fig:porous1}
\end{figure}
Next, we examine how the entropic regularization affects the solution. In the left plot of Fig.~\ref{fig:porous2}, we compare solutions obtained by our scheme with various $\beta^{-1}$ and we observe that near the boundary of the solutions' support where a non-smooth transition is expected (see the black dashed curve for the exact solution), our solution with regularization inevitably smooth out the solution. As $\beta^{-1}$ decreases, the solution improves moderately. On the right, we compare the error between our solution with the exact formula \eqref{eqn:Barenblatt}:
\begin{equation}\label{eqn:err3}
e(t) = \| \rho(\cdot, t) - \rho_{\text{exact}}(\cdot, t)\|_{l^1}.
\end{equation}
As expected, smaller $\beta^{-1}$ leads to better accuracy. However, as the regularization parameter is closely related to the convexity of the problem and thus affects the convergence of the method, one has to strike a balance between the accuracy and efficiency by choosing $\beta^{-1}$ neither too big nor too small.

\begin{figure}[h!]
\centering
\includegraphics[width=0.48\textwidth]{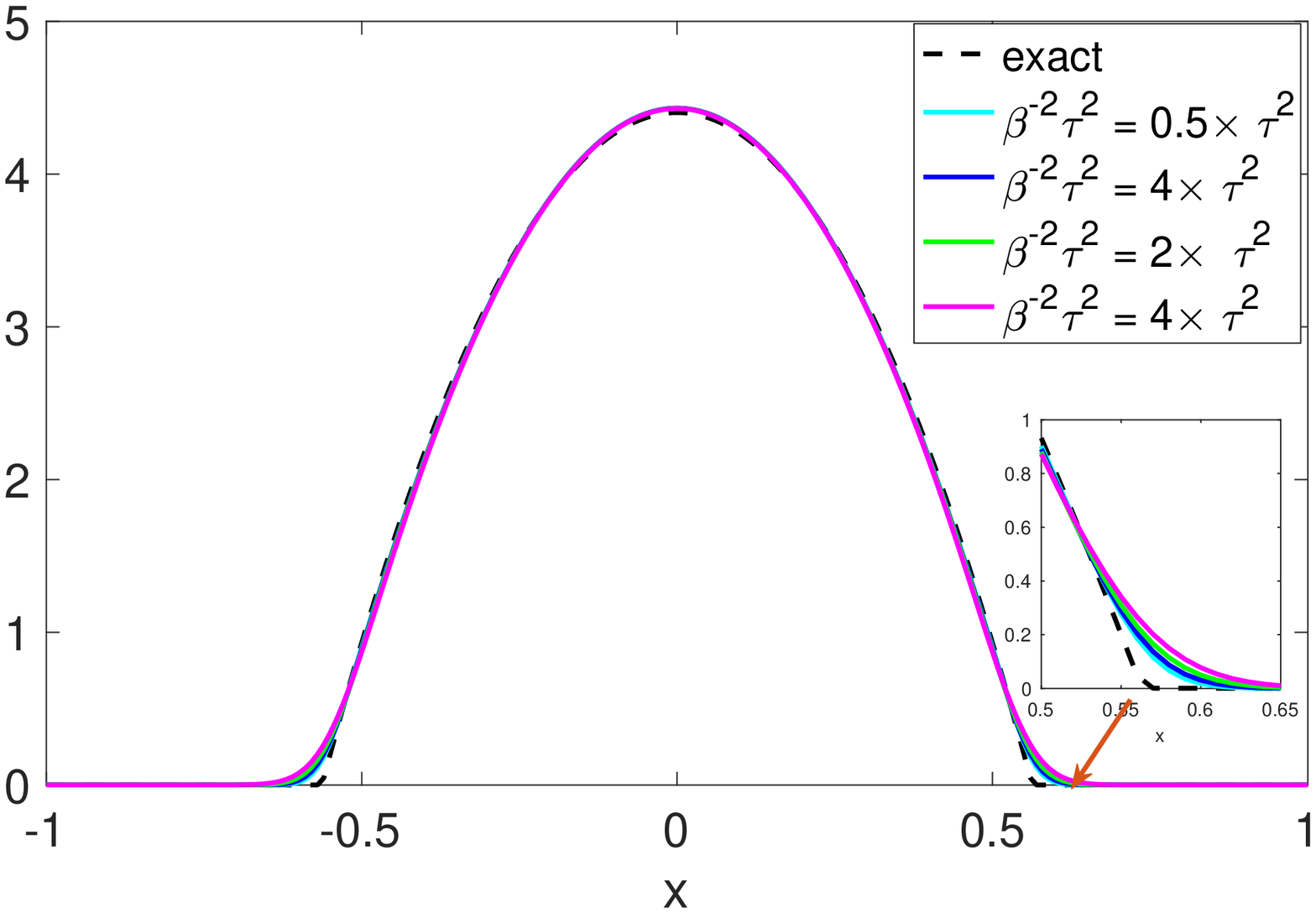}
\includegraphics[width=0.48\textwidth]{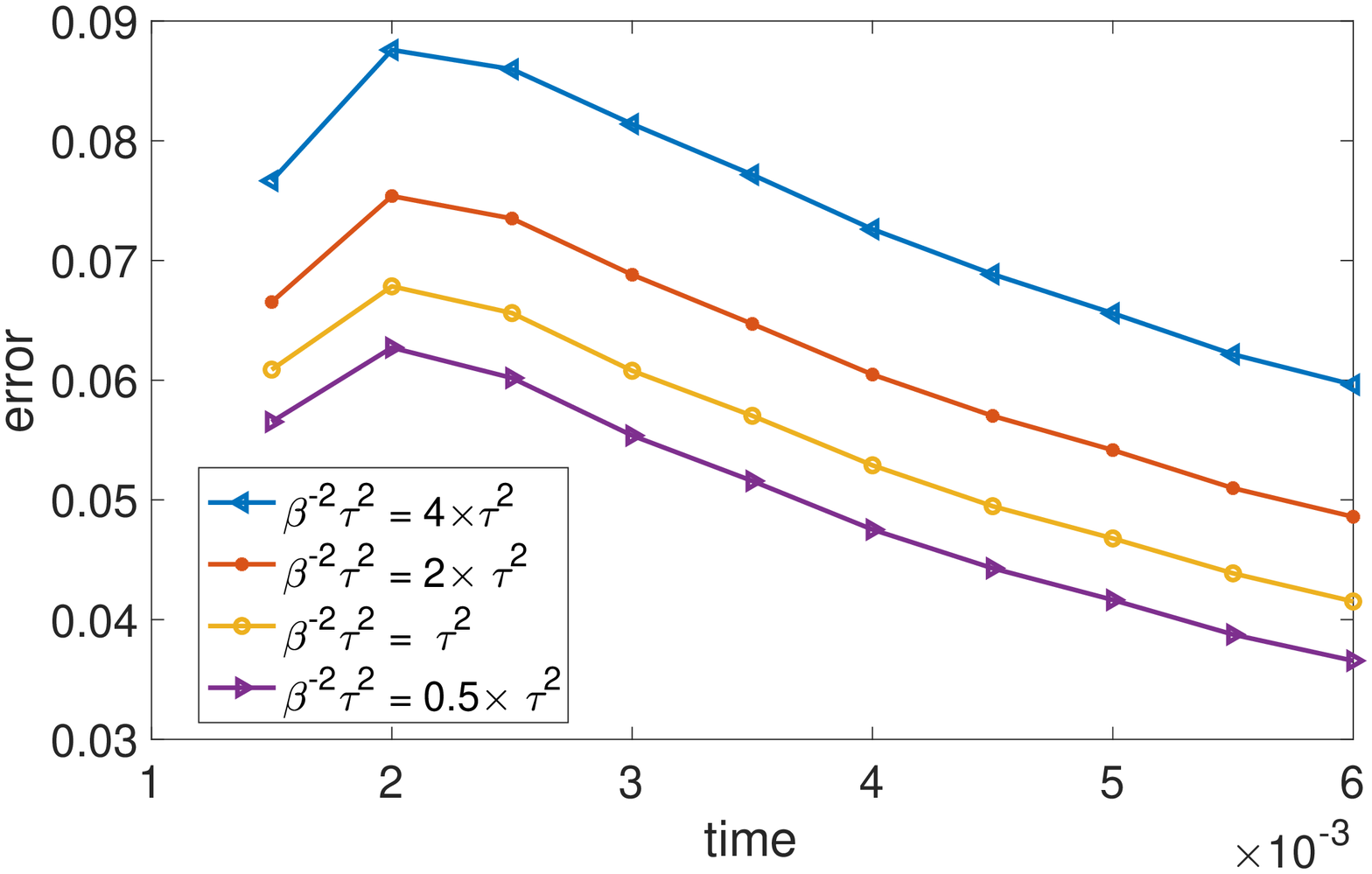}
\caption{Examine the effect of $\beta$. Left: comparison of solutions at $t=6\times 10^{-3}$ with various choices of $\beta$. A zoom in figure near the boundary of the solution's support is also provided. Right: plot the error \eqref{eqn:err3} with time for different $\beta$. Here $\Delta x = 0.01$, $\tau = 5\times 10^{-4}$, TOL = $10^{-8}$. }
\label{fig:porous2}
\end{figure}

\subsubsection{Nonlinear Fokker-Planck equation}
Next, we consider a nonlinear variant of the Fokker-Planck equation, by replacing the linear diffusion with the porous medium type nonlinear diffusion (\ref{PMEeqn}):
\begin{equation} \label{nonFPeqn}
\partial_t \rho = \nabla\cdot (\rho \nabla V) + \Delta \rho^{m}\,, \quad V: \Rd \to \R, \quad m>1, 
\end{equation}
When $V$ is a confining drift potential, all solutions approach the unique steady state
\[
\rho_\infty(x) =  \left( C -  \frac{m-1}{m}{V(x)} \right)_{+}^{\frac{1}{m-1}} \,,
\]
where $C>0$ depends on the mass of the initial data, i.e., denote $M = \int \rho_0 \rd x$, then $C = \left(\frac{3M}{8} \right)^{2/3}$, see \cite{CaTo00,CJMTU} for a derivation.

In Figure \ref{fig:nFP1}, we compute the solutions to the nonlinear Fokker-Planck equation with $V(x) = \frac{x^2}{2}$, $m=2$, and initial data given by $\rho(x,0) = \frac{1}{8} \left( \frac{1}{\sqrt{2\pi}\sigma} e^{-x^2/2\sigma^2} + 10^{-8}  \right) $. On the left, we plot the evolution of the density $\rho(x,t)$ towards the steady state $\rho_\infty(x)$. On the right, we compute the rate of decay of the corresponding energy \eqref{eqn:energy} as a function of time, observing exponential decay as the solution approaches equilibrium, which is consistent with the analytic results on convergence to equilibrium \cite{CaTo00, CDT07}.

\begin{figure}[h!]
\centering
\includegraphics[width=0.48\textwidth]{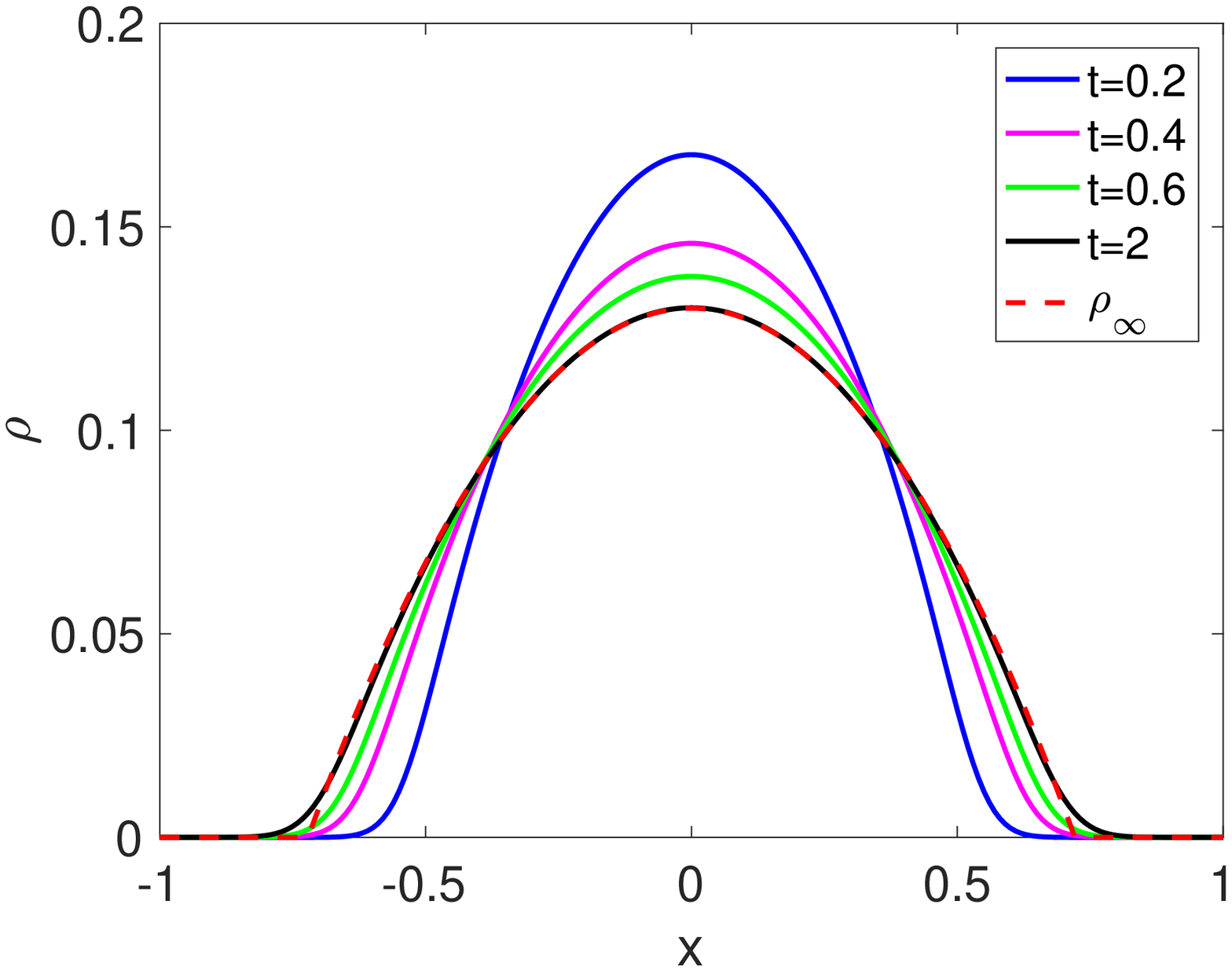}
\includegraphics[width=0.48\textwidth]{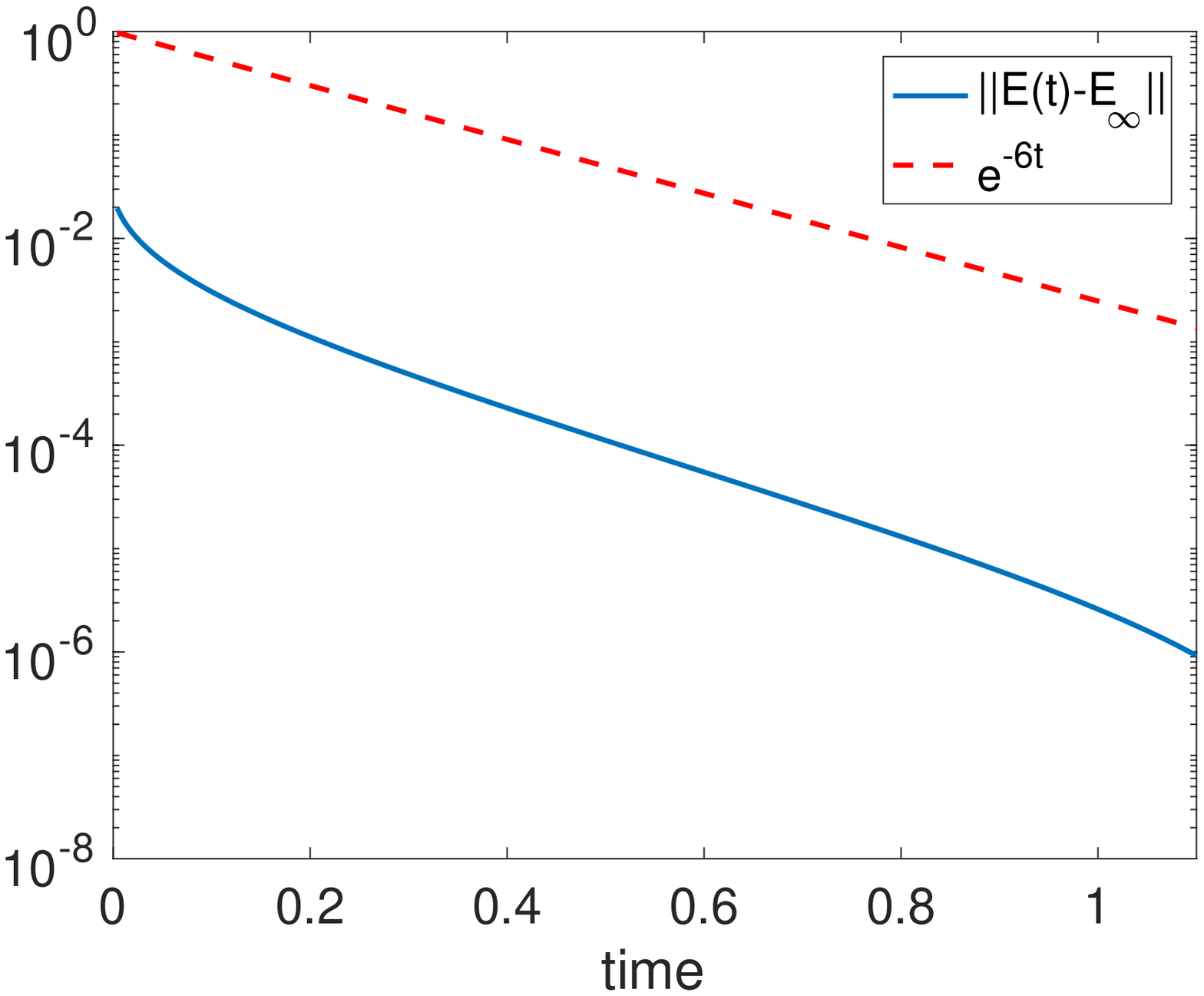}
\caption{Evolution of $\rho(x,t)$ to the one dimensional nonlinear Fokker Planck equation \eqref{nonFPeqn} with $V(x) = \frac{x^2}{2}$, $m=2$, and $x\in [-1,1]$. Numerical parameters are $\tau = 0.004$, $\Delta x = 0.01$, $\beta^{-2}\tau^2 = \tau^2/40 = 4 \times 10^{-7}$.  }
\label{fig:nFP1}
\end{figure}

\subsubsection{Aggregation equation}
In this subsection, we consider a nonlocal aggregation equation of the form
\begin{align} \label{aggeqn1}
\partial_t \rho = \nabla\cdot (\rho \nabla W*\rho)  \,, \quad W: \Rd \to \RR\,,
\end{align}
where the interaction kernel $W$ is repulsive at short length scales and attractive at longer distances. In particular, we choose the following kernel with logarithmic repulsion and quadratic attraction
\begin{equation} \label{Weqn1}
W(x) = \frac{|x|^2}{2} - \text{ln}(|x|)\,,
\end{equation}
then it is proved that there exists a unique equilibrium profile \cite{CFP12}, given by
\[\rho_\infty (x) =   \frac{1}{\pi} \sqrt{(2-x^2)_+}  .\]
In practice, to avoid evaluation of $W(x)$ at $x=0$, we set $W(0)$ to equal the average value of $W$ on the cell of width $2h$ centered at 0, i.e., $W(0) = \frac{1}{2h} \int_{-h}^{h} W(x) \rd x$, where we compute this value analytically. (See also \cite{CCH15, CCWW18} for a similar treatment.)

The numerical results are gathered in Fig.~\ref{fig:agg1}. On the left, we simulate the solution to the aggregation equation with Gaussian initial data $\rho(x,0) = \frac{1}{\sqrt{2\pi}\sigma} e^{-\frac{x^2}{2\sigma^2}} + 10^{-8}$ at varying times, observing convergence to the equilibrium profile $\rho_\infty(x)$. On the right, we compute the rate of the decay of the energy as a function of time, observing exponential decay with the theoretical rate as obtained by Carrillo et. al. \cite{CFP12}.

\begin{figure}[h!]
\centering
\includegraphics[width=0.48\textwidth]{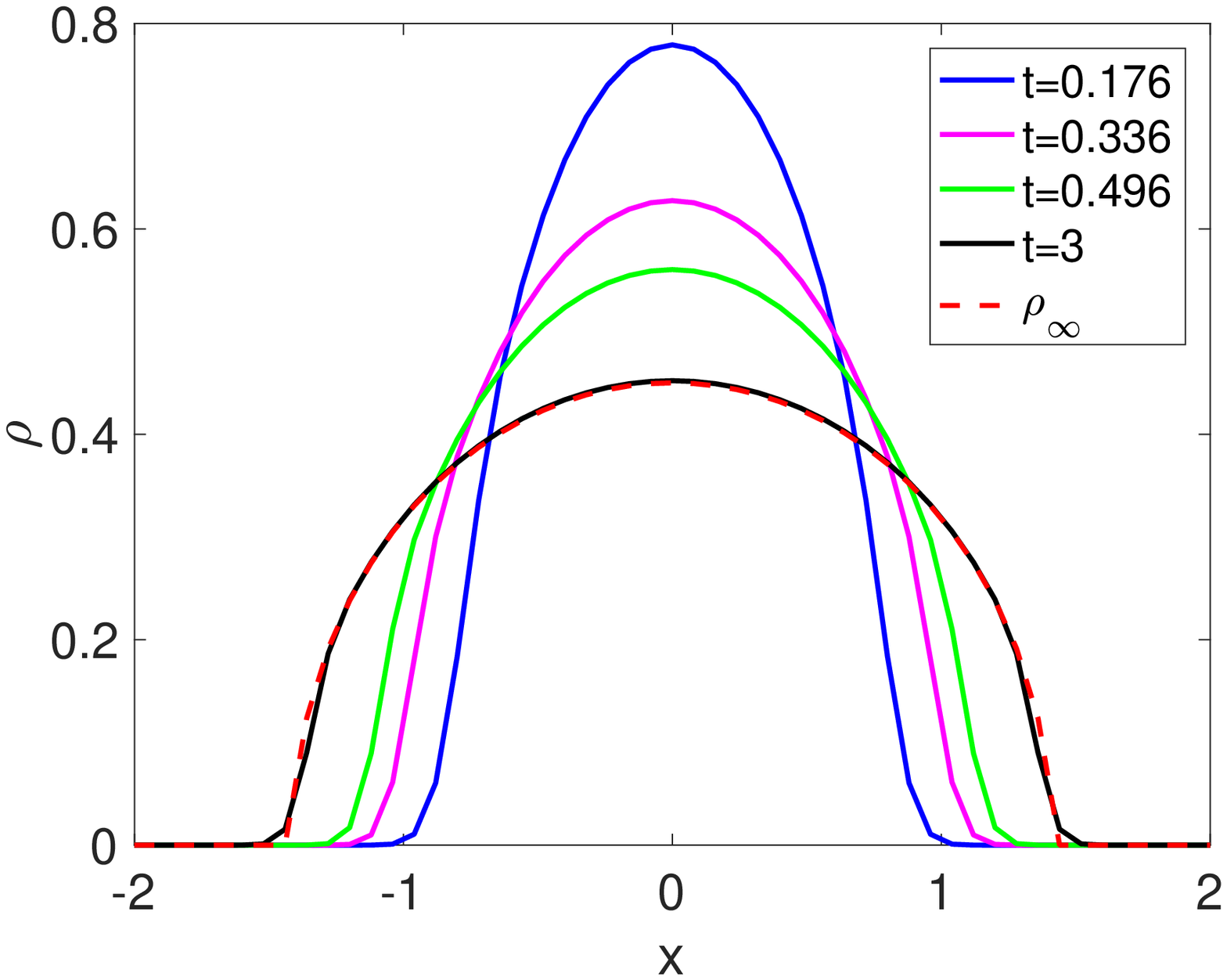}
\includegraphics[width=0.48\textwidth]{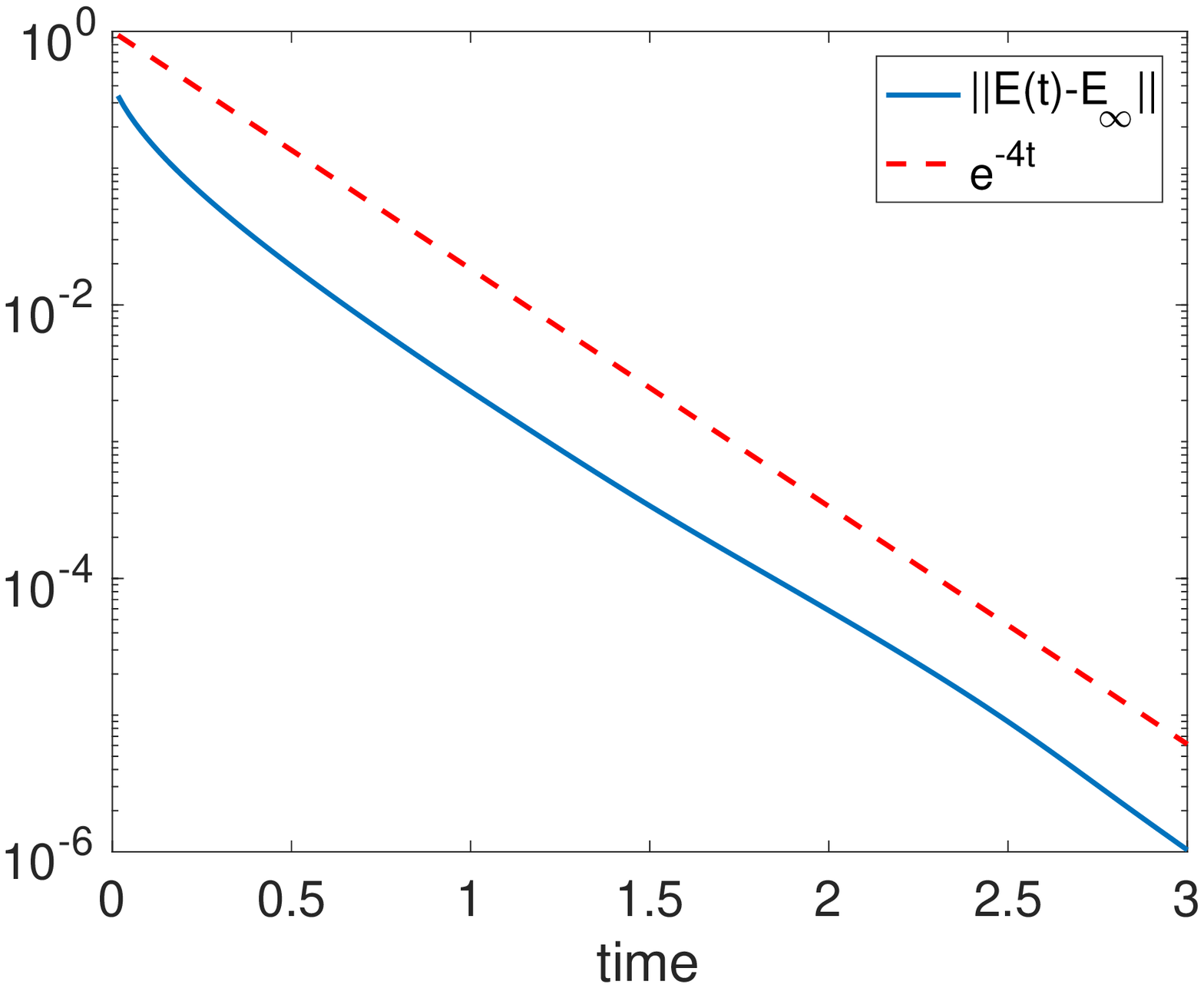}
\caption{Evolution of $\rho(x,t)$ to the one dimensional nonlocal aggregation equation \eqref{aggeqn1} with \eqref{Weqn1}, and $x\in [-2,2]$. Numerical parameters are $\tau = 0.016$, $\Delta x = 0.08$, $\beta^{-2}\tau^2 = \tau^2/640 = 4 \times 10^{-7}$.  }
\label{fig:agg1}
\end{figure}

\subsubsection{Derrida-Lebowitz-Speer-Spohn (DLSS) equation}
We now consider a DLSS equation
\begin{equation*}
\partial_t\rho=\nabla\cdot\left[\rho \nabla \left( \half \delta \mathcal I (\rho)+V(x) \right)\right],
\end{equation*}
where $\mathcal{I}(\rho)=\int_\Omega |\nabla\log\rho(x)|^2\rho(x) \rd x $ and $\delta$ is the first variation operator with 
\begin{equation*}
\delta\mathcal{I}(\rho)=\|\nabla\log\rho(x)\|^2-\frac{2}{\rho(x)}\nabla\cdot(\rho(x)\nabla\log\rho(x)).
\end{equation*}
As written, the DLSS equation can be considered as the Wasserstein gradient flow of functional: $\mathcal{E}(\rho)  = \int_{\mathbb{R}^n}\frac{1}{2}\|\nabla\log\rho(x)\|^2\rho(x)+ V(x)\rho(x) \rd x$. In practice, we just replace $\beta^{-2}\tau^2$ by $\tau$ and choose $\energy(\rho)=\int V(x)\rho(x) \rd x $ in \eqref{MP}.  

When $V(x) = \frac{x^2}{2}$, the stationary solution $\rho_{\infty}$ has an explicit form
\begin{equation} \label{equi_DLSS}
\rho_{\infty}(x)=\frac{1}{\sqrt{2\pi}}e^{-\frac{x^2}{2}}.
\end{equation}
With double-Gaussian initial condition 
\[
\rho(x,0) = \frac{1}{2\sqrt{2\pi} \theta} \left( e^{-\frac{(x-1.5)^2}{2\theta^2}}  + e^{-\frac{(x+1.5)^2}{2\theta^2}} + 10^{-8}  \right), \quad \theta = 0.1\,,
\]
we plot the results in Fig. \ref{fig:DLSS1}. On the left, one sees an evolution of $\rho$ towards the equilibrium \eqref{equi_DLSS}; on the right, an exponential convergence of the energy $E(\rho)$ is demonstrated. 
\begin{figure}[h!]
\centering
\includegraphics[width=0.48\textwidth]{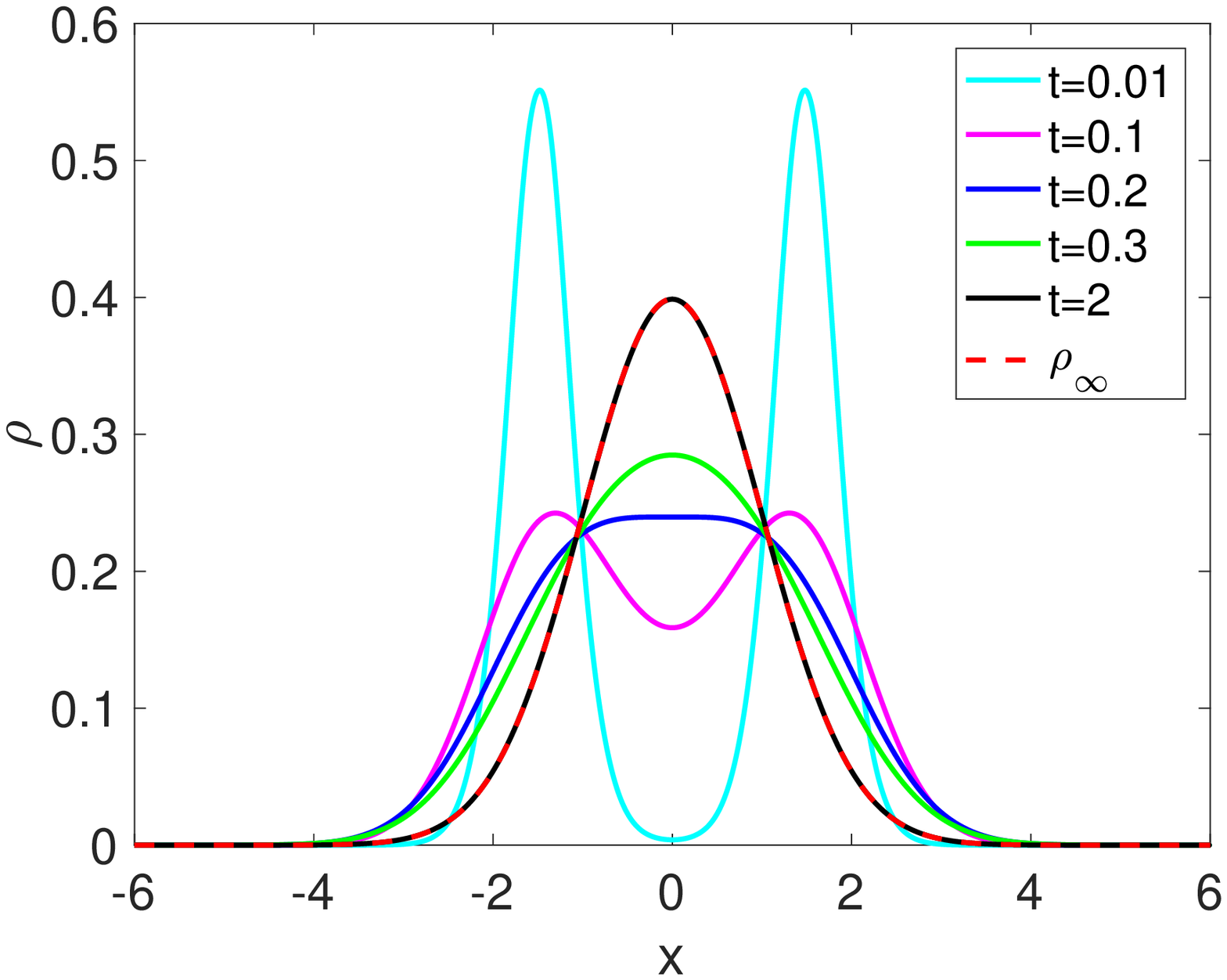}
\includegraphics[width=0.48\textwidth]{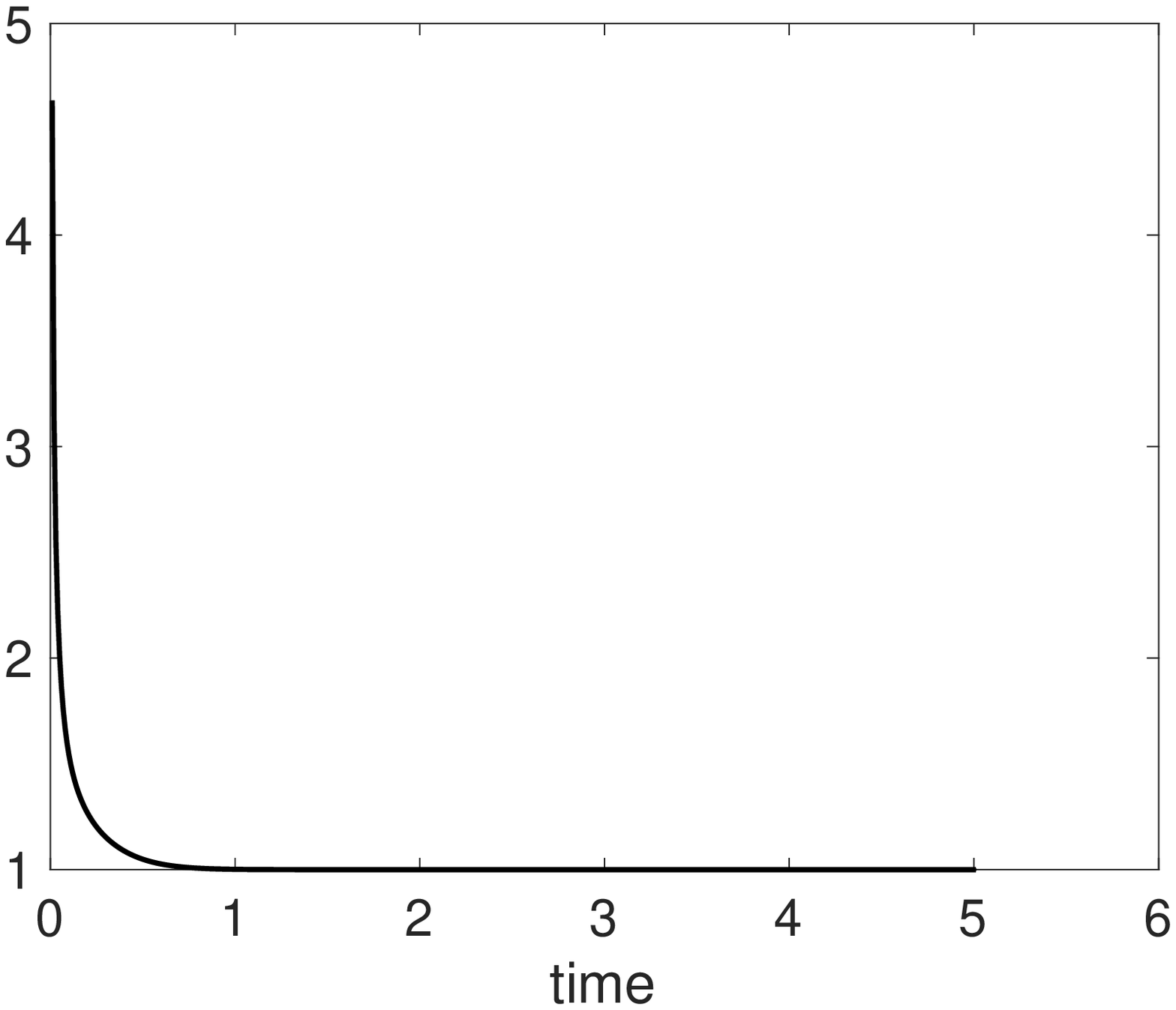}
\caption{Evolution of $\rho(x,t)$ to the one dimensional DLSS equation with $V(x) = \frac{x^2}{2}$. Numerical parameters are $\tau = 0.01$, $\Delta x = 0.01$.  }
\label{fig:DLSS1}
\end{figure}

Likewise, for a double-well potential $V(x) = 10(1-x^2)^2$, with the same initial condition, we collect the results in Fig.\ref{fig:DLSS1D_2}. Unlike the previous case, the steady state here has two bumps. 
\begin{figure}[h!]
\centering
\includegraphics[width=0.48\textwidth]{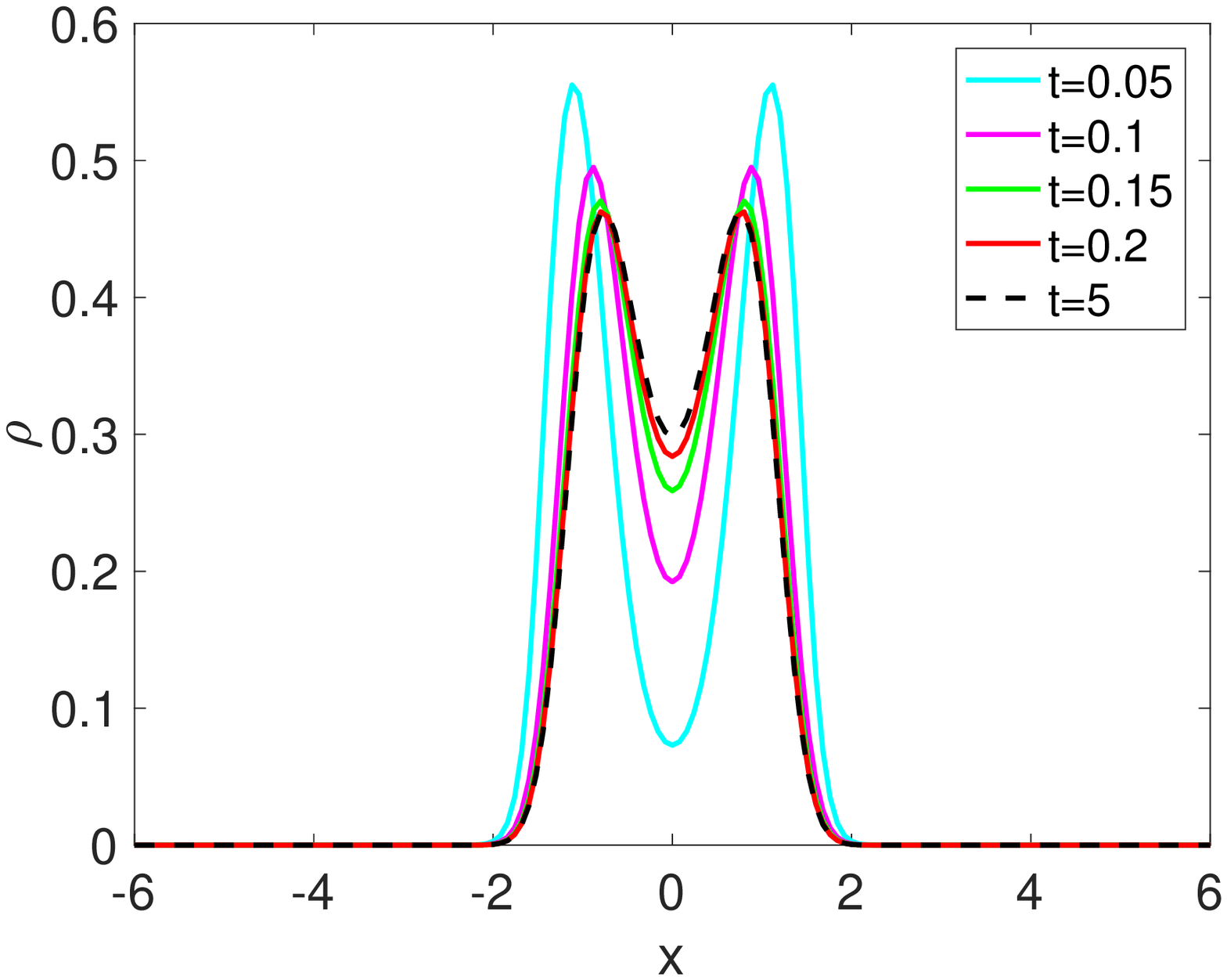}
\includegraphics[width=0.48\textwidth]{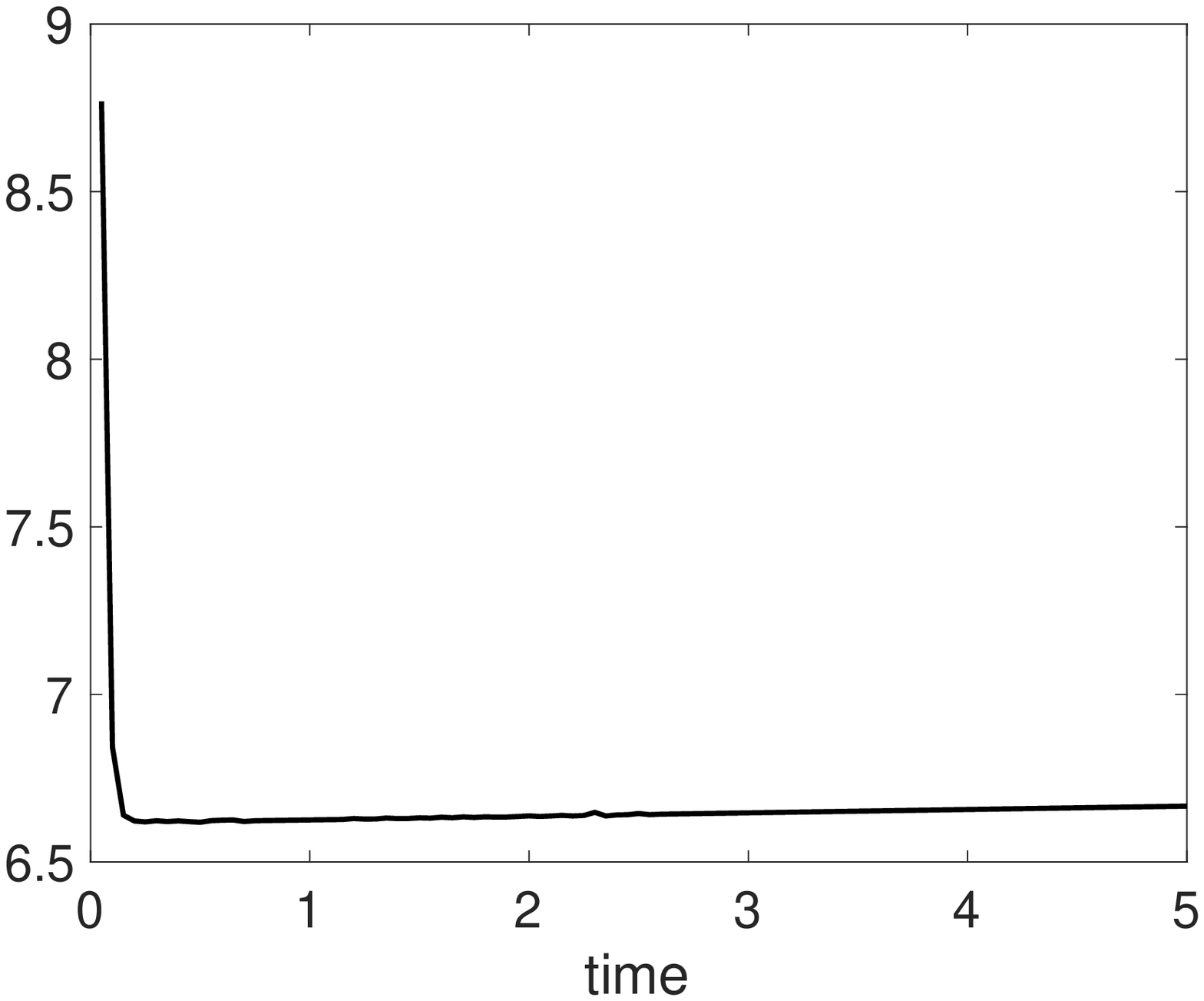}
\caption{Evolution of $\rho(x,t)$ to the one dimensional DLSS equation with $V(x) = 10(1-x^2)^2$. Numerical parameters are $\tau = 0.05$, $\Delta x = 0.08$. }
\label{fig:DLSS1D_2}
\end{figure}


\subsection{2D problem}

\subsubsection{Aggregation equation}
We first consider aggregation equation \eqref{aggeqn1} with attractive-repulsive potentials in two dimensions with interaction kernel 
\begin{equation} \label{W000}
W(x) = \frac{|x|^a}{a} - \frac{|x|^b}{b}, \quad x \in \RR^2, \quad a >b\geq 0\,,
\end{equation}
where $\frac{|x|^0}{0} = \ln (|x|)$. In this case, the repulsion near the origin determines the dimension of the support of the steady state measure, see \cite{BCLR13-2,CDM}. 

In the first example, we choose $a=4$, $b = 2$, and take the initial data to be a Gaussian 
\begin{equation} \label{Gaussian }
\rho(x,0) = \frac{1}{\sqrt{2\pi}\theta} e^{-(x-x^0)^2/\theta^2} + 10^{-5} , \quad x \in \RR^2
\end{equation}
with mean $x^0 =(1.25, 1.25)$ and variance  $\theta=0.2$. Here the steady state concentrates on a Dirac ring with radius 0.5 centered at $\rho^0$, recovering analytical results on the existence of a stable Dirac ring equilibrium \cite{BKSUV15}. We also compare the convergence in the first outer JKO time step of our regularized sequential quadratic programming with the un-regularized primal dual method \cite{CCWW18} in Fig.~\ref{convergence000}, and a much faster convergence in Newton's method is observed.

\begin{figure}[!h]
\includegraphics[width=0.31\textwidth]{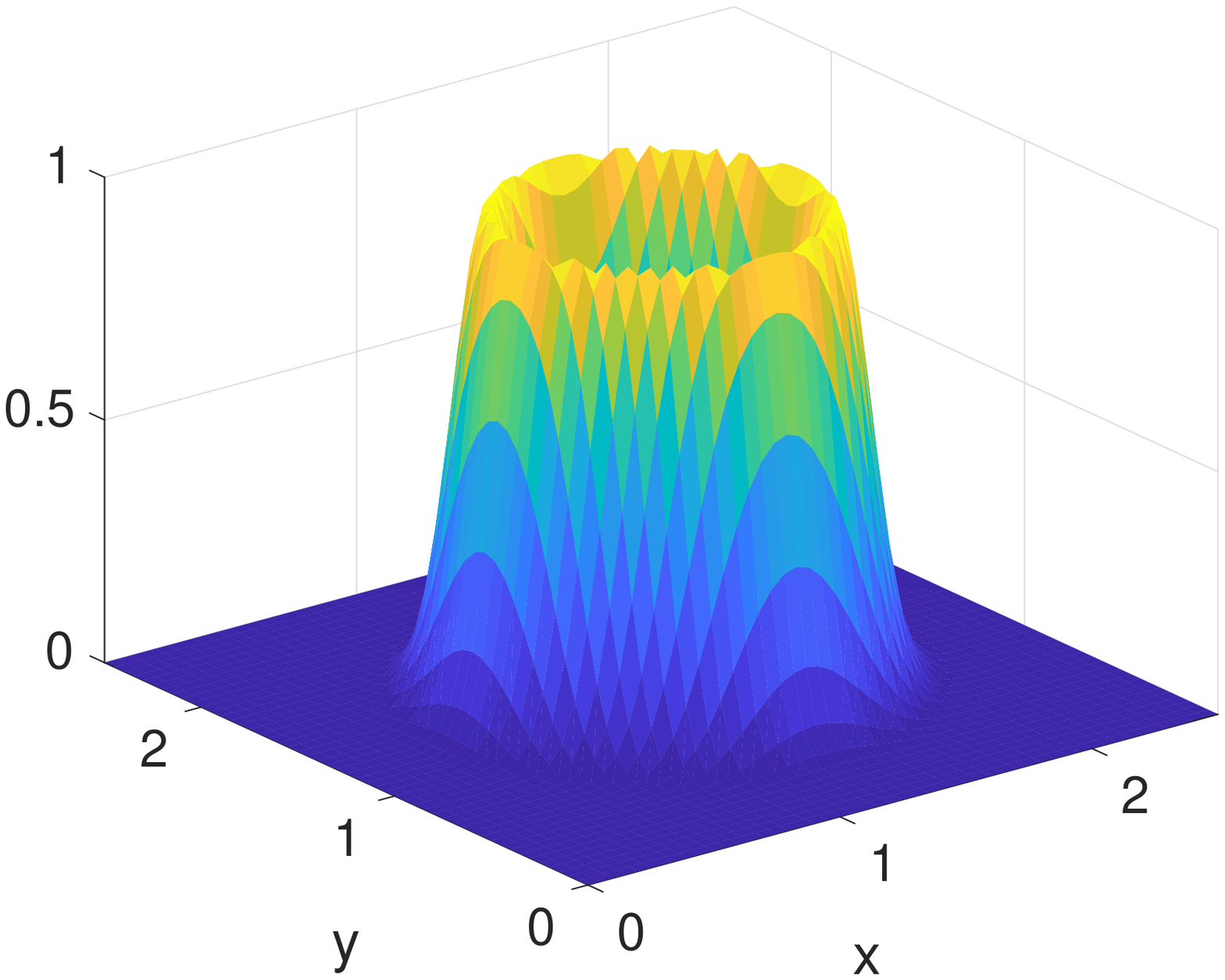}
\includegraphics[width=0.31\textwidth]{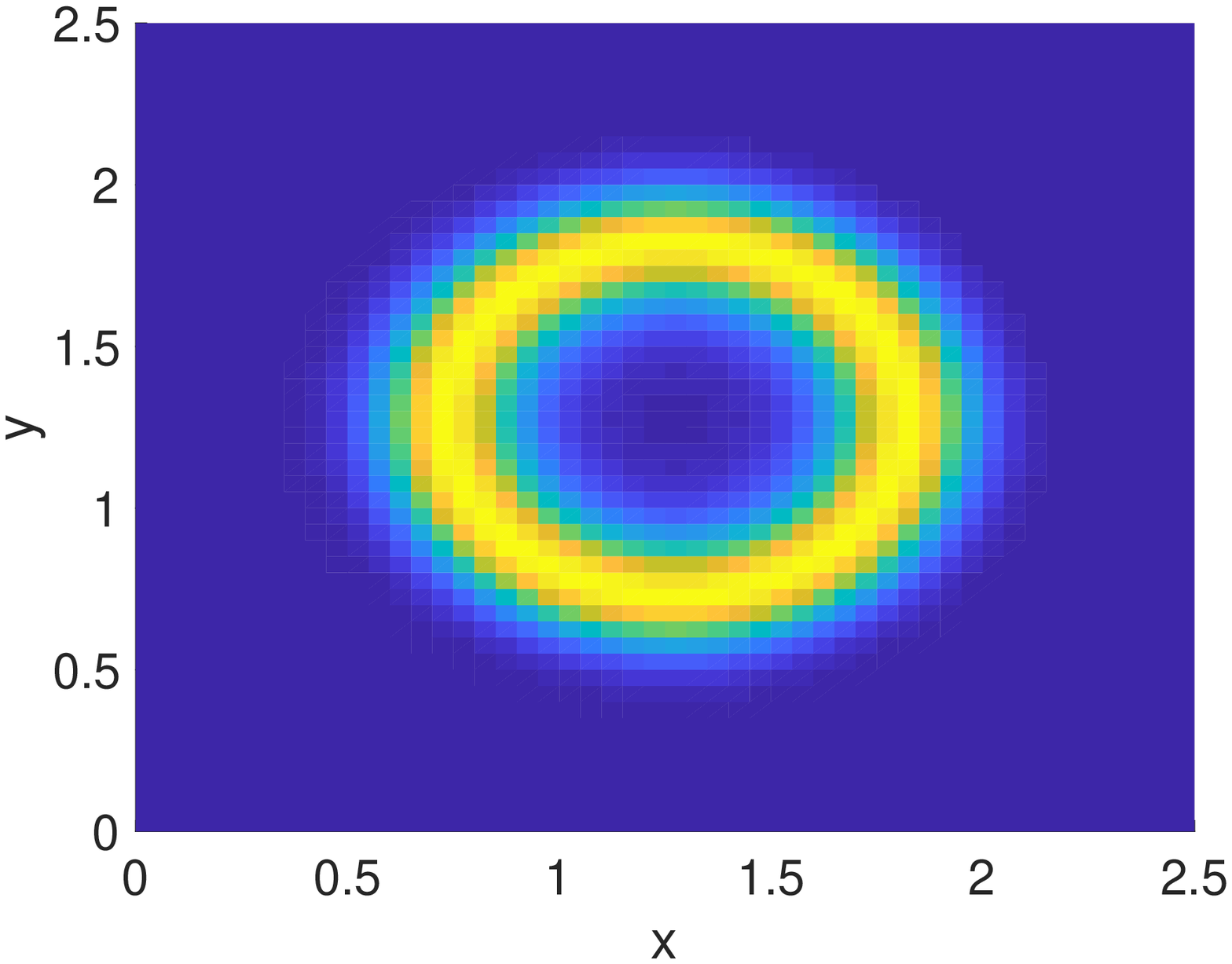}
\includegraphics[width=0.31\textwidth]{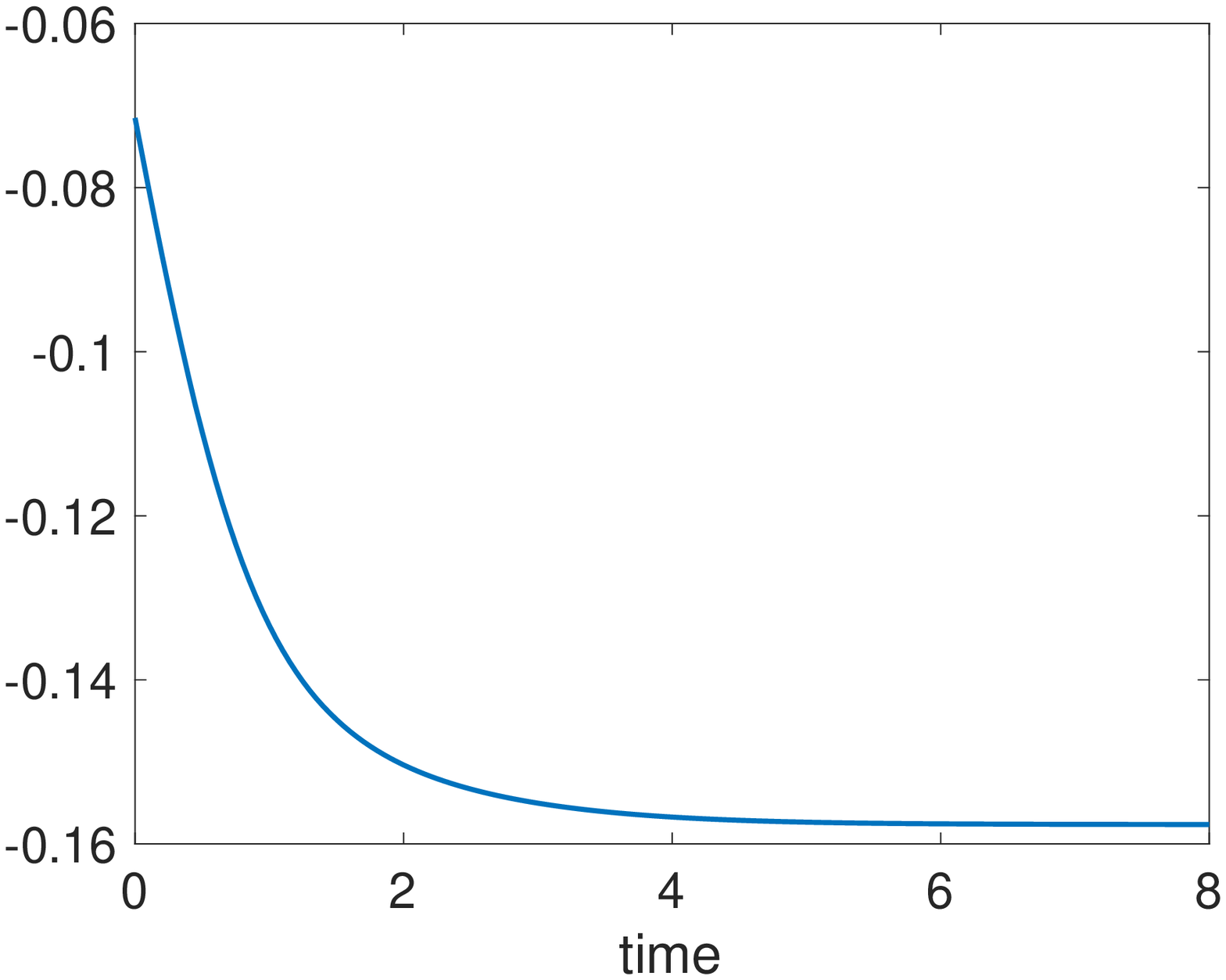}
\caption{We compute the steady state of a solution to the two dimensional aggregation equation \eqref{aggeqn1} with interaction kernel $W(x) = |x|^4/4 - |x|^2/2$, which is a Dirac ring with radius $0.5$ centered at $x^0$. Our computational domain is $(x,y) \in [0,2.5]^2$ and mesh sizes are $\tau = 0.04$, $\Delta x = \Delta y = 0.05$. Regularization constant is $\beta^{-2} \tau^2 = 3.2 \times 10^{-6}$ and Hessian is approximated with $ \tilde{\beta}^{-2} = 80 \beta ^{-2}$. The result showing here are computed at final time $t = 10$. Left: side view of equilibrium. Center: bird's eye view of equilibrium. Right: rate of decay of energy as solution approaches equilibrium.}
\end{figure}

\begin{figure}[!h]
\includegraphics[width=0.45\textwidth]{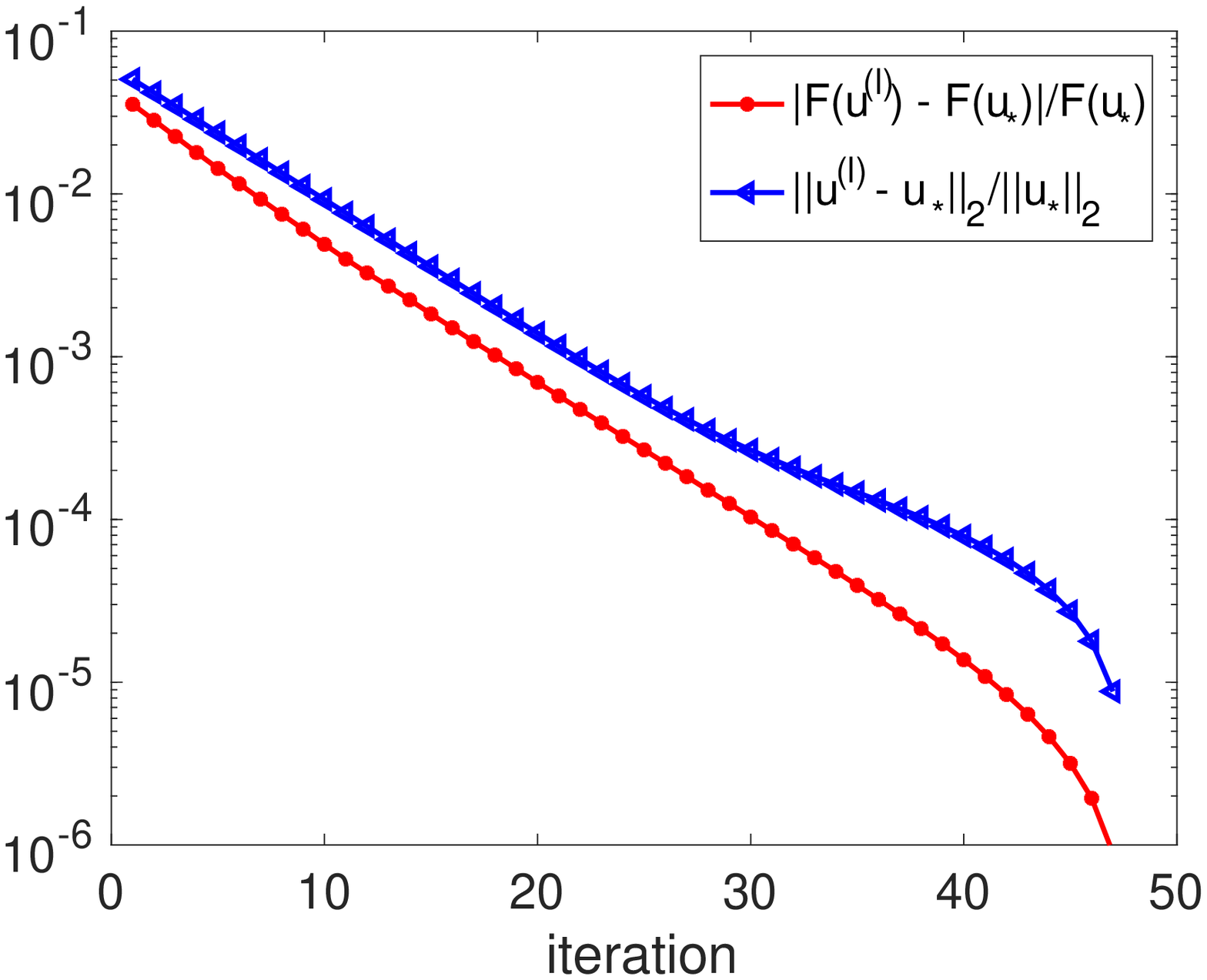}
\includegraphics[width=0.45\textwidth]{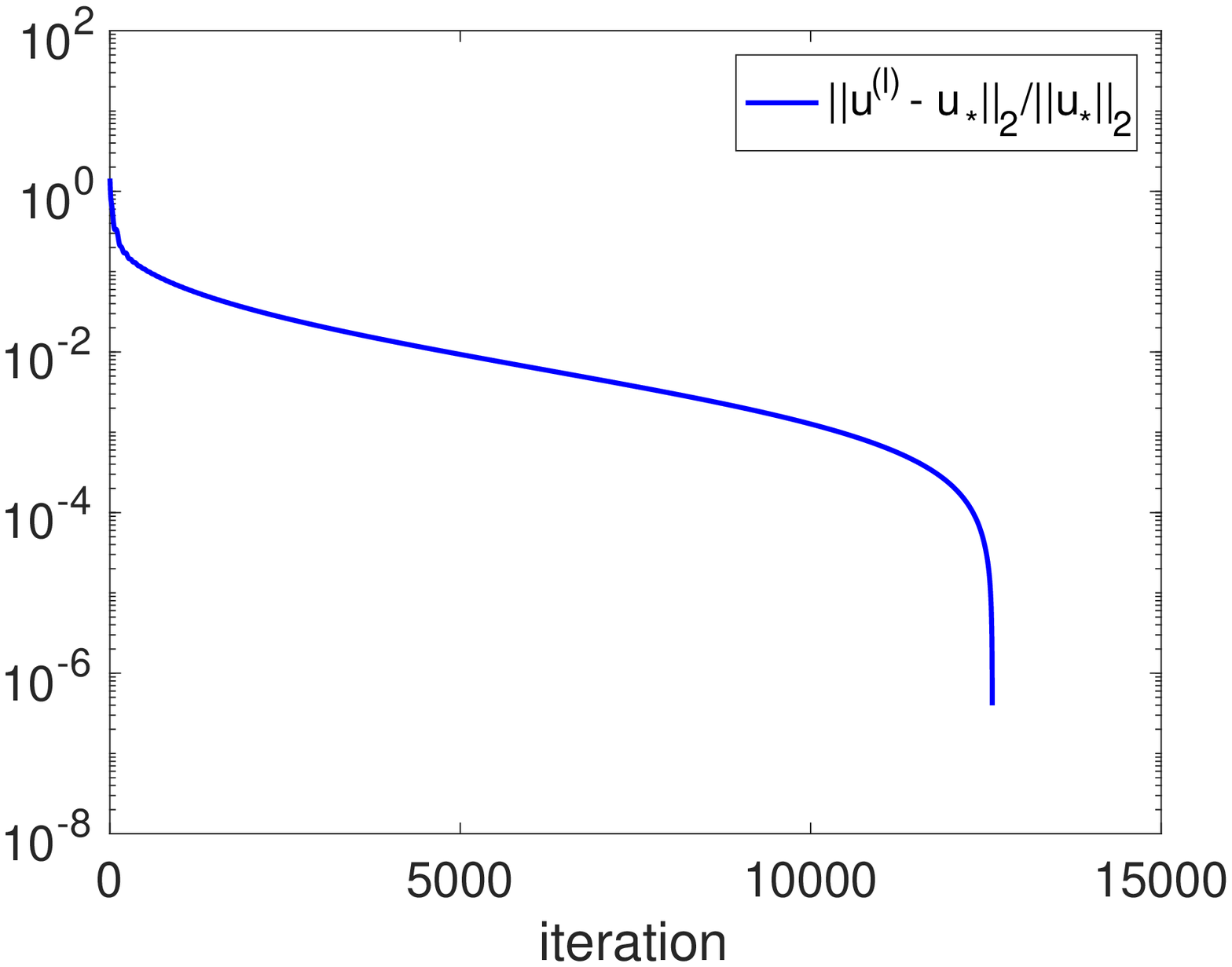}
\caption{Check of convergence in the first time step of computing the 2D aggregation equation with interaction kernel $W(x) = |x|^4/4 - |x|^2/2$. Left: proximal Newton method. Right: primal dual method.}
\label{convergence000}
\end{figure}

In the second example, we consider interaction kernel \eqref{W000} with different parameters: $a = 2$ and $b =0$ and the results are displayed in Fig. \ref{fig:agg2_2d}. We observe that the solution converges to a characteristic function on the disk of radius 1, centered at $x^0$, recovering analytic results on solutions of the aggregation equation with Newtonian repulsion \cite{FHK11, BLL12}. 
\begin{figure}[!h]
\includegraphics[width=0.31\textwidth]{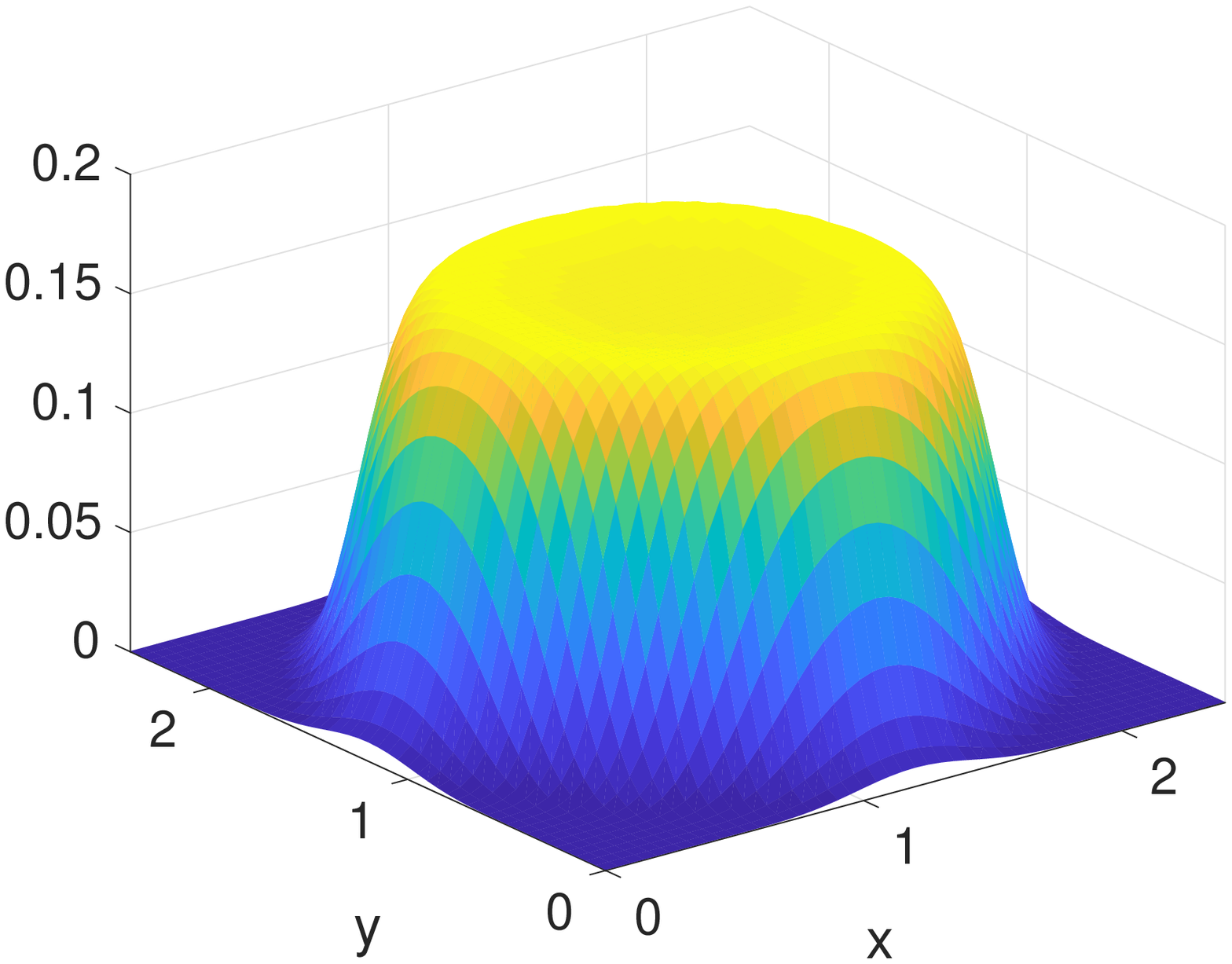}
\includegraphics[width=0.31\textwidth]{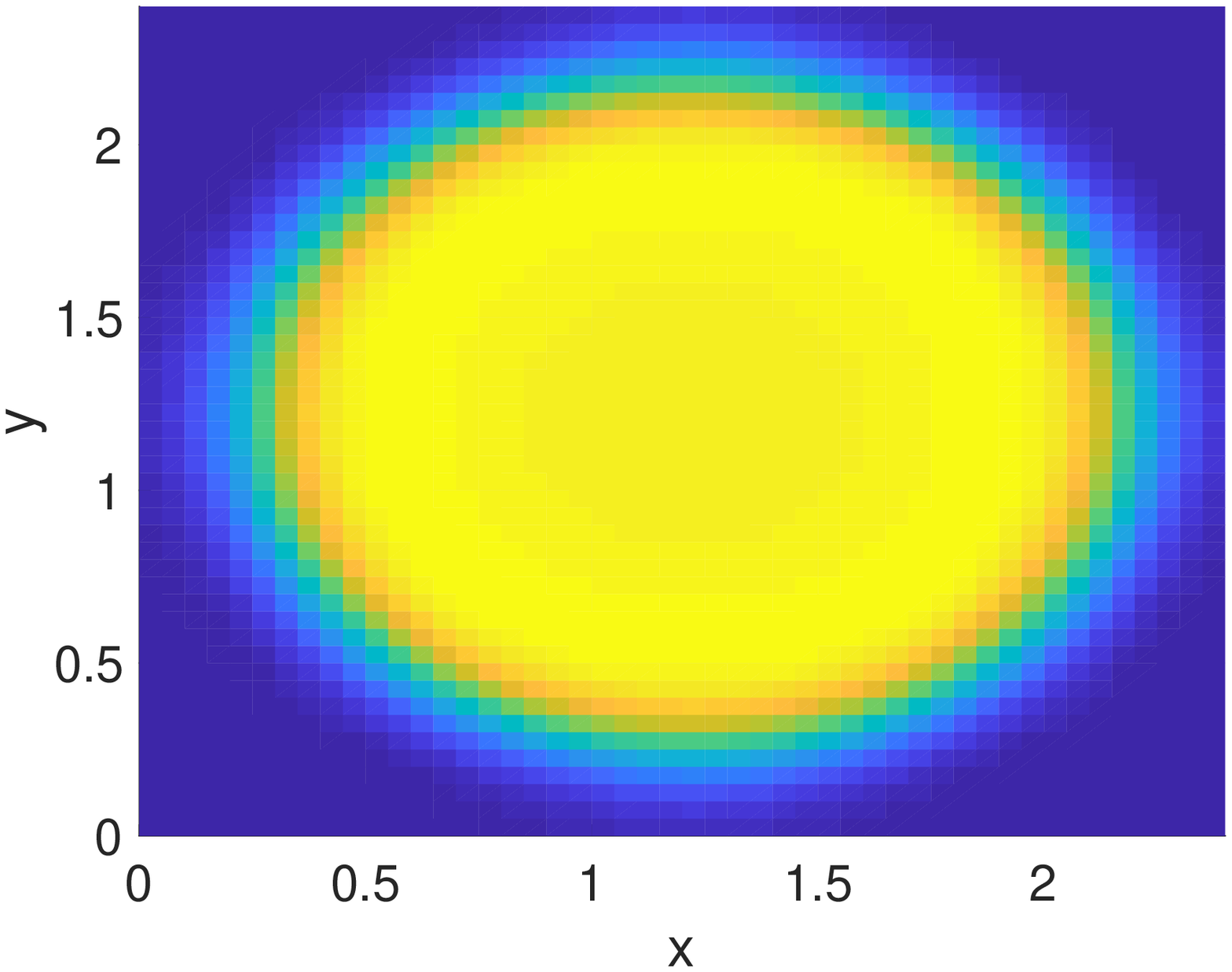}
\includegraphics[width=0.31\textwidth]{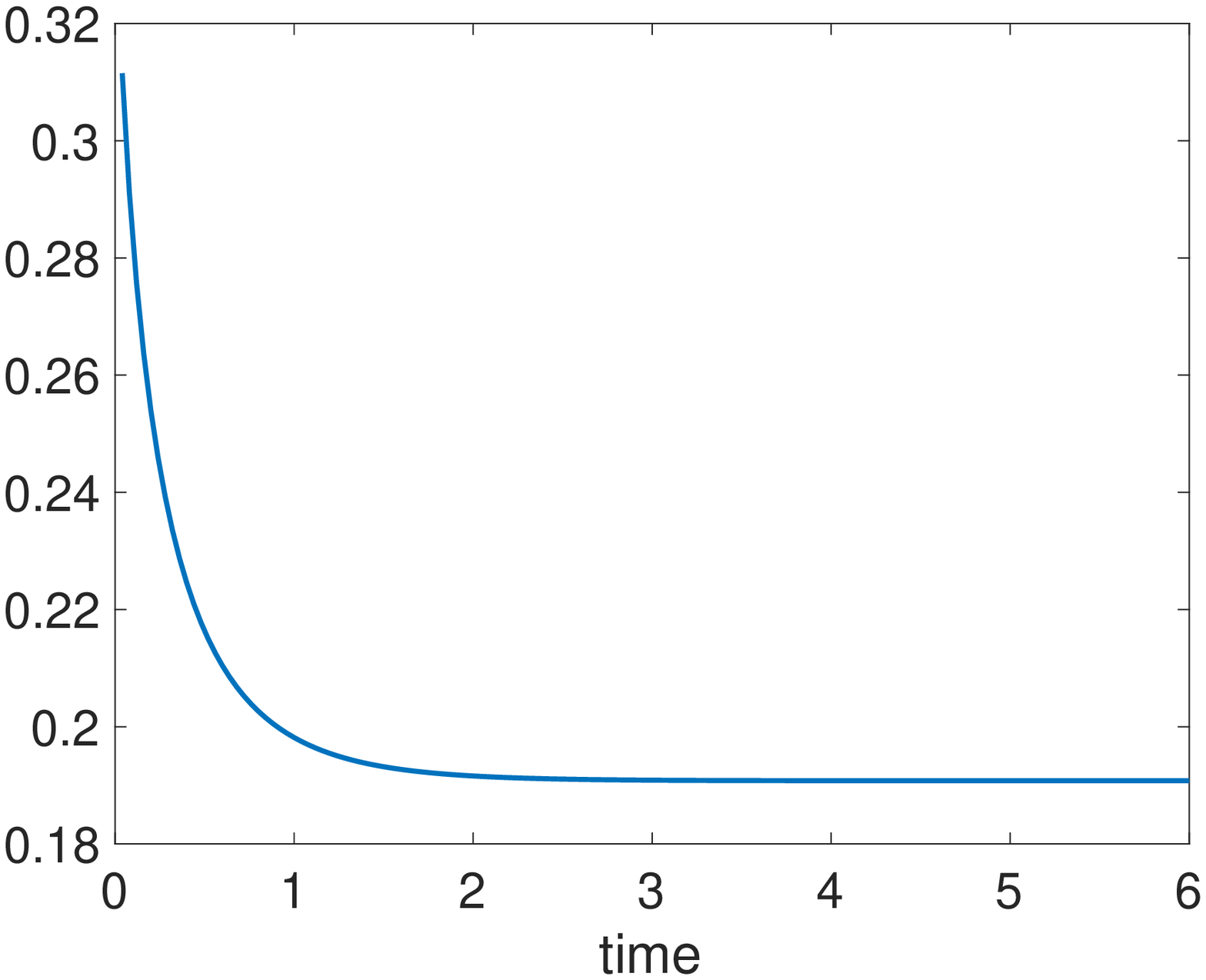}
\caption{We compute the steady state of a solution to the two dimensional aggregation equation \eqref{aggeqn1} with interaction kernel $W(x) = |x|^2/2 - \ln (|x|)$, which is a characteristic function on a disk of radius 1. Our computational domain is $(x,y) \in [0,2.5]^2$ and mesh sizes are $\tau = 0.04$, $\Delta x = \Delta y = 0.05$. Regularization constant is $\beta^{-2} \tau^2 = 3.2 \times 10^{-6}$ and Hessian is approximated with $\tilde{\beta} ^{-2}= 40 \beta^{-2}$. The result showing here are computed at final time $t = 6$. Left: side view of equilibrium. Center: bird's eye view of equilibrium. Right: rate of decay of energy as solution approaches equilibrium.}
\label{fig:agg2_2d}
\end{figure}

\subsubsection{Aggregation drift equation}
We compute solutions of aggregation-drift equations
\[
\partial_t \rho = \nabla \cdot (\rho \nabla W \ast \rho) + \nabla \cdot (\rho \nabla V) ,
\]
where $W(x) = \frac{|x|^2}{2} - \ln(|x|)$ and $V(x) = - \frac{1}{4} \ln (|x|)$. As shown in the analytical results \cite{CK14, CHM14}, the steady state is a characteristic function on a torus, with inner and outer radius given by
$R_1 = \frac{1}{2}$, $R_2 = \sqrt{\frac{5}{4}}$. The initial condition consists of five Gaussians, which is non-radially symmetric. The evolution of $\rho$ towards equilibrium along with the energy decay in time are displayed in Fig.~\ref{fig:aggdrift2_2D}.

\begin{figure}[!h]
\includegraphics[width=0.32\textwidth]{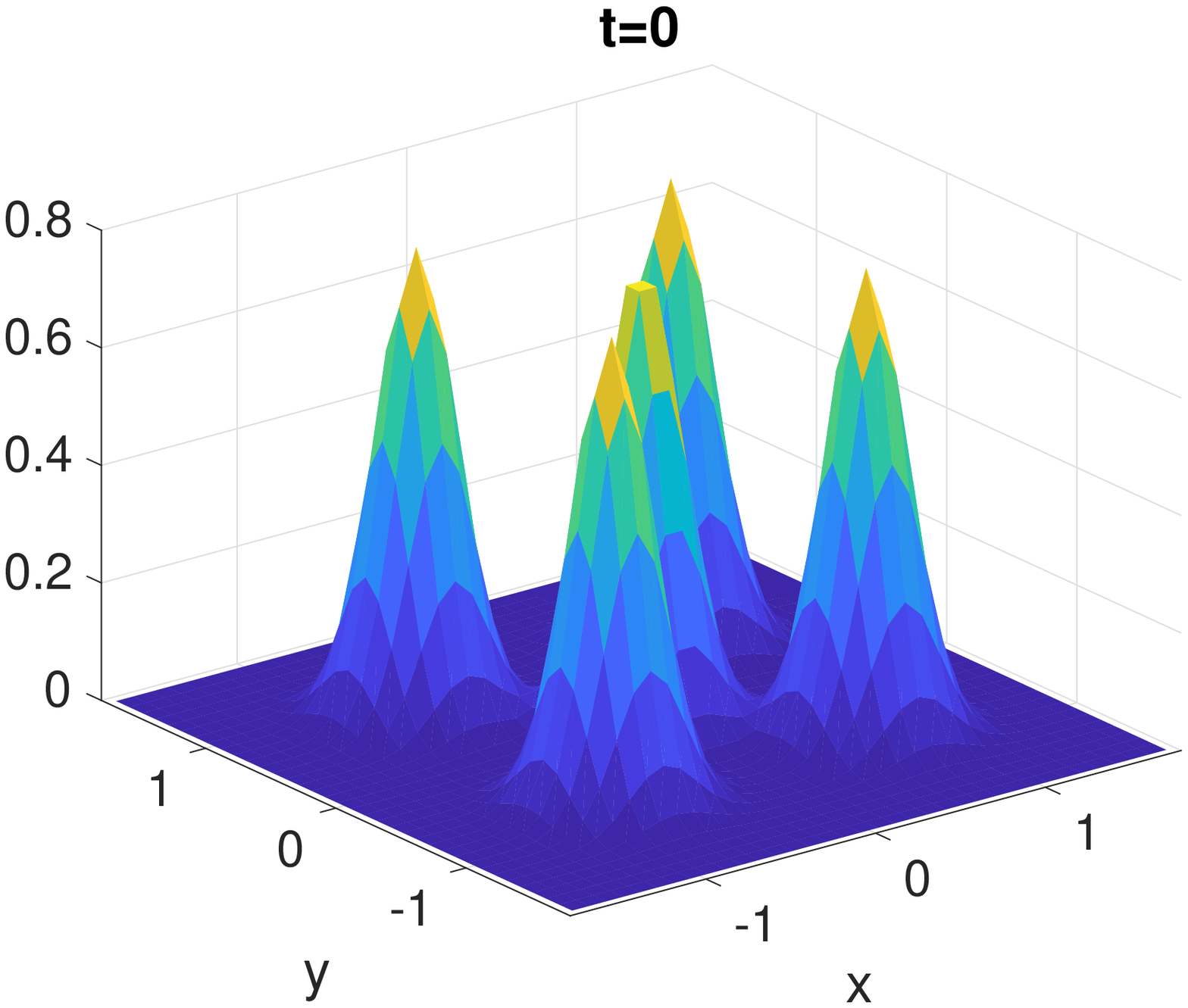}
\includegraphics[width=0.32\textwidth]{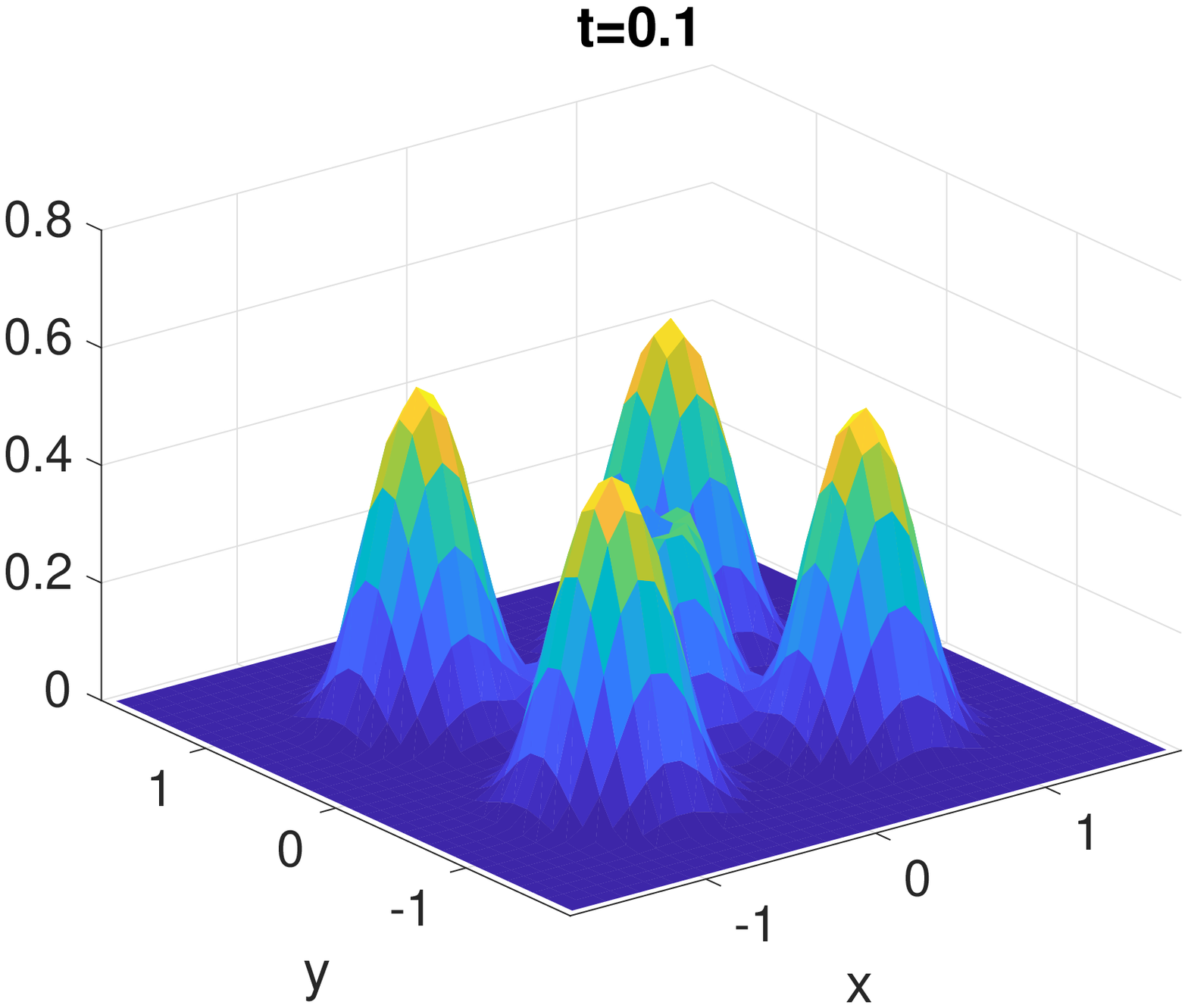}
\includegraphics[width=0.32\textwidth]{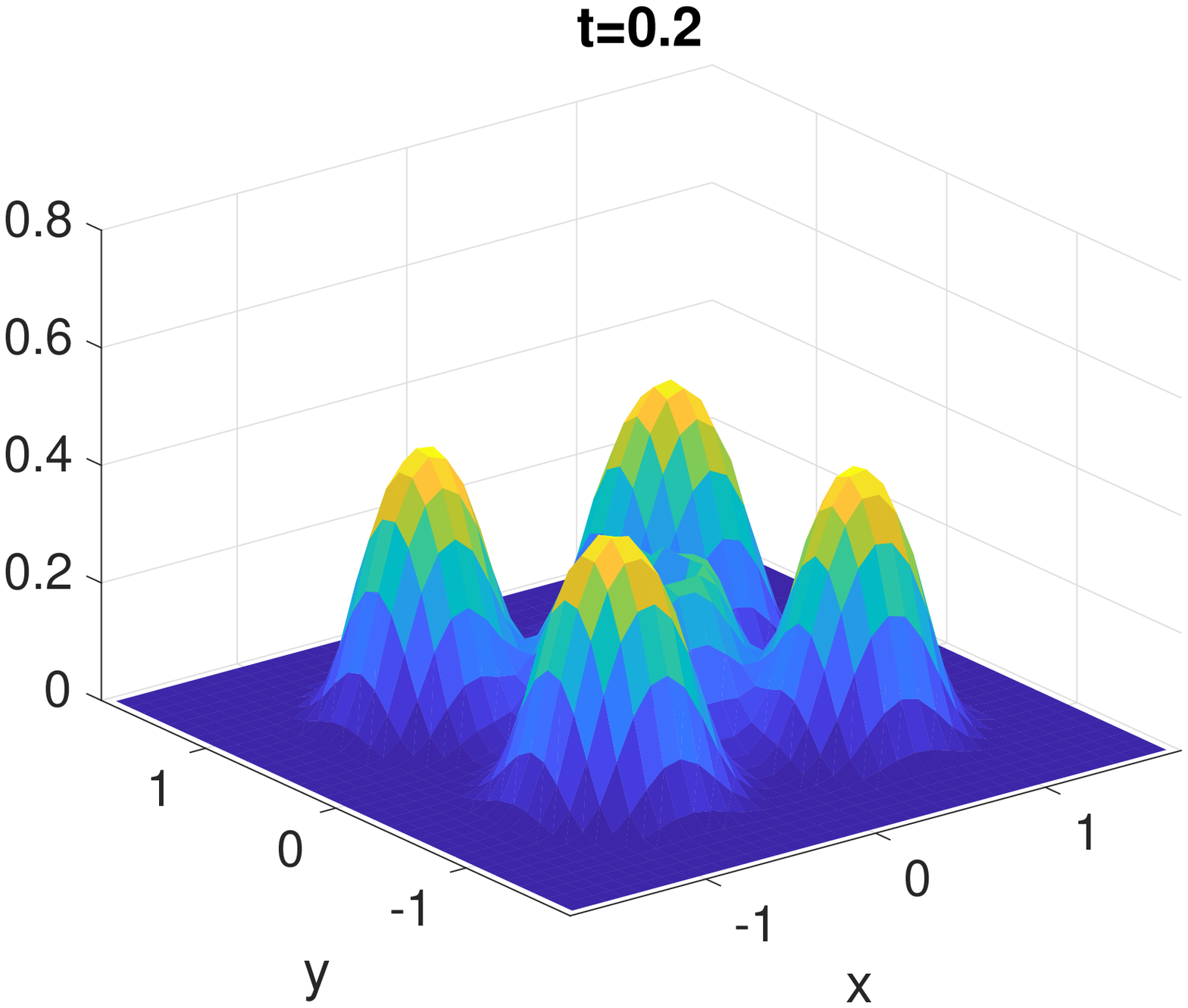}
\\
\includegraphics[width=0.32\textwidth]{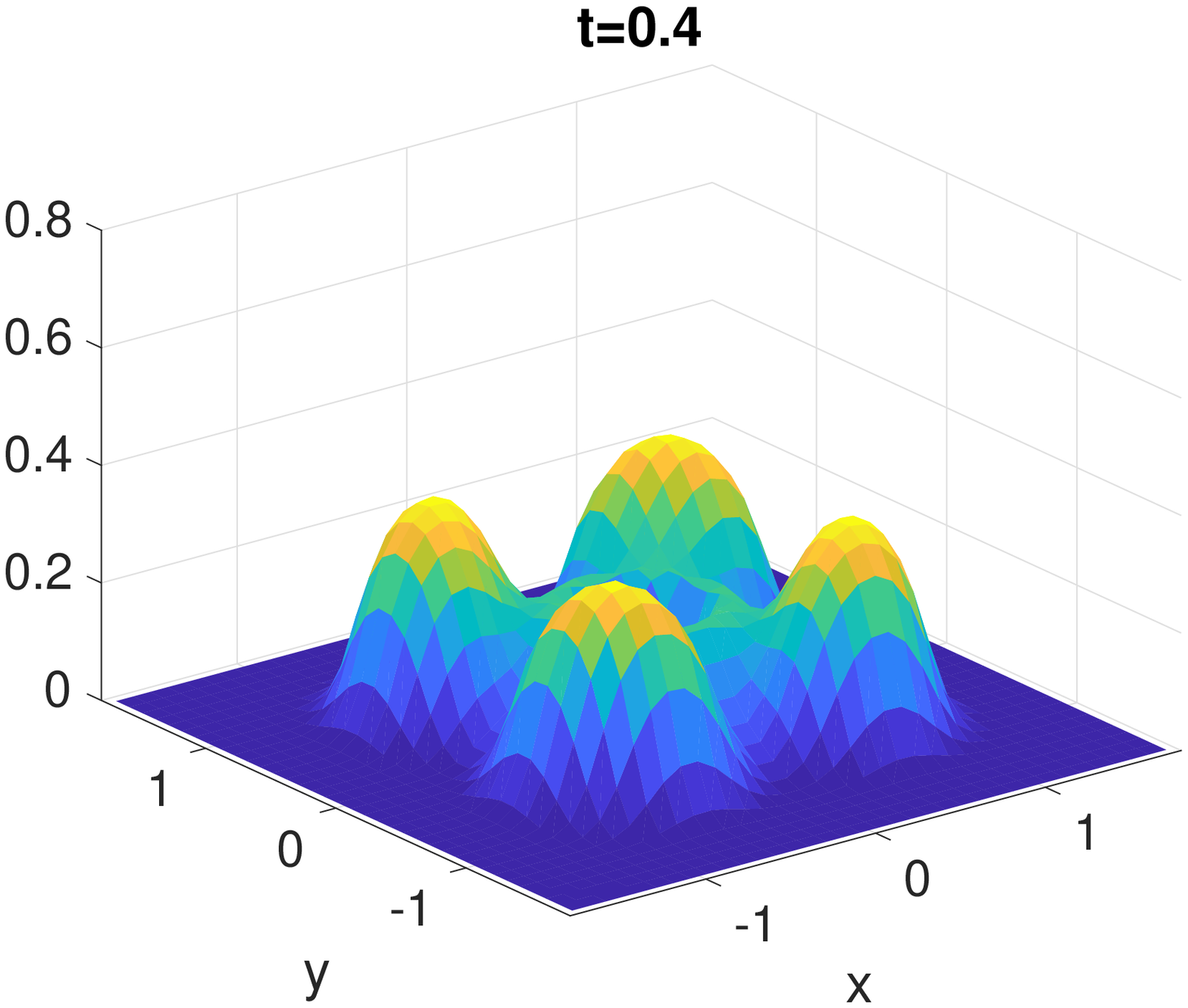}
\includegraphics[width=0.32\textwidth]{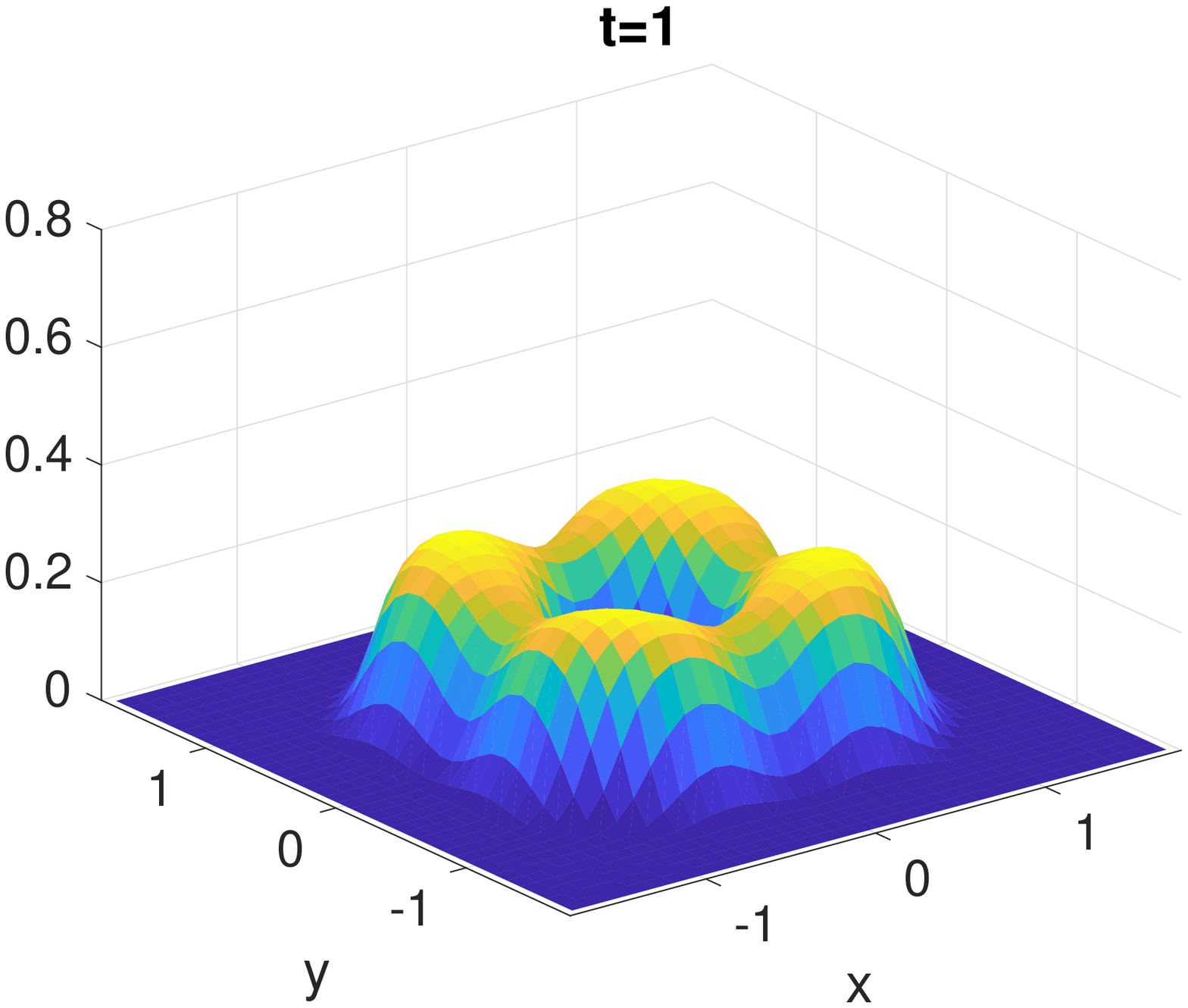}
\includegraphics[width=0.32\textwidth]{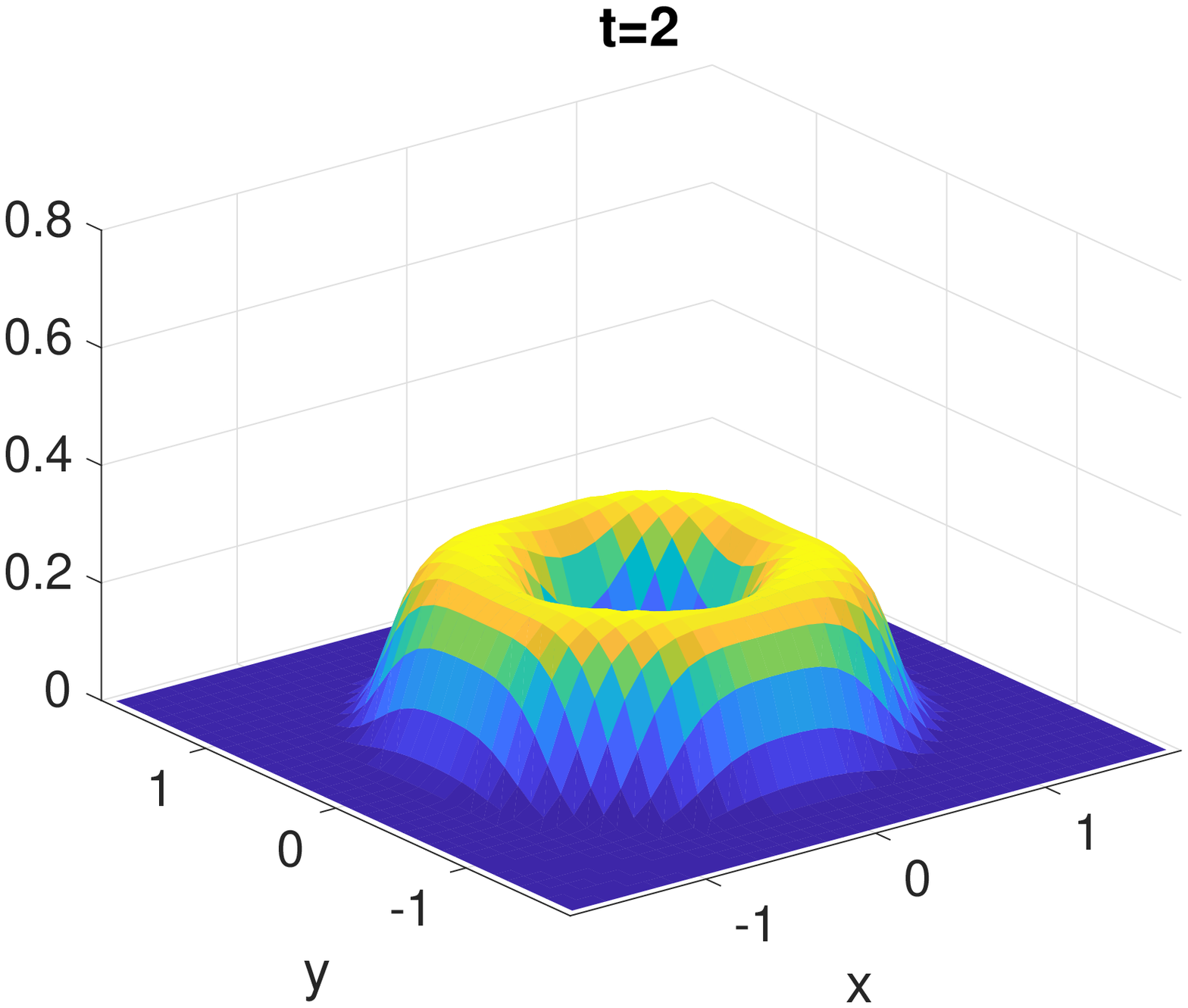}
\\
\includegraphics[width=0.32\textwidth]{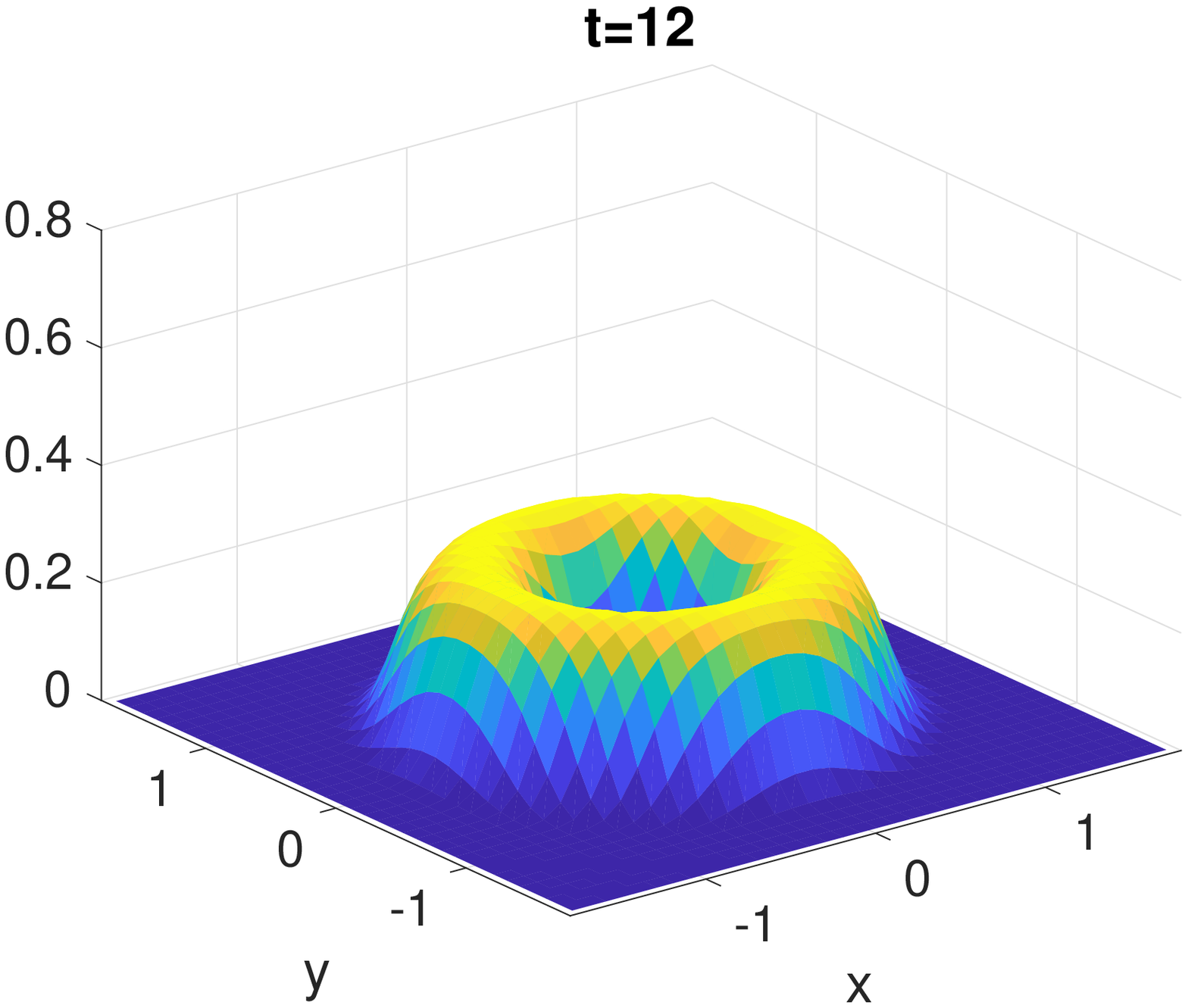}
\includegraphics[width=0.32\textwidth]{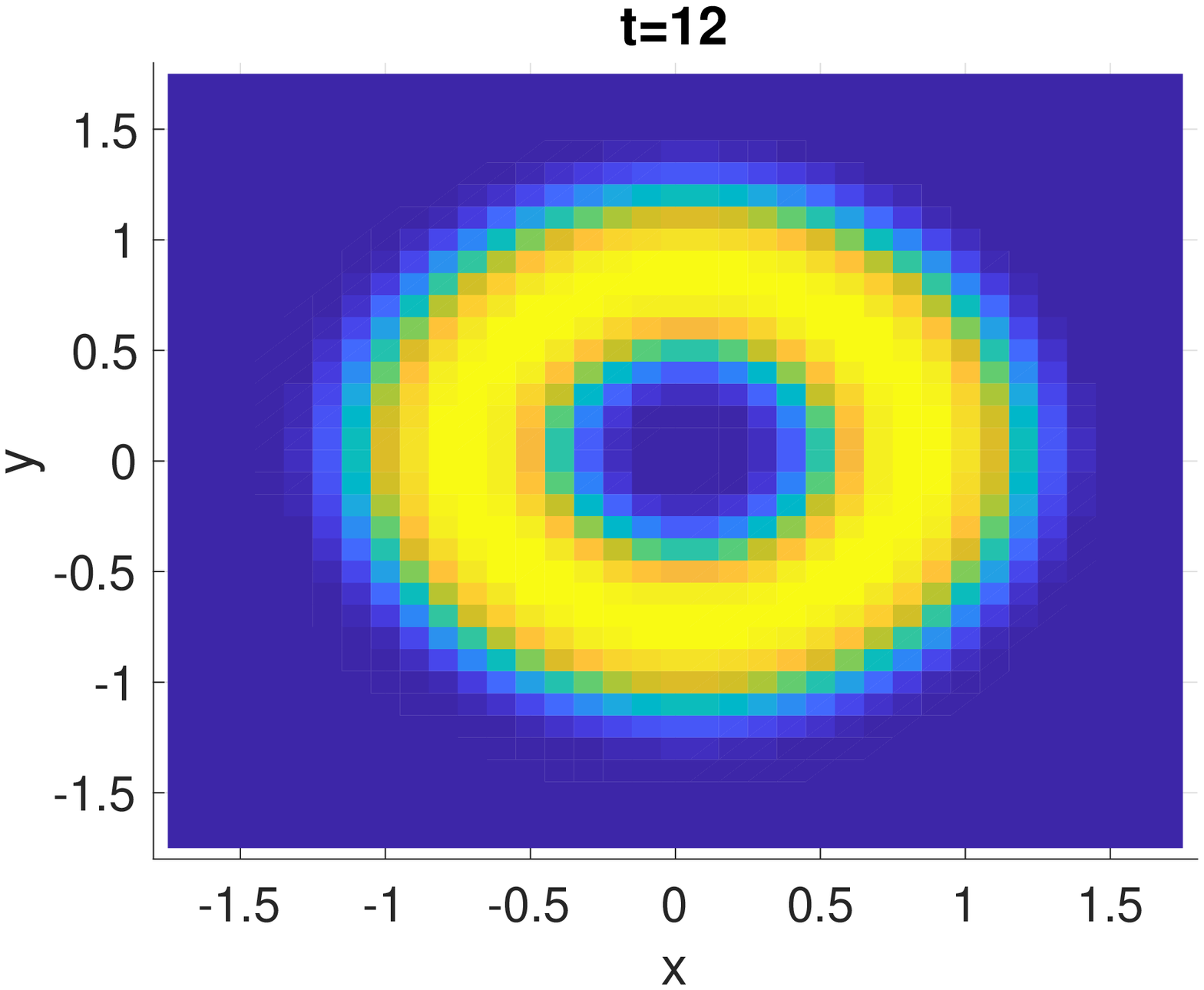}
\includegraphics[width=0.32\textwidth]{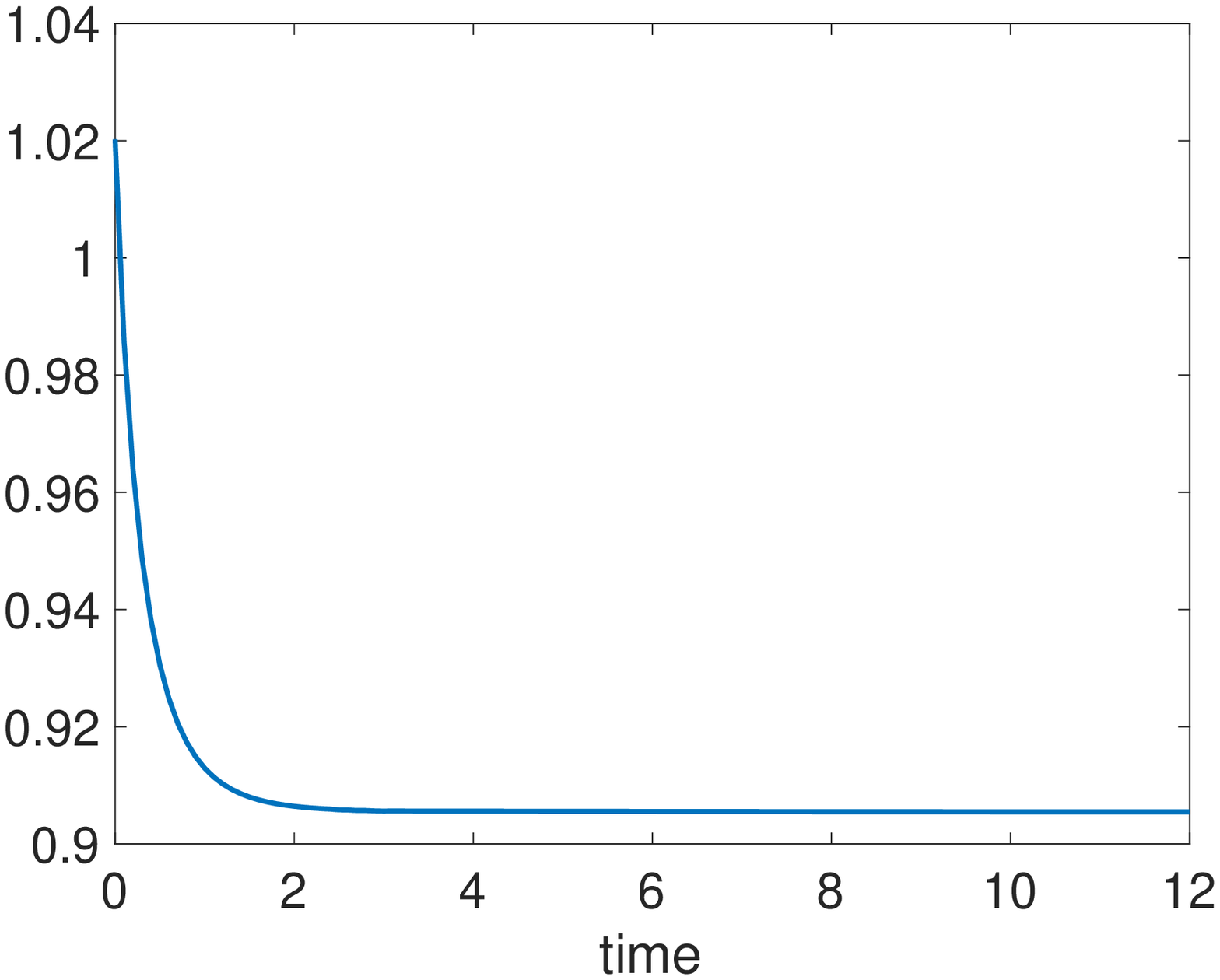}
\caption{The evolution of 2D aggregation drift equation with $W(x) = \frac{|x|^2}{2} - \ln(|x|)$ and $V(x) = - \frac{1}{4} \ln (|x|)$. The steady state is a milling profile with inner and outer radius $R_1 = \frac{1}{2}$, $R_2 = \sqrt{\frac{5}{4}}$. Our computational domain is $[-1.8,1.8]^2$, and mesh sizes are $\Delta x = \Delta y = 0.1$, $\tau = 0.1$. Regularization parameter is $\beta \tau^2= 1.25\times 10^{-5}$, and $\tilde{\beta} = 80 \beta$ is used in approximating the Hessian.}
\label{fig:aggdrift2_2D}
\end{figure}

\subsubsection{Aggregation diffusion equation}
Consider the aggregation diffusion equations
\begin{equation} \label{aggdiffeqn}
\partial_t \rho = \nabla \cdot (\rho \nabla W*\rho) + \nu \Delta \rho^m , \quad W: \Rd \to \R, \quad m \geq 1\,.
\end{equation}
When the interaction kernel $W$ is attractive, the competition between the nonlocal aggregation $ \nabla \cdot (\rho \nabla W*\rho)$ and nonlinear diffusion $\nu \Delta \rho^m$ causes solutions to behave differently in various regimes---either finite time blow up or globally exist in time, see the survey \cite{CCY}. In Fig.~\ref{fig:AggDiff}, we take $W(x) = -\frac{e^{-|x|^2}}{\pi}$, $m=3$, $\nu = 0.1$. Computational domain is chosen as $[-3, 3]^2$, initial data is $\rho(0,x,y) = \chi_{|x| \leq 2.5, |y| \leq 2.5}$.

\begin{figure}[!h]
\includegraphics[width=0.32\textwidth]{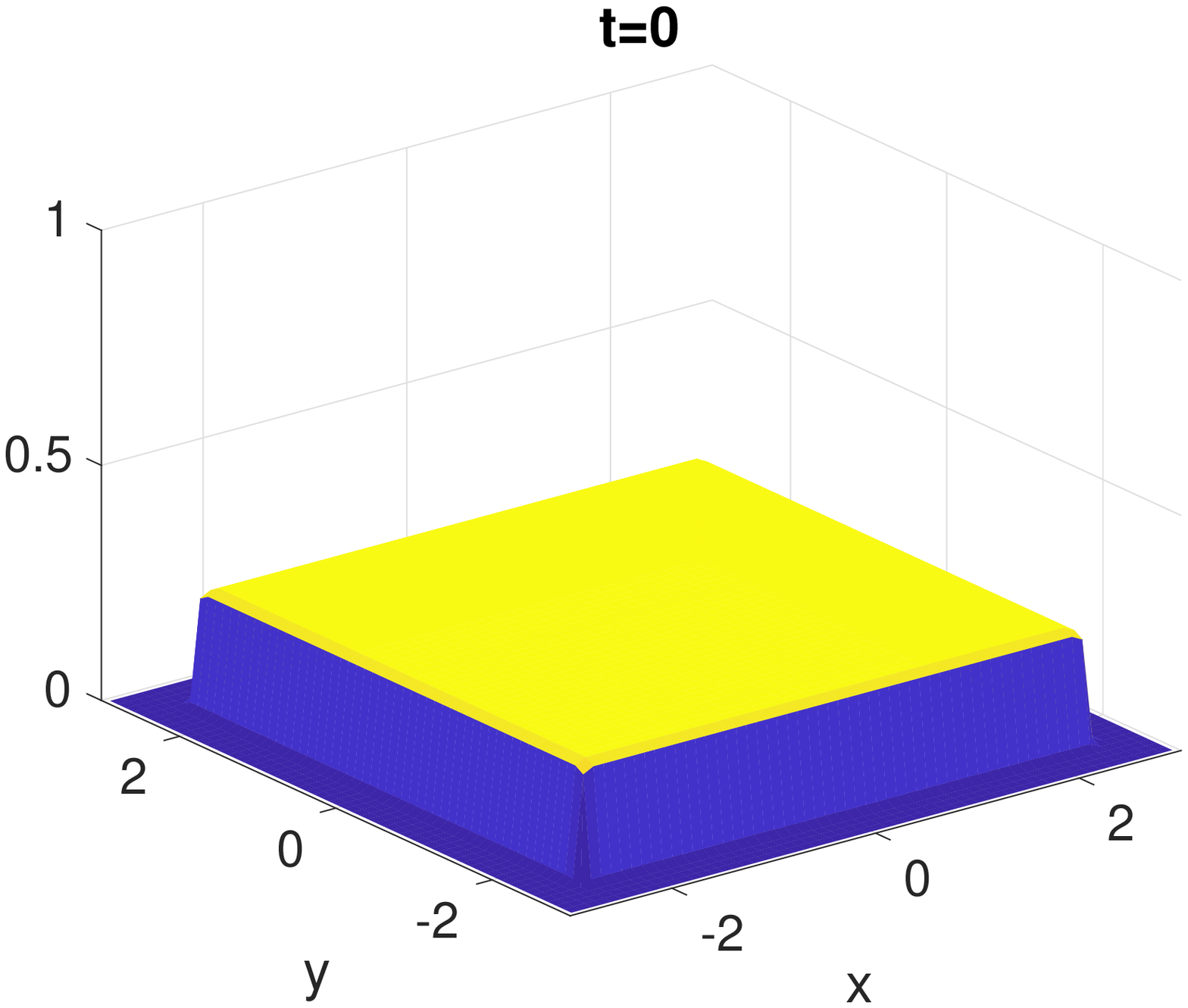}
\includegraphics[width=0.32\textwidth]{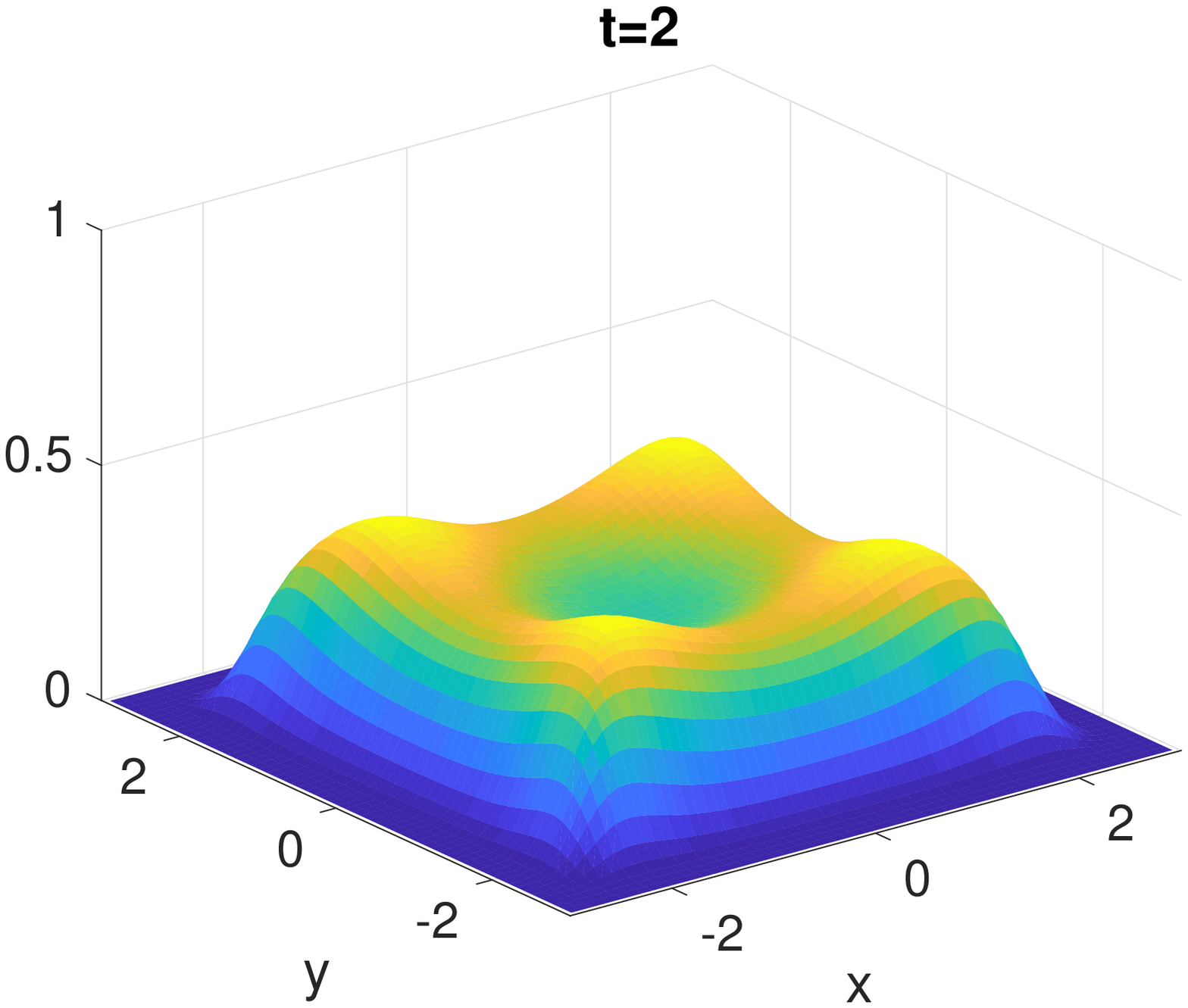}
\includegraphics[width=0.32\textwidth]{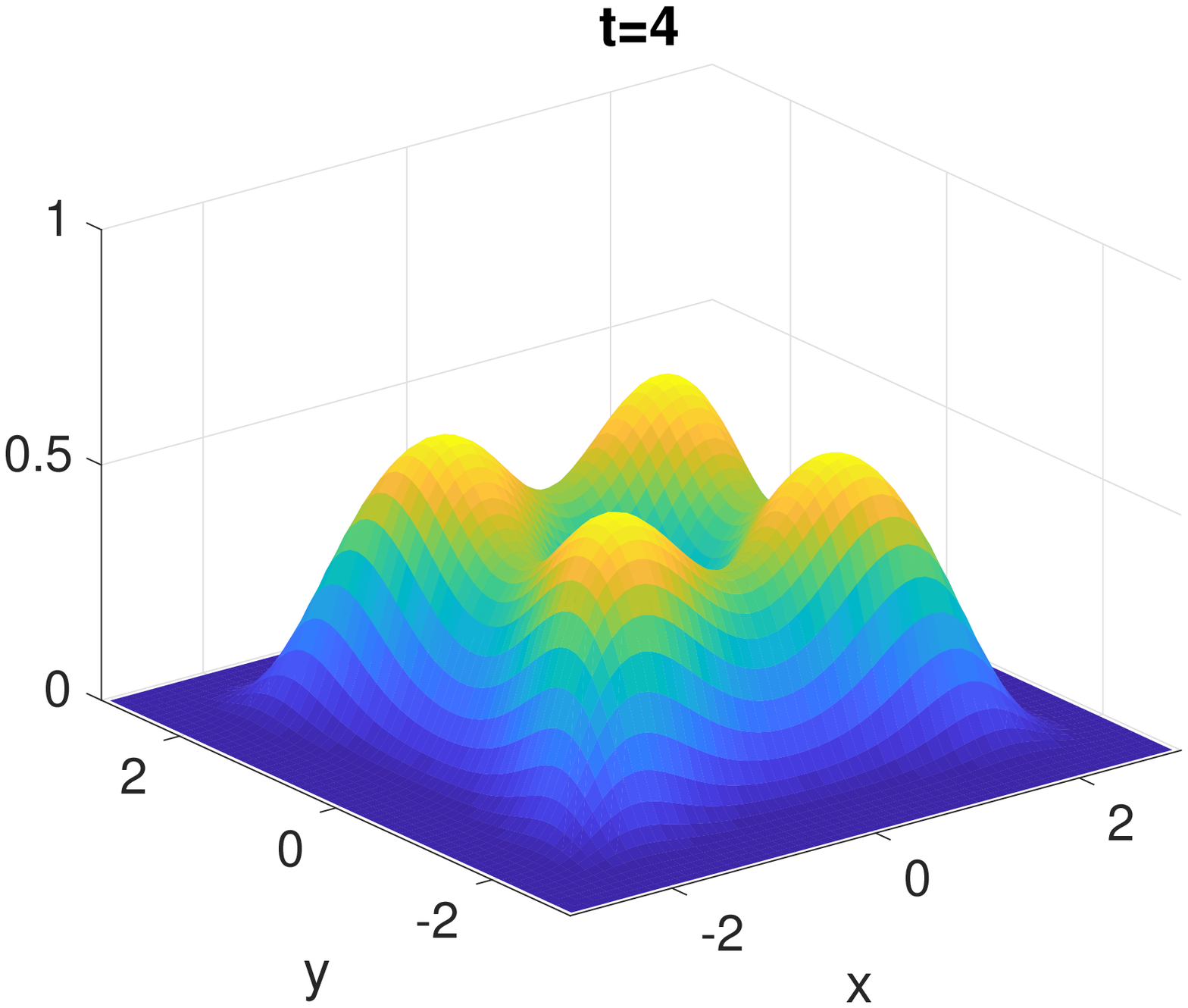}
\\
\includegraphics[width=0.32\textwidth]{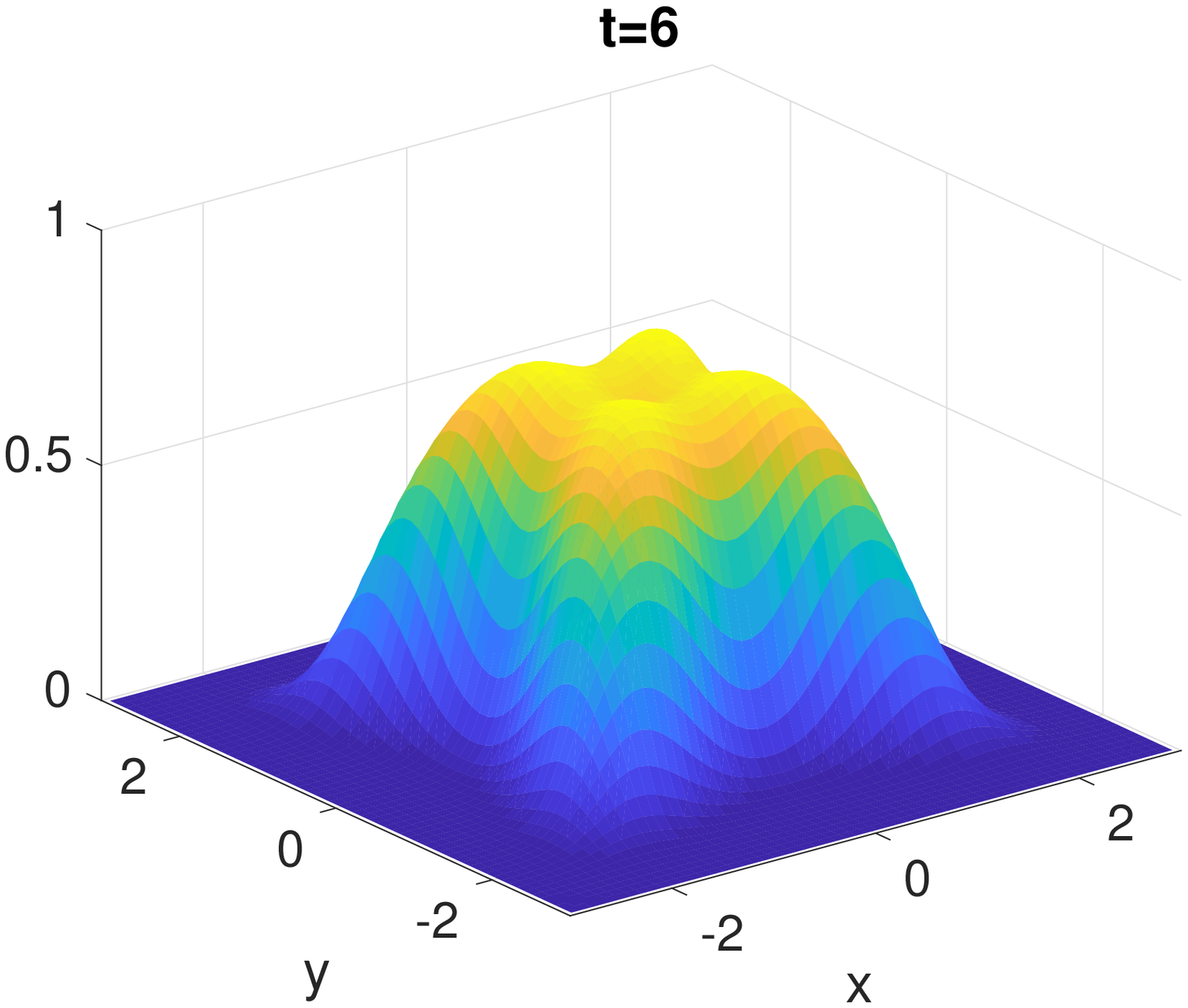}
\includegraphics[width=0.32\textwidth]{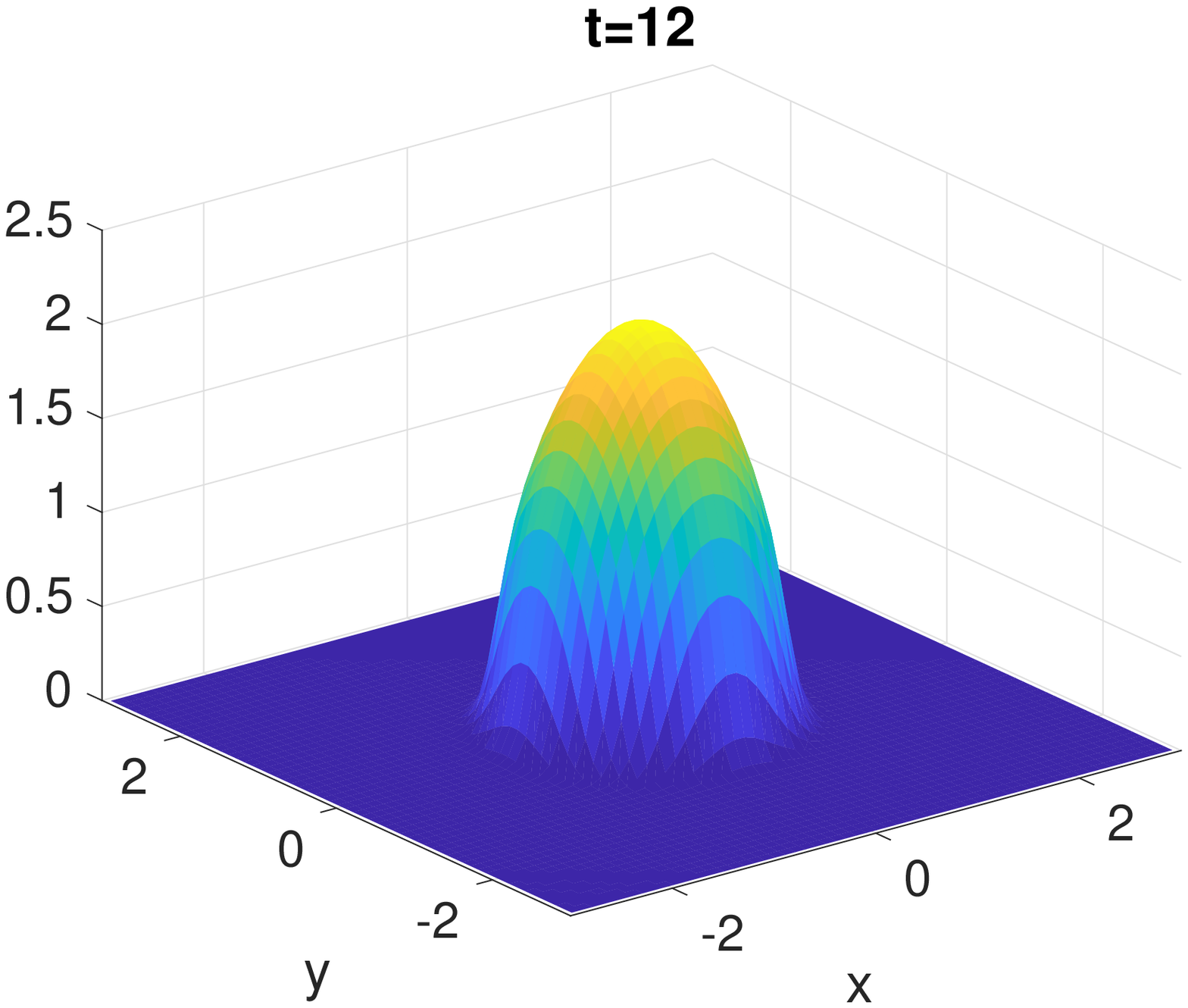}
\includegraphics[width=0.32\textwidth]{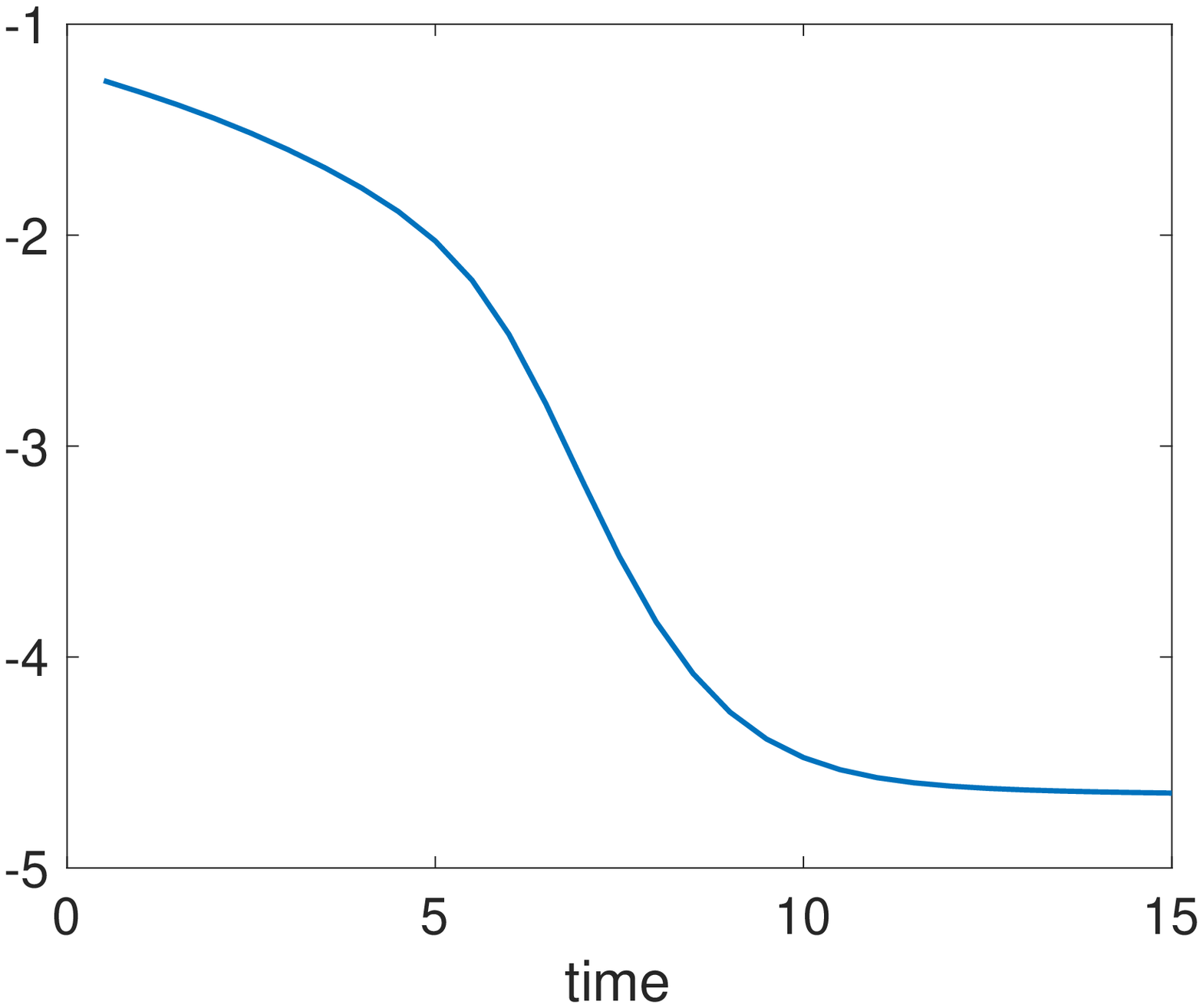}
\caption{The evolution of 2D aggregation diffusion equation with $W(x) = -\frac{e^{-|x|^2}}{\pi}$ and $U(\rho) = 0.05 \rho^3$. The computational domain is $[-3,3]^2$, and meshes are $\Delta x = \Delta y = 0.1$, $\tau = 0.5$. Regularization parameter is $\beta^{-2} \tau^2= 3.125\times 10^{-4}$, and $\tilde{\beta}^{-2} = 40 \beta^{-2}$ is used in approximating the Hessian. }
\label{fig:AggDiff}
\end{figure}

\subsubsection{The DLSS model}
We close the section by computing a two dimensional DLSS equation
\begin{equation*}
\partial_t\rho=\nabla\cdot\left[\rho \nabla \left( \half \delta \mathcal I (\rho )+V(x) \right)\right]
\end{equation*}
with $V(x) = \frac{|x|^2}{2}$ and initial condition consisting of four Gaussians. In this case, we do not need a regularization and Hessian is computed exactly. As seen in Fig. \ref{fig:dlss_2d}, the density $\rho$ converges to the equilibrium $\rho_\infty = \frac{1}{\sqrt{2\pi}} e^{-\frac{x^2}{2}}$ very rapidly. In the bottom center plot, we also compare slice of the steady state computed via our method with the exact equilibrium, and observe a good match. 

\begin{figure}[!h]
\includegraphics[width=0.32\textwidth]{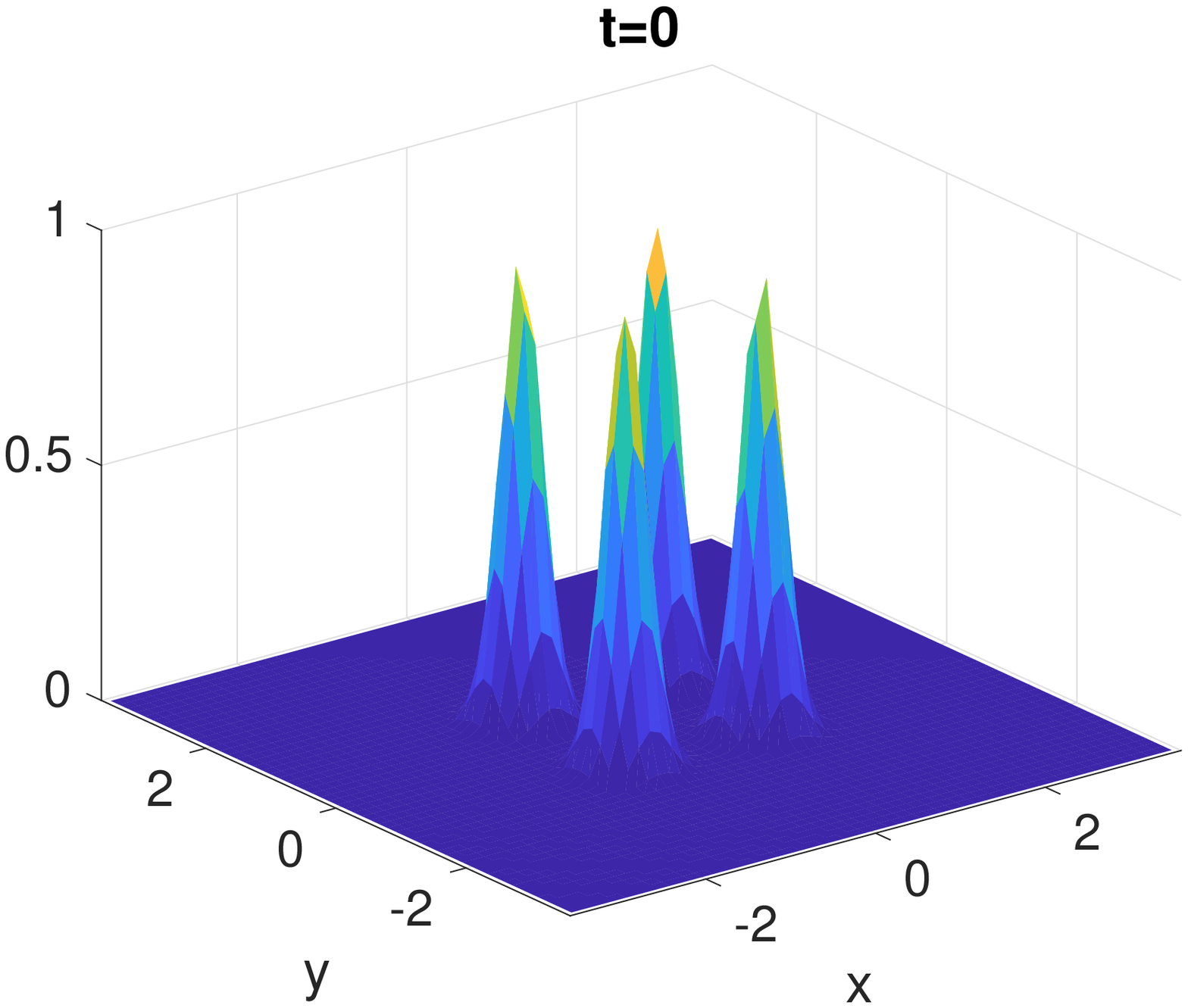}
\includegraphics[width=0.32\textwidth]{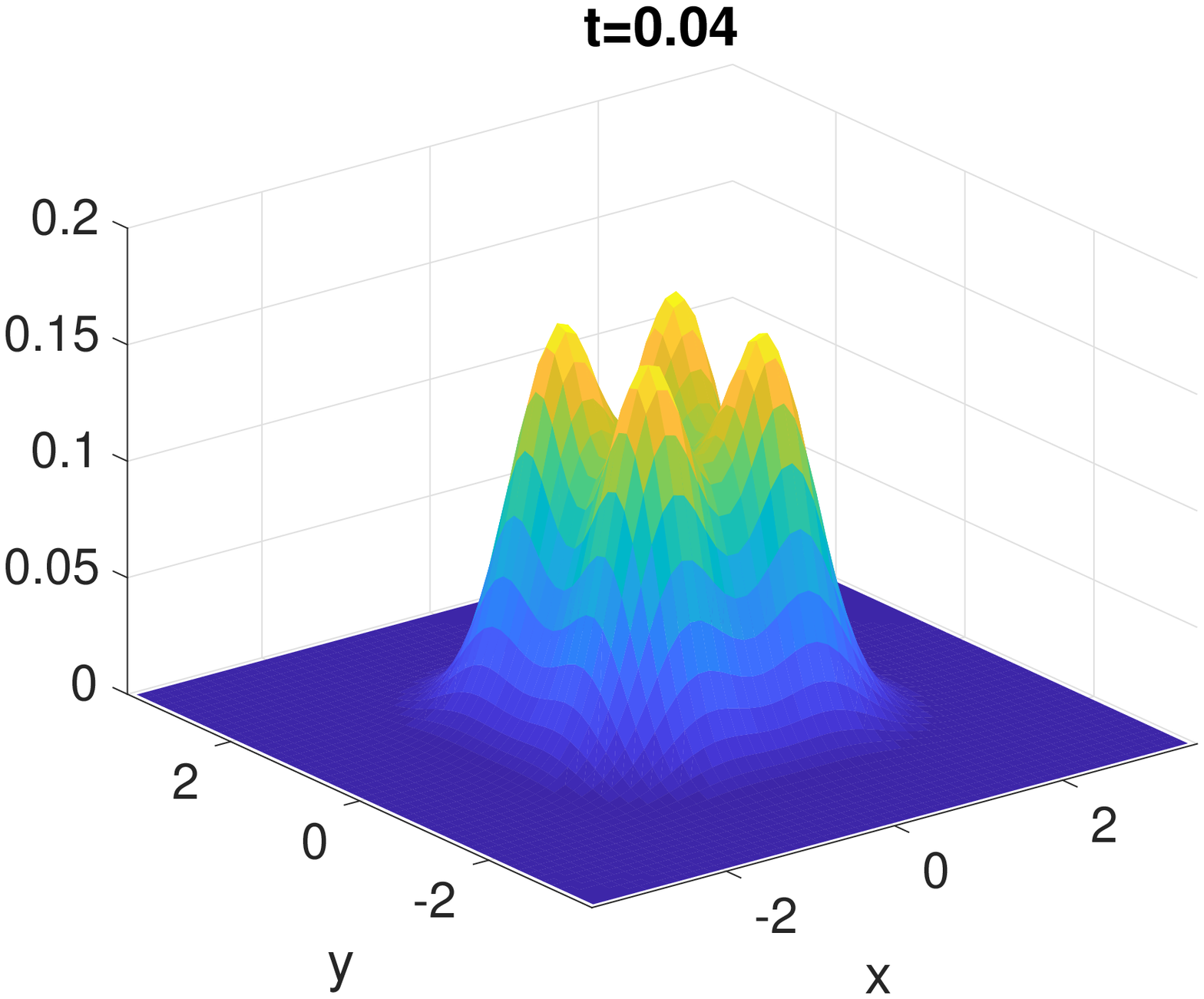}
\includegraphics[width=0.32\textwidth]{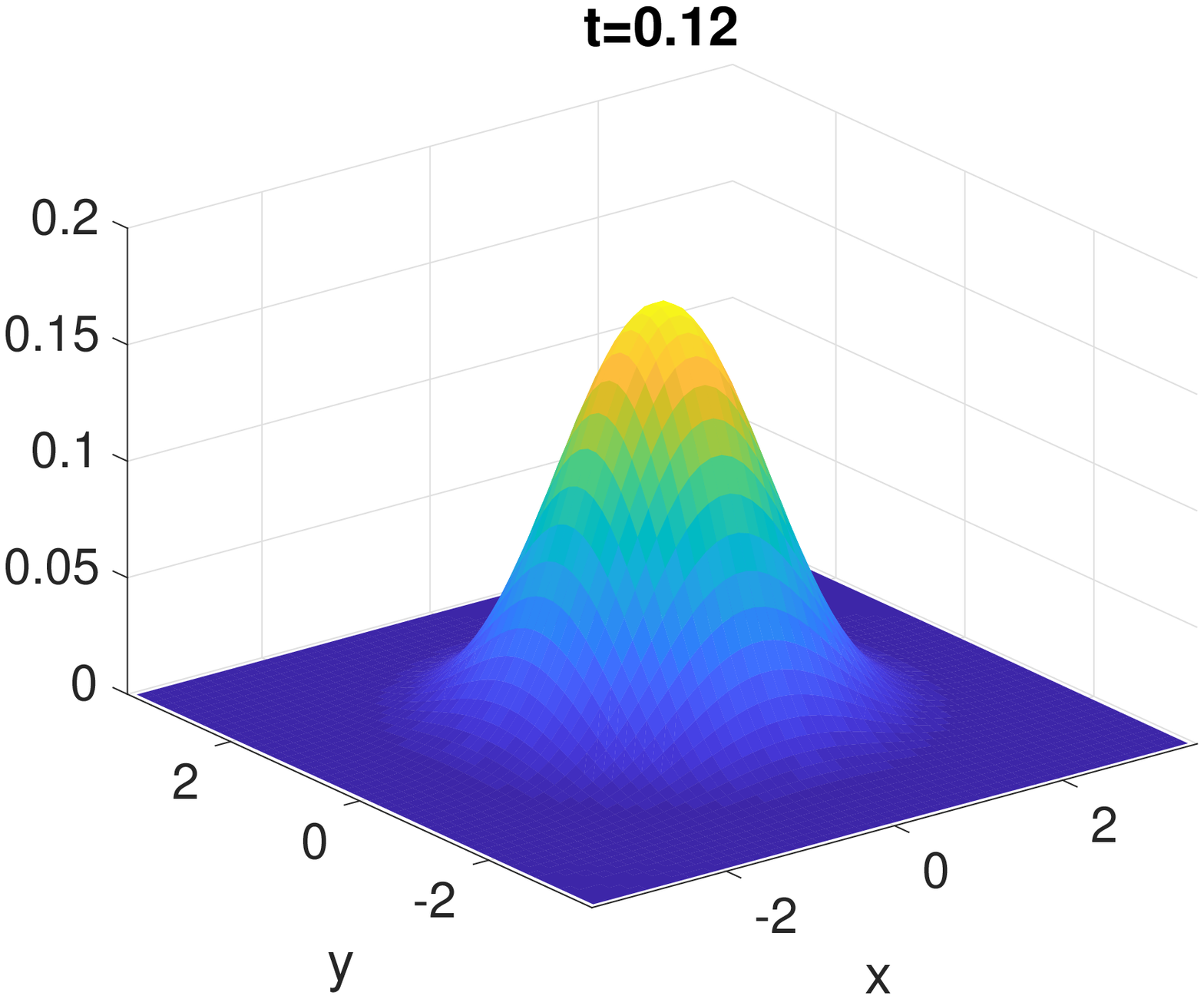}
\\
\includegraphics[width=0.32\textwidth]{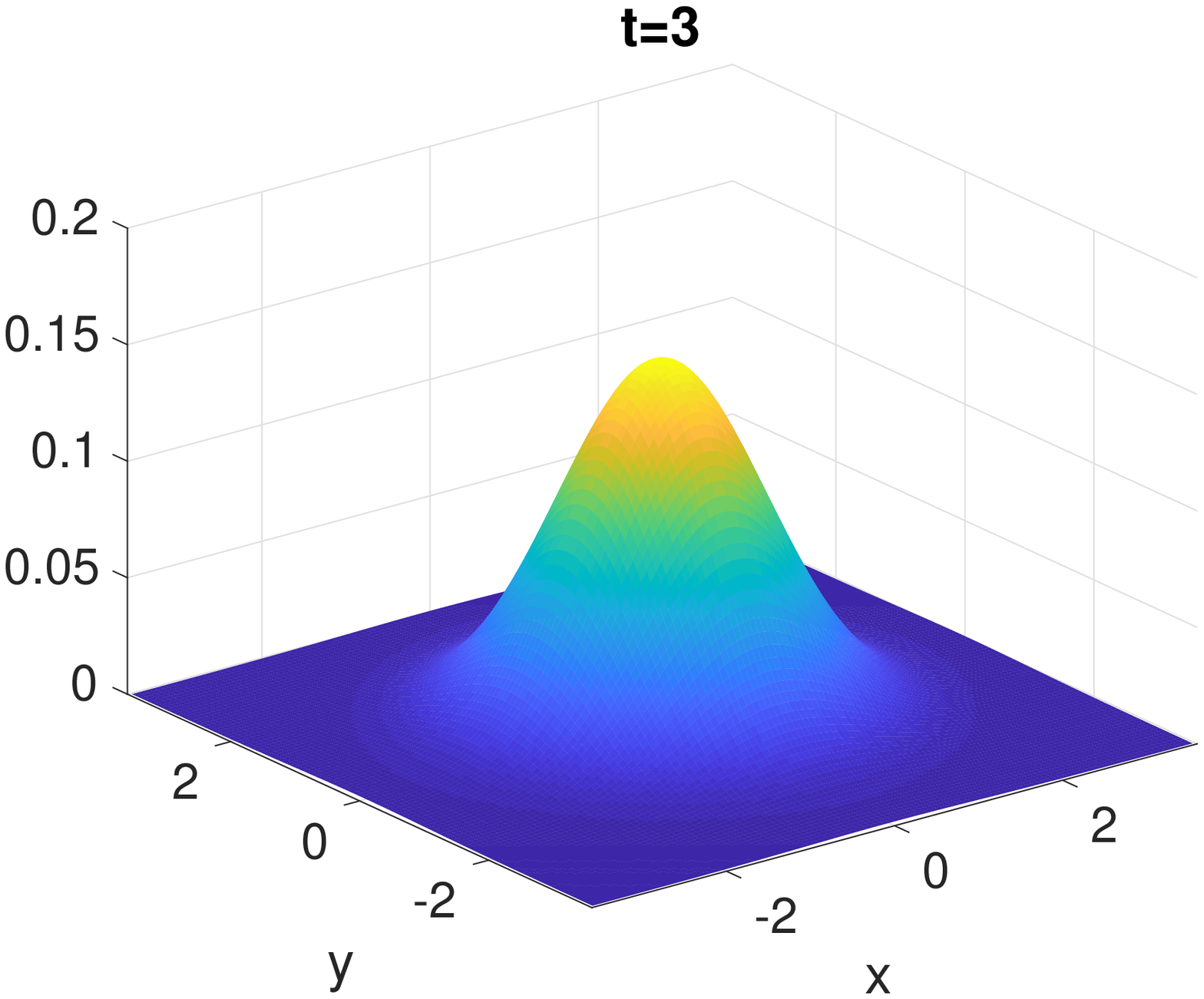}
\includegraphics[width=0.32\textwidth]{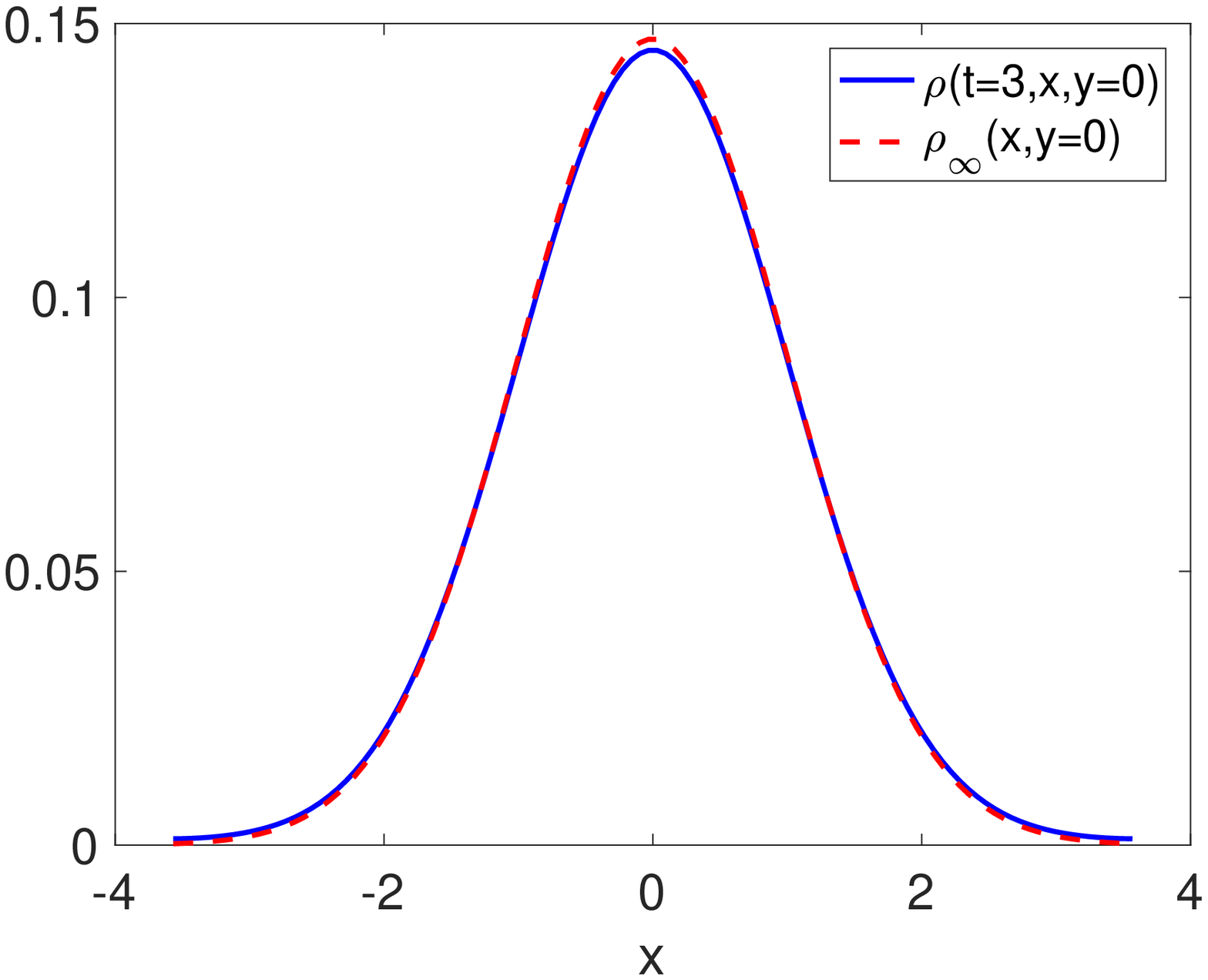}
\includegraphics[width=0.32\textwidth]{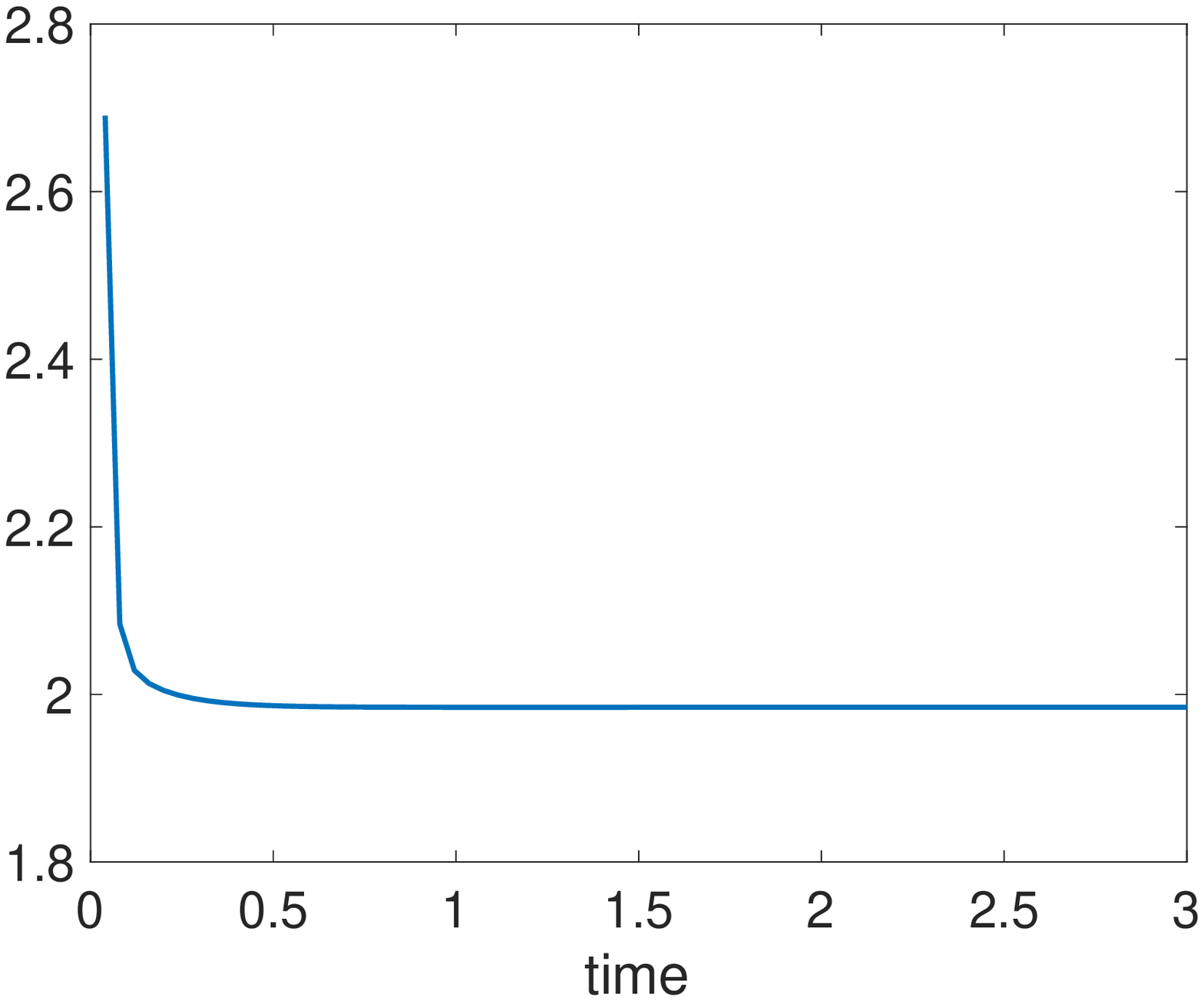}
\caption{The evolution of 2D DLSS equation with $V(x) = \frac{|x|^2}{2}$. Our computational domain is $(x,y) \in [-3.6,3.6]^2$ and mesh sizes are $\tau = 0.04$, $\Delta x = \Delta y = 0.0643$. }
\label{fig:dlss_2d}
\end{figure}

\section{Discussion}\label{section5}
In this paper, we propose a variational time discretization scheme for Wasserstein gradient flows. The scheme applies the quadric approximation of Wasserstein-2 metric and introduces the Fisher information regularization into the iterative regularization. 
In discrete grids, this regularized term helps the gradient flow path to maintain positivity during the evolution and further improves the convexity of the variational problem.

\medskip 
\noindent{\bf Acknowledgement:} WL was partially supported by AFOSR
MURI FA9550-18-1-0502. JL was partially supported by NSF under grant
DMS-1454939. LW was partially supported by NSF grant DMS-1903420 and
NSF CAREER grant DMS-1846854. The authors are grateful to the support
from KI-Net (NSF grant RNMS-1107444) and UMN-Math Visitors Program to
facilitate the collaboration.

\end{document}